\documentclass[11pt,hyp,]{nyjm}
\usepackage{hyperref}
\hypersetup{nesting=true, debug=true,naturalnames=true}

\let\<\langle
\let\>\rangle

\let\uml\"

\usepackage{graphicx, upref}
\allowdisplaybreaks

\usepackage[top=2cm, bottom=2cm, left=3cm, right=3cm] {geometry}

\AtBeginDocument{\hypersetup{citecolor=blue, urlcolor=magenta, linkcolor=magenta}}

\newcommand\rurl[1]{%
  \href{http://#1}{\nolinkurl{#1}}%
}

\usepackage{amsthm, amssymb, amsfonts, mathrsfs, amsmath}
\usepackage{lmodern}
\usepackage{hyperref}
\usepackage[T5]{fontenc}
\usepackage{amscd}

\theoremstyle{definition}
\newtheorem{defn}{Definition}[section]
\newtheorem{rema}[defn]{Remark}

\newtheorem*{acknow}{Acknowledgments}

\theoremstyle{plain}
\newtheorem{thm}[defn]{Theorem}

\newtheorem{conj}[defn]{Conjecture}
\newtheorem{corl}[defn]{Corollary}

\newtheorem{theo}{Theorem}[subsection]
\newtheorem{lema}[theo]{Lemma}
\newtheorem{hq}[theo]{Corollary}
\newtheorem{md}[theo]{Proposition}

\theoremstyle{definition}
\newtheorem{rem}[theo]{Remark}

\title[Structure of the coinvariants space $\mathbb Z_2\otimes_{GL_4(\mathbb Z_2)} PH_*(\mathbb Z_2^4, \mathbb Z_2)$ and its application]{Structure of the space of $GL_4(\mathbb Z_2)$-coinvariants $\mathbb Z_2\otimes_{GL_4(\mathbb Z_2)} PH_*(\mathbb Z_2^4, \mathbb Z_2)$ in some generic degrees\\ and its application} 

\author{\DJ\d \abreve ng V\~o Ph\'uc}  
\address{Faculty of Education Studies, University of Khanh Hoa, Viet nam} 
\email{dangphuc150488@gmail.com}  
\thanks{\textit{{\bf Dedicated to Prof. Daciberg L. Goncalves for his seventies birthday and to Prof. Bill Singer for his 76th birthday.}}} 



\begin{document} 

\begin{abstract}
 Let $A$ denote the Steenrod algebra at the prime 2 and let $k = \mathbb Z_2.$ An open problem of homotopy theory is to determine a minimal set of $A$-generators for the polynomial ring $P_q = k[x_1, \ldots, x_q] = H^{*}(k^{q}, k)$ on $q$ generators $x_1, \ldots, x_q$ with $|x_i|= 1.$ Equivalently, one can write down explicitly a basis for the graded vector space $Q^{\otimes q} := k\otimes_{A} P_q$ in each non-negative degree $n.$ This is the content of the classical "hit problem" in literature \cite{F.P}. Based on this problem, we are interested in the $q$-th cohomological transfer $Tr_q^{A}$ of William M. Singer \cite{W.S1}, which is one of the useful tools for describing mod-2 cohomology of the algebra $A.$ This transfer is a linear map from the space of $GL_q(k)$-coinvariant $k\otimes _{GL_q(k)} P((P_q)_n^{*})$ of $Q^{\otimes q}$ to the $k$-cohomology group of the Steenrod algebra, ${\rm Ext}_{A}^{q, q+n}(k, k).$ Here $GL_q(k)$ is the general linear group of degree $q$ over the field $k,$ and $P((P_q)_n^{*})$ is the primitive part of $(P_q)^{*}_n$ under the action of $A.$ Singer conjectured that $Tr_q^{A}$ is a monomorphism, but this remains unanswered for all $q\geq 4.$ The present paper is to devoted to the investigation of this conjecture for the rank 4 case. More specifically,  basing the techniques of the hit problem of four variables, we explicitly determine the structure of $k\otimes _{GL_4(k)} P((P_4)_{n}^{*})$ in some generic degrees $n.$ Applying these results and a representation of $Tr_4^{A}$ over the lambda algebra, we notice that Singer's conjecture is true for the rank 4 transfer in those degrees $n$. Also, we give some conjectures on the dimensions of $k\otimes_{GL_q(k)} ((P_4)_n^{*})$ for the remaining degrees $n.$ As a consequence, Singer's conjecture holds for $Tr_4^{A}.$ This study and our previous results \cite{D.P11, D.P12} have been provided a panorama of the behavior of the fourth cohomological transfer.
\end{abstract} 
\maketitle
\tableofcontents


\section{Introduction}

We will adopt the following notations and conventions throughout the paper. We work at the prime 2 and write $k := \mathbb Z_2$ for the field of two elements. Let us denote by $A$ the classical, singly-graded Steenrod algebra over $k,$ which is the ring of stable operations on cohomology with $k$ coefficients generated by the Steenrod squares $Sq^{t}$ for $t\geq 0$ modulo the Adem relations. Let $P_q = k[x_1, \ldots, x_q] = H^{*}(k^{q}, k)$ be the polynomial ring on generators $x_i$ of degree $1,$ which may be considered as the symmetric power algebra on $(k^{q})^{*}$ and as an unstable left $A$-module. Here $k^{q}$ denotes a rank $q$ elementary abelian 2-group, which views as a $q$-dimensional $k$-vector space. The algebra $A$ acts upon $P_q$ by use the Cartan formula $Sq^t(fg) = \sum_{i+j = t}Sq^{i}(f)Sq^{j}(g)$ for arbitrary $f,\, g\in P_q,$ and the rule that $Sq^t(x_j)$ is $x_j$ if $t = 0,$ $x_j^{2}$ if $t = 1,$ and is 0 otherwise.  The algebra $P_q$ is equipped with the usual structure of a right module over the general linear group $GL_q(k)$ by means of substitution of variables as follows: For each $\sigma = (\sigma_{i, j})\in GL_q(k)$ and any $f(x_1, \ldots, x_q):=f\in P_q,$ one defines 
$$(f\sigma)(x_1, \ldots, x_q) = f(x_1\sigma, \ldots, x_q\sigma) =  f(\sum_{1\leq i\leq q}x_i\sigma_{i, 1}, \ldots, \sum_{1\leq i\leq q}x_i\sigma_{i, q}).$$
Consider the tensor product $Q^{\otimes q} := k\otimes_{A} P_q$. We notice that the mapping $k\to End(Q^{\otimes q})$ is a ring homomorphism in which $End(Q^{\otimes q})$ is the ring of endomorphism of $Q^{\otimes q}$ (as abelian group).  So, $Q^{\otimes q}$ has the structure of a $k$-module (or, equivalently, a $k$-vector space). Writing $(P_q)_n = H^{n}(k^{q}, k)$ for the $k$-subspace of $P_q$ generated by the homogeneous polynomials of the non-negative degree $n$ in $P_q.$ Then, we have that $(P_q)_n$ is a right $kGL_q(k)$-module for every $n\geq 0$ and that the graded space $\{(P_q)^{*}_n\}_{n\geq 0},$ where $(P_q)^{*}_n$ denotes the dual of $(P_q)_n,$ is the divided power algebra $\Gamma(a_1^{(1)}, \ldots, a_q^{(1)})$ generated by $a_1^{(1)}, \ldots, a_q^{(1)},$ in which  $a_i^{(1)}$ is the linear dual to $x_i$. Moreover, this $\Gamma(a_1^{(1)}, \ldots, a_q^{(1)})$ is a bicommutative Hopf algebra with the vector space basis $a_1^{(i_1)}\ldots a_q^{(i_q)},\, i_r\geq 0,$ for all $1\leq r\leq q,$ with multiplication $$\prod_{1\leq r\leq q}a_r^{(i_r)}\prod_{1\leq r\leq q}a_r^{(j_r)} = \prod_{1\leq r\leq q}\binom{i_r+j_r}{i_r}\prod_{1\leq r\leq q}a_r^{(i_r+j_r)}.$$ The right $A$-module structure of this algebra is described by the rule $$(a^{(n)}_i)Sq^{t} = \binom{n-t}{t}a^{(n-t)}_i = Sq_*^{t}(a^{(n)}_i)$$ and Cartan's formula for Steenrod squares on cross products. The group $GL_q(k)$ acts on $\Gamma(a_1^{(1)}, \ldots, a_q^{(1)})$ by the formula $$\sigma(\prod_{1\leq r\leq q}a_r^{(i_r)}) = (\sum_{1\leq i\leq q}\sigma_{i, 1}a_i^{(1)})^{(i_1)}(\sum_{1\leq i\leq q}\sigma_{i, 2}a_i^{(1)})^{(i_2)}\ldots (\sum_{1\leq i\leq q}\sigma_{i, q}a_i^{(1)})^{(i_q)},$$ for all $\sigma = (\sigma_{i, j})\in GL_q(k).$ So, $\Gamma(a_1^{(1)}, \ldots, a_q^{(1)})$ is a left  $kGL_q(k)$-module. Let us denote by $Q_n^{\otimes q}:= (k\otimes_{A} P_q)_n$ the $k$-vector subspace of $Q^{\otimes q}$ consisting of all the classes represented by the elements in $(P_q)_n.$ The action of $GL_q(k)$ commutes with that of the Steenrod squares $Sq^{t}$ on $P_q$, (i.e., $Sq^{t}: (P_q)_{n-t}\to (P_q)_{n}$ is a map of $kGL_q(k)$-modules, and  $Sq^{t}(f)\sigma = Sq^{t}(f\sigma)$ for every $f\in (P_q)_{n-t},\, \sigma\in GL_q(k)$) and therefore there exists an action of $GL_q(k)$ on $Q^{\otimes q}_n.$ 

Let us now recall that an interesting problem of homotopy theory is to determine the set of homotopy classes $[\mathbb S^{n+q}, \mathbb S^{n}]$ of continuous based map between spheres. It is known that for $n+q > 0,$ these sets have a natural group structure and they are abelian when $n+q > 1.$ The Freudenthal suspension theorem in \cite{H.F} showed a relationship between the groups $[\mathbb S^{n+q}, \mathbb S^{n}]$ for fixed $q$ and varying $n.$ The supension map induces a sequence:
$$ \cdots \longrightarrow [\mathbb S^{n-1+q}, \mathbb S^{n-1}]\longrightarrow [\mathbb S^{n+q}, \mathbb S^{n}]\longrightarrow [\mathbb S^{n+1+q}, \mathbb S^{n+1}]\longrightarrow \cdots$$ of group homomorphisms, and when $n+q > 1,$ these homomorphisms are isomorphisms. Then, the stable value $[\mathbb S^{n+q}, \mathbb S^{n}]$ for $n$ sufficiently large is known as the $q$-th stable homotopy group of spheres, $\pi_q.$ The cohomology of $A$ with $k$-coefficients, ${\rm Ext}_A(k, k) = \{{\rm Ext}^{q, t}_A(k, k)\}_{(q, t)\in\mathbb Z^{2},\, q\geq 0,\, t\geq 0}$ is an object of much interest in Algebraic topology. It features prominently in homotopy theory as the $E_2$-page of the Adams spectral sequence (the ASS) $E_2^{q, t} = {\rm Ext}^{q, t}_A(k, k)$ for the computation of the groups $\pi_q.$ (It is to be noted that a spectral sequence consists of a sequence of intermediate dual chain complexes called pages $E_0, E_1, E_2, \ldots,$ with differentials denoted by $d_0, d_1, d_2,\ldots,$ such that $E_{t+1}$ is the cohomology of $E_t.$) Although the bigraded algebra ${\rm Ext}_A(k, k)$ has been intensively studied by many authors like Adams \cite{J.A2}, Adem \cite{Adem}, Wall \cite{Wall}, Wang \cite{Wang}, Lin \cite{W.L}, Chen \cite{Chen}), its structure remains largely mysterious. Up to now, it remains open for every cohomological degree $q >5.$ The May spectral sequence \cite{J.M} is the best way to compute ${\rm Ext}_A(k, k)$ by hand. Another tool that is also quite efficient to study the cohomology of $A$ is the Singer cohomological "transfer" \cite{W.S1}, 
which is a linear transformation
$$ Tr_q^{A}: k\otimes _{GL_q(k)} P((P_q)_n^{*}) \to {\rm Ext}_A^{\dim k^{q}, \dim k^{q}+n}(k, k) = {\rm Ext}_A^{q, q+n}(k, k),$$
where $P((P_q)_n^{*}) := \langle \{\theta\in (P_q)_n^{*}:\, (\theta)Sq^{i} = 0,\, \mbox{for all $i > 0$}\} \rangle  = (Q_n^{\otimes q})^{*},$ the space of primitive homology classes as a representation of $GL_q(k)$ for all $n$ and the coinvariant $k\otimes _{GL_q(k)} P((P_q)_n^{*})$ is isomorphic as an $k$-vector space to $(Q_n^{\otimes q})^{GL_q(k)},$ the subspace of $GL_q(k)$-invariants of $Q^{\otimes q}.$ The Singer transfer has been studying for a long time: see Boardman \cite{J.B}, Ch\ohorn n and H\`a \cite{C.Ha}, Crossley \cite{M.C2}, H\`a \cite{Ha}, H\uhorn ng \cite{Hung}, H\uhorn ng-Qu\`ynh \cite{H.Q}, Minami \cite{N.M}, the present writer \cite{D.P4, D.P6, D.P10, D.P7, D.P11, D.P12}, Sum \cite{N.S2, N.S4}, and others. By the works of Singer himself \cite{W.S1} and Boardman \cite{J.B}, $Tr_q^{A}$ is known to be an isomorphism for $q\leq 3.$ Moreover, Singer \cite{W.S1} shows that the "total" transfer $ Tr^{A}:=\{Tr^{A}_q\}_{q\geq 0}: \{k\otimes _{GL_q(k)} P((P_q)_n^{*})\}_{q\geq 0} \longrightarrow \{{\rm Ext}_A^{q, q+n}(k, k)\}_{q\geq 0}$ is a homomorphism of bigraded algebras with respect to the product by concatenation in the domain and the usual Yoneda product for the Ext group. These works demonstrated that the cohomological transfer is highly nontrivial. Furthermore, we knew that this transfer homomorphism is induced over the $E_2$-page of the ASS by the geometrical transfer map $\Sigma^{\infty}(B(k^{q})_{+})\longrightarrow \Sigma^{\infty}(\mathbb S^{0})$ between the suspension spectrum in stable homotopy  category (see also Mitchell \cite{S.M}). Here $\mathbb S^{0} = \{0, 1\}$ with $0$ as base-point. The work of Minami \cite{N.M2} tells us that these transfers are closely related to the problem of finding permanent cycles in the ASS. In the second cohomology groups of $A$, by Mahowald \cite{Mahowald} and Lin-Mahowald \cite{L.M}, the classes $h_1h_j$ for $j\geq 3$ and $h_j^{2}$ for $0\leq j\leq 5,$ are known to be the permanent cycles in the ASS. In 2016, Hill, Hopkins, and Ravenel \cite{Hill} claimed that when $j\geq 7,$ the class $h_j^{2}$ is not a permanent cycle in the ASS. The case $j = 6$ has not been solved and has the status of a hypothesis by Snaith \cite{Snaith}. Most recently, Snaith's conjecture can be established by Akhmet'ev \cite{Akhmet}.  The question of whether these $h_j^{2}$ are the permanent cycles in the ASS or not is called \textit{Kervaire invariant problem} in literature \cite{W.B}.  This is one of the oldest unresolved issues in Differential and Algebraic topology. Return to Singer's transfer, in \cite{W.S1}, Singer sets up the following conjecture, which has not been proven or refuted for $q\geq 4.$

\begin{conj}\label{gtS}
The transfer homomorphism $Tr_q^{A}$ is a monomorphism.
\end{conj}

The present study is to devoted to the investigation of this conjecture for $q = 4.$ Beyond the above methods, the mod two lambda algebra $\Lambda$ of Bousfield et al. \cite{Bousfield} can also be used as an useful tool to describe mysterious Ext groups. One can view $\Lambda$ as the $E_1$-term of the classical Adams spectral sequence converging to the 2-component of the stable homotopy groups of spheres. For the reader's convenience, let us note again that $\Lambda$ is an associative differential bigraded algebra with generators $\lambda_n\in \Lambda^{1, n}$ ($n\geq 0$) and the Adem relations
\begin{equation}\label{pt1}
\begin{array}{ll}
\medskip
\lambda_i\lambda_{2i+n+1} &= \sum_{j\geq 0}\binom{n-j-1}{j}\lambda_{i+n-j}\lambda_{2i+1+j} \\
& \quad \quad (i\geq 0,\ n\geq 0)
\end{array}
\end{equation}
with differential 
\begin{equation}\label{pt2}
\begin{array}{ll}
\medskip
d(\lambda_{n-1}) &= \sum_{j\geq 1}\binom{n-j-1}{j}\lambda_{n-j-1}\lambda_{j-1}\\
 & \quad \quad\quad  (n\geq 1),
\end{array}
\end{equation}
where $d(\lambda_0) = 0.$ We refer to \cite{Wang} for the relations \eqref{pt1} and \cite{Bousfield, Priddy} for that the differential in \eqref{pt2} is a well-defined endomorphism of $\Lambda.$ According to \cite{Wang},  there is a unique differential algebra endomorphism $\theta: \Lambda\to \Lambda$ with $\theta(\lambda_{n}) = \lambda_{2n+1}.$ This map, which is one-to-one, induces the classical $Sq^{0}$, an endomorphism of ${\rm Ext}_A(k, k)$. Indeed, Palmieri \cite{Palmieri} explained that as follows: It is well-known that the graded dual $A^{*}$ of $A$ is a commutative Hopf algebra over $k$, and as an algebra over $k,$ $A^{*}\cong k[\xi_1, \xi_2, \ldots],$ the polynomial algebra on generators $\xi_i$ of degree $2^{i}-1.$ Thence, the Frobenius map of $A^{*}$ ($A^{*}\longrightarrow A^{*},\, \xi\longmapsto \xi^{2}$) is a Hopf algebra map. Due to Priddy \cite{Priddy}, $\lambda_n$ is represented in the cobar construction by $[\xi_1^{n+1}].$ Due to May \cite{J.M2}, $Sq^{0}$ is induced by the dual to the Frobenius map and therefore, $Sq^{0}$ takes $[\xi_1^{n+1}]$ to $[\xi_1^{2(n+1)}].$ Thus, it agrees with endomorphism induced by $\theta.$  Alternatively, people have shown that $Sq^{0}$ commutes with the so-called Kameko $Sq^0$ through the cohomological transfer $Tr_q^{A}$ (see also \cite{J.B}, \cite{N.M} for discussions therein). Now for non-negative integers $j_1, \ldots, j_q$, a monomial  $\prod_{1\leq s\leq q}\lambda_{j_s}$ in $\Lambda$ is said to be \textit{an admissible of length q} if $j_s\leq 2j_{s+1}$ for all $1\leq s \leq q-1.$ The admissible monomials form an additive basis of $\Lambda$ by \cite{Bousfield, Priddy}. We denote by $\Lambda^{q, n}$ the $k$-vector subspace of $\Lambda$ generated by all the admissible monomials of length $q.$ In \cite{C.Ha}, Ch\ohorn n and H\`a gave an interesting linear  transformation $\psi_q: (P_q)_n^{*}\longrightarrow \Lambda^{q, n},$ which is determined by $\psi_q(\prod_{1\leq s\leq q}a_{s}^{(j_{s})}) =  \lambda_{j_q}$ if $q = 1,$ while $\psi_q(\prod_{1\leq s\leq q}a_{s}^{(j_{s})}) = \sum_{k\geq j_q}\psi_{q-1}(Sq_*^{k-j_q}(\prod_{1\leq s\leq q-1}a_{s}^{(j_{s})}))\lambda_k$ if $q > 1,$ for any $\prod_{1\leq s\leq q}a_{s}^{(j_{s})}$ in $(P_q)_n^{*}.$ This map can be considered as the $E_1$-level of the Singer transfer; further the authors have stated the following, which is a dual version of the one in H\uhorn ng \cite{Hung0}.
\begin{thm}\label{dlCH}
With the notation chosen, if $\zeta\in P((P_q)_n^{*}),$ then $\psi_q(\zeta)$ is a cycle in $\Lambda$ and is a representative of $Tr_q^{A}([\zeta]).$
\end{thm}

Our motivation to write up this work is to describe the structure of the coinvariants $k\otimes _{GL_q(k)} P((P_k)_{n}^{*})$ and also to investigate Conjecture \ref{gtS} in the rank 4 case. One of the major difficulties involved in studying Singer's conjecture is that we do not determine a representation of $GL_q(k)$ (or a basis of $k\otimes_{GL_q(k)}P((P_q)_n^{*})$). So, our current goal is to make some progress towards that problem when $q = 4$. More precisely, by using the techniques in the hit problem of four variables (see the works of Sum \cite{N.S, N.S1}), we explicitly determine the structure of $k\otimes _{GL_4(k)} P((P_4)_{n}^{*})$ in some generic degrees $n.$ Applying these results and the representation in the algebra $\Lambda$ of the fourth transfer, we show that Conjecture \ref{gtS} holds for $q = 4$ and degrees given.  We also give some conjectures on the dimensions of $k\otimes_{GL_q(k)} ((P_4)_n^{*})$ for the remaining degrees $n.$ As a consequence, Singer's conjecture for the rank 4  transfer has been favored. From these results and our previous works \cite{D.P11, D.P12}, we obtain a complete picture of the behavior of $Tr_4^{A}$. Our method is completely different from that of Singer \cite{W.S1}, Boardman \cite{J.B}, Ch\ohorn n-H\`a \cite{C.Ha}, H\`a \cite{Ha}, H\uhorn ng \cite{Hung}, H\uhorn ng-Qu\`ynh \cite{H.Q} in the study of the domain of the  cohomological transfer.

\section{Outline of main results}

To motivate the statement of our main result, we review the hit problem of Frank Peterson \cite{F.P} and some known results on the graded space ${\rm Ext}_{A}^{q,*}(k, k)$ for $q\leq 4.$ Firstly, we say that a polynomial $f\in (P_q)_n$ is "hit" (or $A$-decomposable), if it satisfies a \textit{hit equation}, $f = \sum_{t > 0}Sq^{t}(f_t)$ for some $Sq^{t}\in A$ and $f_t\in (P_q)_{n-t}.$ Said differently, $f$ belongs to $kGL_q(k)$-submodule, $\widehat{A}(P_q)_n:= (P_q)_n\cap \widehat{A}P_q = (P_q)_n\cap \sum_{t > 0}{\rm Im}(Sq^{t}),$ where $\widehat{A}$ is the kernel of the epimorphism of graded $k$-algebras $A\to k$ given by $Sq^{0}\longmapsto 1$ and $Sq^{i}\longmapsto 0$ for all $i > 0.$  For instance, $x^{2}\in (P_1)_2$ is hit since $x^{2} = Sq^{1}(x).$  In general, it is not easy to determine whether a polynomial is hit, except for special cases given in Theorem \ref{dlSi} (see section four). Solving the hit problem is to determine a minimal set of $A$-generators for the polynomial ring $P_q$ in each degree $n.$ Equivalently, when $k$ has the trivial $A$ action, one can write down explicitly a minimal generating set for the "cohit" $k$-module, $${\rm Tor}^{A}_{0, n}(k, P_q) = Q_n^{\otimes q}\cong (P_q)_n/\widehat{A}(P_q)_n,\ q \geq 1,\ n\geq 0.$$ The structure of this space, which is closely related to some classical problems in homotopy theory (e.g., cobordism theory of manifolds, modular representation theory of the general linear group, stable homotopy type of classifying spaces of finite groups),  has explicitly been determined for $q\leq 4$ and all degrees $n$: see the works of Peterson \cite{F.P}, Kameko \cite{M.K} and Sum \cite{N.S, N.S1}. Due to Minami \cite{N.M}, the hit problem also plays an important role for studying permanent cycles in the ASS. Recently, the problem for $q\geq 5$ in some certain degrees is known by Sum \cite{N.S3, N.S4} and the present author \cite{D.P4, D.P5, D.P6, D.P10, D.P7, D.P12}. By these, it seems likely that an explicit description of $Q_n^{\otimes q}$ for general $q$ will appear in the near future. Hit problems for the symmetric polynomials have been investigated by the works of Janfada and Wood (see \cite{Janfada1, Janfada2, Janfada3, Janfada4}). The following Kameko maps \cite{M.K} are often used in studying the hit problem: For each $n\geq 0,$ the \textit{down Kameko map} $\overline {Sq}^0: (P_q)_{2n+q}\to (P_q)_n$ is a surjective linear map defined on monomials by $\overline {Sq}^0(f) = g$ if $f = \prod_{1\leq i\leq q}x_ig^{2}$ and $\overline {Sq}^0(f) = 0$ otherwise. This map sends hit polynomials to hit polynomials. The \textit{up Kameko map} $\varphi: (P_q)_{n}\to (P_q)_{2n+q}$ is an injective linear map defined on monomials by $\varphi(g) = \prod_{1\leq i\leq q}x_ig^{2}.$ In general, this map does not send hit polynomials to hit polynomials. However, let $\mu(n) = \mbox{min}\big\{*\in \mathbb N:\ \alpha(n + *)\leq *\big\},$ where the $\alpha$ function counts the number of ones in the binary expansion of its argument. Then, if $\mu(2n+q) = q$, then $\overline {Sq}^0$ induces an isomorphism of $kGL_q(k)$-modules $\overline {Sq}^0: Q^{\otimes q}_{2n+q}\to Q^{\otimes q}_{n}$ with inverse $\varphi.$ In \cite{F.P},  Peterson conjectured that $Q^{\otimes q}_n$ is trivial if and only if $\mu(n) > q.$ (His motivation for this was to prove that if $\mathcal M$ is a smooth manifold of dimension $n$ such that all products of length $q$ of Stiefel-Whiney classes of its nomal bundle vanish, then either $\alpha(n)\leq q$ or $\mathcal M$ is cobordant to zero.)  A number of researchers have investigated Peterson's conjecture, but at last it was solved by Wood \cite{R.W}. Additionally, in \cite{W.S2}, Singer proved another result which generalizes the Peterson conjecture and identifies a new class of monomials in $\widehat{A}(P_q)_n.$ So, from these facts, it is sufficient to solve the hit problem in degrees $n$ satisfying $\mu(n) < q.$ By this condition, a careful but straightforward computation shows that $n$ is of the "generic" form:
\begin{equation}\label{ct}
n = r(2^{s}-1) + v.2^{s},
\end{equation}
where $0\leq \mu(v) < r < q$ and $s\geq 0$. Thus, we only need to study the behavior of the algebraic transfer at these degrees. Next, we provide a brief information on the graded space ${\rm Ext}_{A}^{q, *}(k, k)$ for $q\leq 4.$ 

\begin{thm}[see Adams \cite{J.A2},  Adem \cite{Adem}, Wall \cite{Wall}, Wang \cite{Wang}, Lin \cite{W.L}]\label{dlntg}
The following hold:

\begin{enumerate}

\item[i)] ${\rm Ext}_{A}^{1, *}(k, k)$ is generated by  $h_i$ for $i\geq 0;$

\item[ii)] ${\rm Ext}_{A}^{2, *}(k, k)$ is generated by $h_ih_j$ for $j\geq i\geq 0$ and $j\neq i+1;$

\item[iii)] ${\rm Ext}_{A}^{3, *}(k, k)$ is generated by $h_ih_jh_{\ell},\, c_t$ for $t\geq 0;$ $\ell\geq j\geq i\geq 0,$ and subject only to the relations $h_ih_{i+1} = 0,\ h_ih_{i+2}^{2} = 0$ and $h_i^3 = h^2_{i-1}h_{i+1};$ 
\item[iv)] ${\rm Ext}_{A}^{4, *}(k, k)$ is generated by $h_ih_jh_{\ell}h_m,\, h_uc_v,\, d_t,\, e_t,\, f_t,\, g_{t+1},\, p_t,\, D_3(t),\, p'_t$ for $m\geq \ell\geq j\geq i\geq 0$, $u,\ v,\ t\geq 0$, and subject to the relations in iii) together with $h^2_ih^2_{i+3} = 0,\, h_{v-1}c_v = 0,\, h_{v}c_v = 0,\, h_{v+2}c_v = 0$ and $h_{v+3}c_v = 0.$
\end{enumerate}
\end{thm}

We are now in a position to formulate the main results of the paper. For this let us study Conjecture \ref{gtS} in the rank $4$ case and the generic degrees of the form \eqref{ct} where 
$$  (r, v) = \left\{\begin{array}{ll}
 (3, 2^{t+1} - 1)&\ \mbox{for $t\neq 3$},\\[1mm]
 (2, 2^{t} - 1)&\ \mbox{for $t = 1$ and $t\geq 5$},\\[1mm]
\end{array}\right.$$
We refer the reader to some remarks in section three for explaining why we need to calculate these degrees. We first consider the pair $(r, v) = (3, 2^{t+1} - 1)$ for $t\neq 3.$  Thence, one has the following cases.

{\bf The case \boldmath{$t = 1$}.} Based upon an admissible basis of the $k$-vector space $Q^{\otimes 4}_{3(2^{s}-1) + 3.2^{s}}$ in \cite{N.S1}, we find that 

\begin{thm}\label{dlc1}
For a positive integer $s$, then
$$ \dim k\otimes_{GL_4(k)}P((P_4)_{3(2^{s}-1) + 3.2^{s}}^{*}) = \left\{\begin{array}{ll}
0 &\mbox{if $s = 2$},\\
1 &\mbox{if $s \neq 2$}.
\end{array}\right.$$
\end{thm}

Let us sketch the proof of the theorem. Firstly, by Sum \cite{N.S1}, the dimensions of the $k$-vector spaces $Q^{\otimes 4}_{3(2^{s} -1) + 3.2^{s}}$ are computed as follows:
$$ \dim Q^{\otimes 4}_{3(2^{s} -1) + 3.2^{s}} = \left\{ \begin{array}{ll}
46 &\mbox{if $s = 1$},\\
94 &\mbox{if $s = 2$},\\
105 &\mbox{if $s \geq 3$}.
\end{array}\right.$$
Moreover, they have an admissible monomial basis, which are given in \cite{N.S}. Now, thanks to these results, for $s \in \{1, 2\},$ a direct computation indicates that 
$$ k\otimes_{GL_4(k)}P((P_4)_{3(2^{s}-1) + 3.2^{s}}^{*}) = \left\{\begin{array}{ll}
\langle [\zeta_1] \rangle &\mbox{if $s = 1$},\\
0 &\mbox{if $s = 2$},\\
\end{array}\right.
$$
where $\zeta_1 = a_1^{(1)}a_2^{(3)}a_3^{(3)}a_4^{(2)} + a_1^{(1)}a_2^{(3)}a_3^{(4)}a_4^{(1)} + a_1^{(1)}a_2^{(5)}a_3^{(2)}a_4^{(1)} +  a_1^{(1)}a_2^{(6)}a_3^{(1)}a_4^{(1)}$ belongs to $P((P_4)_{3(2^{1}-1) + 3.2^{1}}^{*}).$ By the unstable condition, to verify that $\zeta_1$ is $\widehat{A}$-annihilated, we only need to consider the effects of $Sq^{1}$ and $Sq^{2}$. 

For $s\geq 3,$ based on the basis of $Q^{\otimes 4}_{3(2^{s} -1) + 3.2^{s}},$ we obtain
\begin{equation}\label{pt3}
 \dim k\otimes_{GL_4(k)}P((P_4)_{3(2^{s}-1) + 3.2^{s}}^{*})\leq 1.
\end{equation}
On the other side, it is easy to check that the elements $ \zeta_s = a_1^{(0)}a_2^{(2^{s+1}-1)}a_3^{(2^{s+1}-1)}a_4^{(2^{s+1}-1)}$ in $(P_4)^{*}_{3(2^{s} -1) + 3.2^{s}}$ are $\widehat{A}$-annihilated. Moreover, notice that $\lambda_{2^{s}-1}\in \Lambda^{1, 2^{s}-1}$ are a cycle in the algebra $\Lambda$ and $h_s = [\lambda_{2^{s}-1}]\in {\rm Ext}_A^{1, 2^{s}}(k, k).$ Then, since $\zeta_s\in P((P_4)_{3(2^{s}-1) + 3.2^{s}}^{*})$,  by Theorem \ref{dlCH}, deduce that the cycles $\psi_4(\zeta_s) = \lambda_0\lambda_{2^{s+1}-1}^{3}$ in $\Lambda$ are representative of the non-zero elements $h_0h_{s+1}^{3}\in {\rm Ext}_A^{4, 6.2^{s}  +1}(k, k).$ These lead to $h_0h_{s+1}^{3}\in {\rm Im}(Tr_4^{A}).$ Invoking Theorem \ref{dlntg}, it follows that
 \begin{equation}\label{pt4}
{\rm Ext}_A^{4, 6.2^{s}  +1}(k, k)\\
=  \left\{\begin{array}{ll}
\langle h_0h_{2}^{3}, h_1c_0 \rangle = \langle h_1c_0\rangle &\mbox{if $s = 1$},\\
\langle h_0h_{s+1}^{3}\rangle  &\mbox{if $s\geq 2$},
\end{array}\right.
\end{equation}
with $h_0h_{s+1}^{3} = 0$ for $s = 2$ and $h_0h_{s+1}^{3} = h_0h_s^{2}h_{s+2}\neq 0$ for $s\geq 3.$ These data and the inequality \eqref{pt3} imply that the coinvariants space $k\otimes_{GL_4(k)}P((P_4)_{3(2^{s}-1) + 3.2^{s}}^{*})$ is $1$-dimensional. Moreover, by a direct computation using the monomial bases of $Q^{\otimes 4}_{3(2^{s} -1) + 3.2^{s}},$ one can obtain that the coinvariants $k\otimes_{GL_4(k)}P((P_4)_{3(2^{s}-1) + 3.2^{s}}^{*})$ are generated by the classes $[\zeta_s],$ for all $s\geq 3.$

\begin{rema}\label{nxt} Clearly, $\lambda_3^{2}\lambda_2$ is a representative of the non-zero element $c_0\in {\rm Ext}_A^{3, 11}(k, k).$ Then, since $\zeta_1$ is $\widehat{A}$-annihilated, by a direct computation using the representation of $Tr_4^{A}$ over $\Lambda$ and Theorem \ref{dlCH}, we deduce that 
$$\begin{array}{ll}
\medskip
\psi_4(a_1^{(1)}a_2^{(3)}a_3^{(3)}a_4^{(2)}) &= \lambda_1\lambda_3^{2}\lambda_2 + \lambda_1\lambda_3\lambda_4\lambda_1 +  \lambda_1\lambda_4\lambda_3\lambda_1,\\
\medskip
\psi_4(a_1^{(1)}a_2^{(3)}a_3^{(4)}a_4^{(1)}) &=\lambda_1\lambda_3\lambda_4\lambda_1  + \lambda_1\lambda_4\lambda_3\lambda_1+ \lambda_1\lambda_5\lambda_2\lambda_1,\\
\medskip
  \psi_4(a_1^{(1)}a_2^{(5)}a_3^{(2)}a_4^{(1)}) &= \lambda_1\lambda_5\lambda_2\lambda_1 + \lambda_1\lambda_6\lambda_1^{2},\\
\medskip
 \psi_4(a_1^{(1)}a_2^{(6)}a_3^{(1)}a_4^{(1)}) &=\lambda_1\lambda_6\lambda_1^{2},
\end{array}$$
and therefore the cycle $\psi_4(\zeta_1) = \lambda_1\lambda_3^{2}\lambda_2$ in $\Lambda^{4, 9}$ is a representative of the element $h_1c_0\in {\rm Ext}_A^{4, 13}(k, k).$ This fact and the equality \eqref{pt4} lead to the non-zero element  $h_1c_0$ being in the image of $Tr_4^{A}.$
\end{rema}

Combining Remark \ref{nxt} with Theorem \ref{dlc1} and the equality \eqref{pt4}, it may be concluded that

\begin{corl}\label{hqc1}
The Singer transfer $$ Tr_4^{A}: k\otimes_{GL_4(k)}P((P_4)_{6.2^{s}-3}^{*}) \to {\rm Ext}_A^{4, 6.2^{s}  +1}(k, k)$$ is an isomorphism for every positive integer $s.$
\end{corl}

{\bf The case \boldmath{$t = 2$}.} The structure of the coinvariant spaces $k\otimes_{GL_4(k)}P((P_4)_{3(2^{s}-1) + 7.2^{s}}^{*})$ are given as follows. 

\begin{thm}\label{dlc2}
With a positive integer $s,$ we have 
$$ k\otimes_{GL_4(k)}P((P_4)_{3(2^{s}-1) + 7.2^{s}}^{*}) = \left\{\begin{array}{ll}
\langle [\zeta] \rangle &\mbox{if $s = 1$},\\
0 &\mbox{if $s > 1$},\\
\end{array}\right.$$
where $\zeta$ is the following sum:
$$\begin{array}{ll}
&a_1^{(5)}a_2^{(5)}a_3^{(5)}a_4^{(2)}+
 a_1^{(5)}a_2^{(5)}a_3^{(6)}a_4^{(1)}+
 a_1^{(3)}a_2^{(5)}a_3^{(8)}a_4^{(1)}+
a_1^{(5)}a_2^{(3)}a_3^{(8)}a_4^{(1)} +
\medskip
 a_1^{(3)}a_2^{(6)}a_3^{(7)}a_4^{(1)}\\
&+ a_1^{(5)}a_2^{(7)}a_3^{(4)}a_4^{(1)}+
a_1^{(7)}a_2^{(5)}a_3^{(4)}a_4^{(1)}+
 a_1^{(3)}a_2^{(9)}a_3^{(4)}a_4^{(1)}+
 a_1^{(9)}a_2^{(3)}a_3^{(4)}a_4^{(1)}+
\medskip
a_1^{(3)}a_2^{(9)}a_3^{(3)}a_4^{(2)}\\
&+
a_1^{(9)}a_2^{(3)}a_3^{(3)}a_4^{(2)}+
a_1^{(5)}a_2^{(9)}a_3^{(2)}a_4^{(1)}+
a_1^{(9)}a_2^{(5)}a_3^{(2)}a_4^{(1)} +
a_1^{(5)}a_2^{(10)}a_3^{(1)}a_4^{(1)}+
\medskip
a_1^{(9)}a_2^{(6)}a_3^{(1)}a_4^{(1)}\\
&+
 a_1^{(3)}a_2^{(11)}a_3^{(2)}a_4^{(1)}+
 a_1^{(11)}a_2^{(3)}a_3^{(2)}a_4^{(1)} +
 a_1^{(5)}a_2^{(5)}a_3^{(3)}a_4^{(4)}+
 a_1^{(5)}a_2^{(3)}a_3^{(5)}a_4^{(4)}+
\medskip
 a_1^{(3)}a_2^{(5)}a_3^{(5)}a_4^{(4)}\\
&+
 a_1^{(3)}a_2^{(12)}a_3^{(1)}a_4^{(1)}+
  a_1^{(11)}a_2^{(4)}a_3^{(1)}a_4^{(1)}+
 a_1^{(7)}a_2^{(8)}a_3^{(1)}a_4^{(1)}+
 a_1^{(7)}a_2^{(7)}a_3^{(1)}a_4^{(2)}+
\medskip
a_1^{(13)}a_2^{(2)}a_3^{(1)}a_4^{(1)}\\
&+
 a_1^{(14)}a_2^{(1)}a_3^{(1)}a_4^{(1)}+
 a_1^{(6)}a_2^{(5)}a_3^{(3)}a_4^{(3)}+
 a_1^{(5)}a_2^{(3)}a_3^{(6)}a_4^{(3)}+
 a_1^{(3)}a_2^{(6)}a_3^{(5)}a_4^{(3)} +
\medskip
 a_1^{(6)}a_2^{(3)}a_3^{(3)}a_4^{(5)}\\
&+
 a_1^{(3)}a_2^{(3)}a_3^{(6)}a_4^{(5)}+
 a_1^{(3)}a_2^{(6)}a_3^{(3)}a_4^{(5)}+
  a_1^{(5)}a_2^{(3)}a_3^{(3)}a_4^{(6)}+
 a_1^{(3)}a_2^{(5)}a_3^{(3)}a_4^{(6)}+
\medskip
 a_1^{(3)}a_2^{(3)}a_3^{(5)}a_4^{(6)}\\
&+
 a_1^{(3)}a_2^{(3)}a_3^{(3)}a_4^{(8)}+
a_1^{(3)}a_2^{(3)}a_3^{(4)}a_4^{(7)} +
 a_1^{(3)}a_2^{(5)}a_3^{(2)}a_4^{(7)}+
 a_1^{(3)}a_2^{(6)}a_3^{(1)}a_4^{(7)}+
\medskip
 a_1^{(3)}a_2^{(3)}a_3^{(9)}a_4^{(2)}\\
&+
 a_1^{(3)}a_2^{(3)}a_3^{(10)}a_4^{(1)}+
a_1^{(5)}a_2^{(3)}a_3^{(7)}a_4^{(2)}+
 a_1^{(5)}a_2^{(7)}a_3^{(3)}a_4^{(2)}+
 a_1^{(7)}a_2^{(5)}a_3^{(3)}a_4^{(2)}.
\end{array}$$
\end{thm}

To check that $\zeta$ is $\widehat{A}$-annihilated, we need only to compute the actions of $Sq^{i}$ for $i\in \{1, 2, 4\}$ because of the unstable condition. The proof of the theorem is based upon the admissible bases for the $k$-vector spaces $Q^{\otimes 4}_{3(2^{s}-1)+7.2^{s}}$ in \cite{N.S}. 

It is apparently that the non-zero element $e_0$ in ${\rm Ext}_{A}^{4, 21}(k, k)$ is represented by the cycle
$$\overline{e}_0:= \lambda_3^{3}\lambda_8 + \lambda_3\lambda_5^{2}\lambda_4 + \lambda_3^{2}\lambda_7\lambda_4 + \lambda_7\lambda_5\lambda_3\lambda_2 + \lambda_3^{2}\lambda_5\lambda_6$$ in $\Lambda^{4, 17}.$ Then, since $\zeta\in P((P_4)_{3(2^{1}-1) + 7.2^{1}}^{*}),$ by a direct computation using the differential \eqref{pt2} and the representation of $Tr_4^{A}$ over $\Lambda,$ we conclude that
$$
\psi_4(\zeta) =  \overline{e}_0+ d(\lambda_3\lambda_5\lambda_{10} + \lambda_3\lambda_{12}\lambda_3+ \lambda_4\lambda_7^{2} + \lambda_0\lambda_{11}\lambda_7)
$$
 is a cycle in $\Lambda^{4, 3(2^{1}-1) + 7.2^{1}},$ from which one gets
\begin{equation}\label{pt6}
Tr_4^{A}([\zeta]) = [\psi_4(\zeta)] = [\overline{e}_0] = e_0.
\end{equation}
On the other side, using Theorem \ref{dlntg}, we may deduce that
 \begin{equation}\label{pt7}
{\rm Ext}_A^{4, 10.2^{s}  +1}(k, k)\\
 =  \left\{\begin{array}{ll}
\langle h_0h_{2}h_{3}^{2}, e_0 \rangle =  \langle e_0\rangle &\mbox{if $s = 1$},\\
 \langle h_0h_{s+1}h_{s+2}^{2} \rangle = 0 &\mbox{if $s >1$}.
\end{array}\right.
\end{equation}
Taking Theorem \ref{dlc2}, together with the equalities \eqref{pt6} and \eqref{pt7}, the readers can easily see that

\begin{corl}\label{hqc2}
$Tr_4^{A}$ is an isomorphism in the internal degree $10.2^{s}-3$ for any $s\geq 1.$
\end{corl}

{\bf The cases \boldmath{$t\geq 4$}.} By Sum \cite{N.S}, for each $t\geq 4,$ the dimension of the $k$-vector spaces $Q^{\otimes 4}_{3(2^{s}-1) + 2^{s}(2^{t+1}-1)}$ are determined as follows:
$$ \dim Q^{\otimes 4}_{3(2^{s}-1) + 2^{s}(2^{t+1}-1)} =  \left\{\begin{array}{ll}
\medskip
150&\mbox{if $s = 1$},\\
\medskip
195&\mbox{if $s = 2$},\\
\medskip
210&\mbox{if $s\geq 3$},
\end{array}\right.$$
Thanks to these results, by direct calculations, we obtain the following.

\begin{thm}\label{dlct}
Let $s$ and $t$ be positive integers such that $t\geq 4.$ Then, 
$$ 
\dim k\otimes_{GL_4(k)}P((P_4)_{3(2^{s}-1) + 2^{s}(2^{t+1}-1)}^{*})  = \left\{\begin{array}{ll}
1&\mbox{if $s = 1, 2$},\\
2&\mbox{if $s \geq 3$}.
\end{array}\right.$$
Moreover, 
$$ \begin{array}{ll}
\medskip
&k\otimes_{GL_4(k)}P((P_4)_{3(2^{s}-1) + 2^{s}(2^{t+1}-1)}^{*})\\
&  = \left\{\begin{array}{ll}
\langle [a_1^{(0)}a_2^{(2^{s+1}-1)}a_3^{(2^{s+t}-1)}a_4^{(2^{s+t}-1)}] \rangle&\mbox{if $s = 1, 2$},\\[1mm]
\langle [a_1^{(0)}a_2^{(2^{s+1}-1)}a_3^{(2^{s+t}-1)}a_4^{(2^{s+t}-1)}] , [a_1^{(0)}a_2^{(2^{s}-1)}a_3^{(2^{s}-1)}a_4^{(2^{s+t+1}-1)}] \rangle&\mbox{if $s \geq 3$}.
\end{array}\right.
\end{array}$$
\end{thm}

Following Theorem \ref{dlntg}, one gets $${\rm Ext}_A^{4, 2^{s+t+1} + 2^{s+1}+1}(k, k) = \langle h_0h_{s+1}h_{s+t}^{2},\, h_0h_s^{2}h_{s+t+1}\rangle.$$ It should be noted that if $t = 3,$ then $h_0h_{s+1}h_{s+t}^{2} = 0.$ If $t = 1$ and $s > 2,$ then  $h_0h_{s+1}h_{s+t}^{2} = h_0h_s^{2}h_{s+t+1}$.  When $t = 2$ and $s > 1,$ we have $h_0h_{s+1}h_{s+t}^{2} = 0$ and $h_0h_s^{2}h_{s+t+1} = 0.$ In the case in which $s\in \{1, 2\}$, we have $h_0h_s^{2}h_{s+t+1} = 0.$ On the other hand, we observe that the elements $\lambda_{2^{i}-1}\in \Lambda^{1, 2^{i}-1}$ are representative of the non-zero elements $h_i\in {\rm Ext}_A^{1, 2^{i}}(k, k)$ for $i = 0,\, s,\, s+1,, s+t,\, s+t+1 $ ($t\geq 4$). So, since $a_1^{(0)}a_2^{(2^{s+1}-1)}a_3^{(2^{s+t}-1)}a_4^{(2^{s+t}-1)}$ and $a_1^{(0)}a_2^{(2^{s}-1)}a_3^{(2^{s}-1)}a_4^{(2^{s+t+1}-1)}$ are $\widehat{A}$-annihilated, by Theorem \ref{dlCH}, we claim that the cycles 
$$ \begin{array}{ll} 
\medskip
\lambda_0\lambda_{2^{s+1}-1}\lambda_{2^{s+t}-1}^2 &=\psi_4(a_1^{(0)}a_2^{(2^{s+1}-1)}a_3^{(2^{s+t}-1)}a_4^{(2^{s+t}-1)}),\\
\lambda_0\lambda^{2}_{2^{s}-1}\lambda_{2^{s+t+1}-1} &=\psi_4(a_1^{(0)}a_2^{(2^{s}-1)}a_3^{(2^{s}-1)}a_4^{(2^{s+t+1}-1)})
\end{array}$$ in $\Lambda^{4, 3(2^{s}-1) + 2^{s}(2^{t+1}-1)}$ are representative of the non-zero elements $h_0h_{s+1}h_{s+t}^{2}$ and $h_0h_{s}^{2}h_{s+t+1},$ respectively and so, the following is a direct consequence from these data and Theorem \ref{dlct}.

\begin{corl}\label{hqct}
The fourth transfer is an isomorphism when acting on the coinvariant spaces $k\otimes_{GL_4(k)}P((P_4)_{2^{s+t+1} + 2^{s+1} - 3}^{*})$ for all $s > 0$ and $t > 3.$
\end{corl}

Next, consider the pair $(r, v) = (2, 2^{t} - 1)$ with $t = 1,$ we remark that 
$$ Q^{\otimes 4}_{2(2^{s} - 1) + 2^{s}}  \cong {\rm Ker}[\overline{Sq}^{0}]_{2(2^{s} - 1) + 2^{s}} \bigoplus Q^{\otimes 4}_{2^{s-1} + 2^{s} - 3},$$
because the Kameko homomorphism
$$[\overline{Sq}^{0}]_{2(2^{s} - 1) + 2^{s}}:= \overline{Sq}^{0}: Q^{\otimes 4}_{2(2^{s} - 1) + 2^{s}} \to Q^{\otimes 4}_{2^{s-1} + 2^{s} - 3}$$
is an epimorphism of $kGL_4(k)$-modules, from which we get 
\begin{equation}\label{pt8}
 \begin{array}{ll}
\dim k\otimes_{GL_4(k)}P((P_4)_{2(2^{s} - 1) + 2^{s}}^{*}) &\leq \dim ({\rm Ker}[\overline{Sq}^{0}]_{2(2^{s} - 1) + 2^{s}})^{GL_4(k)}\\
&\qquad  + \dim k\otimes_{GL_4(k)}P((P_4)_{2^{s-1} + 2^{s}-3}^{*}).
\end{array}
\end{equation}

The following assertion is based on a monomial basis of ${\rm Ker}[\overline{Sq}^{0}]_{2(2^{s} - 1) + 2^{s}}$ in \cite{N.S}.

\begin{thm}\label{dlc3}
The subspaces of $GL_4(k)$-invariants $({\rm Ker}[\overline{Sq}^{0}]_{2(2^{s} - 1) + 2^{s}})^{GL_4(k)}$ are trivial for all $s > 0.$
\end{thm}

This theorem together with the inequality \eqref{pt8} give the following results:\\[1mm]
For $s = 1,$ evidently $Q^{\otimes 4}_{2^{1-1} + 2^{1}-3} \cong k,$ and so 
\begin{equation}\label{ptp}
k\otimes_{GL_4(k)}P((P_4)_{2^{1-1} + 2^{1}-3}^{*}) = \langle [1]\rangle.
\end{equation}
Suppose that $[f]\in k\otimes_{GL_4(k)}P((P_4)_{2(2^{1} - 1) + 2^{1}}^{*}).$ Then, $[f]$ is dual to $[\widetilde{f}]\in (Q^{\otimes 4}_{2(2^{1} - 1) + 2^{1}})^{GL_4(k)}.$ Since the Kameko map $[\overline{Sq}^{0}]_{2(2^{1} - 1) + 2^{1}}$ is an epimorphism, the dual of $[\overline{Sq}^{0}]_{2(2^{1} - 1) + 2^{1}}([\widetilde{f}])$ belongs to $k\otimes_{GL_4(k)}P((P_4)_{2(2^{1} - 1) + 2^{1}}^{*}).$ So, due to Theorem \ref{dlc3} and the equality \eqref{ptp}, $[f]$ is dual to $(\gamma [\varphi(1)] + [\widetilde{f}']),$
where $\gamma\in k,$ $[\widetilde{f}']\in {\rm Ker}[\overline{Sq}^{0}]_{2(2^{1} - 1) + 2^{1}}$ and the up Kameko map $\varphi$ is determined by
$$ \begin{array}{ll}
\varphi: k &\longrightarrow (P_4)_{2(2^{1} - 1) + 2^{1}}\\
\ \ \ \ u&\longmapsto\left\{\begin{array}{ll}
0&\mbox{if $u = 0$},\\
\prod_{1\leq i\leq 4}x_i&\mbox{if $u = 1$}.
\end{array}\right.
\end{array}$$
By straightforward computations using Lemma \ref{bdc31}(i) in section four and an admissible monomial basis of $Q^{\otimes 4}_{2(2^{1} - 1) + 2^{1}},$ show that $[\widetilde{f}] = 0,$ and therefore $[f] = ([\widetilde{f}])^{*} = 0.$ This means that  $k\otimes_{GL_4(k)}P((P_4)_{2(2^{1} - 1) + 2^{1}}^{*})$ is trivial.\\[1mm]
For $s \in \{2, 4\},$ combining Theorems \ref{dlc1}, \ref{dlc3} with the inequality \eqref{pt8} and the fact that the invariant space $(Q^{\otimes 4}_{2^{2-1} + 2^{2}-3})^{GL_4(k)}$ is trivial (see Sum \cite{N.S2}), it may be concluded that the coinvariant spaces $k\otimes_{GL_4(k)}P((P_4)_{2(2^{s} - 1) + 2^{s}}^{*})$ are trivial, too.\\[1mm]
For $s \not\in \{1, 2, 4\},$ the following inequality is immediate from Theorems \ref{dlc1} and \ref{dlc3} and the inequality \eqref{pt8}:
\begin{equation}\label{pt10}
\dim k\otimes_{GL_4(k)}P((P_4)_{2(2^{s} - 1) + 2^{s}}^{*})\leq 1.
\end{equation}
On the other side, we observe that $\lambda_{2^{s}-1}\in \Lambda^{1, 2^{s}-1}$ and $\lambda_7^{2}\lambda_5 = Sq^{0}(\lambda_3^{2}\lambda_2)\in \Lambda^{3, 19}$ are cycles in $\Lambda,$ and are representative of $h_s\in {\rm Ext}_A^{1, 2^{s}}(k, k)$ and $c_1 = Sq^{0}(c_0)\in {\rm Ext}_A^{3, 22}(k, k)$ respectively. Moreover, it is routine to verify that the elements 
$$ \begin{array}{ll}
\medskip
\zeta_3 &= (a_1^{(3)}a_2^{(7)}a_3^{(7)}a_4^{(5)} + a_1^{(3)}a_2^{(7)}a_3^{(9)}a_4^{(3)} + a_1^{(3)}a_2^{(11)}a_3^{(5)}a_4^{(3)} +  a_1^{(3)}a_2^{(13)}a_3^{(3)}a_4^{(3)})\in (P_4)^{*}_{2(2^{3}-1) + 2^{3}},\\
\zeta_s&= a_1^{(1)}a_2^{(2^{s-1}-1)}a_3^{(2^{s-1}-1)}a_4^{(2^{s+1}-1)}\in (P_4)^{*}_{2(2^{s}-1) + 2^{s}},\ \mbox{for $s\geq 5$}
\end{array}$$
are $\widehat{A}$-annihilated. So, by Theorem \ref{dlCH}, it implies that the cycles $\psi_4(\zeta_3) = \lambda_3\lambda_7^{2}\lambda_5$ and $\psi_4(\zeta_s) = \lambda_1\lambda_{2^{s-1}-1}^{2}\lambda_{2^{s+1}-1}$ in $\Lambda$ are representative of the elements $h_2c_1\in {\rm Ext}_A^{4, 3.2^{3}+2}(k, k)$ and $h_1h_{s-1}^{2}h_{s+1}\in {\rm Ext}_A^{4, 3.2^{s}+2}(k, k),$ respectively. It should be noted that with the $\widehat{A}$-annihilated elements $a^{(2^{s}-1)}\in (P_1)_{2^{s}-1}^{*}$ and 
$$ \begin{array}{ll}
\medskip
\widehat{\zeta} &= (a_1^{(7)}a_2^{(7)}a_3^{(5)} + a_1^{(7)}a_2^{(9)}a_3^{(3)}+ a_1^{(11)}a_2^{(5)}a_3^{(3)} + a_1^{(13)}a_2^{(3)}a_3^{(3)})\in (P_3)_{19}^{*},
\end{array}$$
we have $h_s=Tr_1^{A}([a_1^{(2^{s}-1)}])$ and $c_1 = Sq^{0}(c_0) = Tr^{A}_3([\widehat{\zeta}])$  (since the classical $Sq^{0}$ commutes with the Kameko $Sq^{0}$ via the rank 3 algebraic transfer). These arguments and the inequality \eqref{pt10} imply that
\begin{equation}\label{pt11}
\dim k\otimes_{GL_4(k)}P((P_4)_{2(2^{s} - 1) + 2^{s}}^{*}) = 1,\ \mbox{for $s\not\in \{1, 2, 4\}$},
\end{equation} 
and that the cohomological transfer is an epimorphism in the bidegree $(4, 3.2^{s}+2).$  Moreover, according to Theorem \ref{dlntg}, we easily get
\begin{equation}\label{pt12}
{\rm Ext}_A^{4, 3.2^{s}  +2}(k, k)\\
=  \left\{\begin{array}{ll}
\langle h_1h_3^{3}, h_2c_1\rangle = \langle h_2c_1\rangle &\mbox{if $s = 3$},\\
 0 &\mbox{if $s\in \{1, 2, 4\}$},\\
\langle h_1h_s^{3}\rangle  &\mbox{if $s > 4$},\\
\end{array}\right.
\end{equation}
where $h_1h_s^{3}=h_1h_{s-1}^{2}h_{s+1}\neq 0,$ from which, by \eqref{pt11}, we have immediately

\begin{corl}\label{hqc3}
The Singer algebraic transfer is an isomorphism in bidegree $(4, 3.2^{s}+2)$ for every positive integer $s.$
\end{corl}

Finally, we consider the pair $(r, v) = (2, 2^{t}-1)$ for $t\geq 5$ and obtain the following theorem.

\begin{thm}\label{dlct2}
Let $s$ and $t$ be positive integers such that $t\geq 5.$ Then, 
$$k\otimes_{GL_4(k)}P((P_4)_{2(2^{s} - 1) + 2^{s}(2^{t}-1)}^{*}) \\
 = \left\{\begin{array}{ll}
\langle [\zeta_{s,\, t}] \rangle &\mbox{if $s = 1,\, 2$},\\[1mm]
\langle [\zeta_{s,\, t}], [\widetilde{\zeta}_{s,\, t}] \rangle &\mbox{if $s = 3,\, 4$},\\[1mm]
\langle [\zeta_{s,\, t}], [\widetilde{\zeta}_{s,\, t}], [\widehat{\zeta}_{s,\, t}] \rangle &\mbox{if $s \geq 5$},
\end{array}\right.$$
where $$ \begin{array}{lll}
\medskip
\zeta_{s,\, t}&:=a_1^{(1)}a_2^{(2^{s}-1)}a_3^{(2^{s+t-1}-1)}a_4^{(2^{s+t-1}-1)},\\
\medskip
 \widetilde{\zeta}_{s,\, t}&:= a_1^{(0)}a_2^{(0)}a_3^{(2^{s}-1)}a_4^{(2^{s+t}-1)},\\
\medskip
\widehat{\zeta}_{s,\, t}&:= a_1^{(1)}a_2^{(2^{s-1}-1)}a_3^{(2^{s-1}-1)}a_4^{(2^{s+t}-1)}.
\end{array}$$
\end{thm}

The theorem indicates that the elements $\zeta_{s,\, t},$ $\widetilde{\zeta}_{s,\, t}$  and $\widehat{\zeta}_{s,\, t}$ belong to $P((P_4)_{2(2^{s} - 1) + 2^{s}(2^{t}-1)}^{*}).$ So, by Theorem \ref{dlCH}, $\psi_4(\zeta_{s,\, t})$ are cycles in $\Lambda;$ moreover and for each $t\geq 5,$ using the representation in $\Lambda$ of $Tr_4^{A},$ it may be concluded that
$$ \begin{array}{ll}
Tr_4^{A}([\zeta_{s,\, t}])  &= [\psi_4(\zeta_{s,\, t})] = [\lambda_1\lambda_{2^{s}-1}\lambda^{2}_{2^{s+t-1}-1}]\\
& = h_1h_{s}h_{s+t-1}^{2}\in {\rm Ext}_{A}^{4, 2^{s+t} + 2^{s} + 2}(k, k),\ \mbox{for $s\geq 1, \ s\neq 2$},\\[1mm]
Tr_4^{A}([\widetilde{\zeta}_{s,\, t}])  &= [\psi_4(\widetilde{\zeta}_{s,\, t})] = [\lambda_0^{2}\lambda_{2^{s}-1}\lambda_{2^{s+t}-1}]\\
& = h_0^{2}h_sh_{s+t}\in {\rm Ext}_{A}^{4, 2^{s+t} + 2^{s} + 2}(k, k), \ \mbox{for $s \geq 2$},\\[1mm]
Tr_4^{A}([\widehat{\zeta}_{s,\, t}])  &= [\psi_4(\widehat{\zeta}_{s,\, t})] = [\lambda_1\lambda^{2}_{2^{s-1}-1}\lambda_{2^{s+t}-1}]\\
& = h_1h^{2}_{s-1}h_{s+t}\in {\rm Ext}_{A}^{4, 2^{s+t} + 2^{s} + 2}(k, k),\ \mbox{for $s\geq 5$}.
\end{array}$$
On the other side, using Theorem \ref{dlntg}, one has that
$$ {\rm Ext}_{A}^{4, 2^{s+t} + 2^{s} + 2}(k, k) = \left\{\begin{array}{ll}
\langle h_1^{2}h_7^{2}, D_3(2) \rangle &\mbox{if $s = 1$ and $t = 7$},\\[1mm]
\langle h_1^{2}h_t^{2} \rangle &\mbox{if $s = 1$ and $t \geq 5,\, t\neq 7$},\\[1mm]
\langle h_0^{2}h_2h_{t+2} \rangle = \langle h_1^{3}h_{t+2} \rangle &\mbox{if $s = 2$ and $t\geq 5$},\\[1mm]
\langle h_1h_sh_{s+t-1}^{2}, h_0^{2}h_sh_{s+t}, h_1h_{s-1}^{2}h_{s+t} \rangle &\mbox{if $s \geq 3$ and $t\geq 5$},
\end{array}\right.$$
where $h_1h_{s-1}^{2}h_{s+t}  = 0$ if $s = 3,\, 4,$ and $t \geq  5.$ Combining these data with a fact in \cite{Hung} that the rank 4 transfer does not detect the element $D_3(2)$ in $Sq^{0}$-family $\{D_3(s):\, s\geq 0\}$ in ${\rm Ext}_{A}^{4, 2^{s}+2^{s+6}}(k,k)$ for all $s\geq 0,$ we are forced to conclude that

\begin{corl}\label{hqct2}
The transfer homomorphism $$ Tr_4^{A}: k\otimes_{GL_4(k)}P((P_4)_{2(2^{s} - 1) + 2^{s}(2^{t}-1)}^{*})  \to {\rm Ext}_{A}^{4, 2^{s+t} + 2^{s} + 2}(k, k)$$ is not an epimorphism if $s = 1$ and $t = 7$ and is an isomorphism if $s\geq 1$ and $t\geq 5,\, t\neq 7.$
\end{corl}

Moreover, the following remark is useful.

\begin{rema}\label{nxqt}
Let us consider the generic degrees of the form \eqref{ct} with $r = 4$ and $v = 61,$ (i.e, $n = 2^{s+6} + 2^{s} - 4.$) Then, it is easily verified that $\mu(n) = 4$ for all $s > 2.$ This leads to the iterated Kameko map $(\overline{Sq}^{0})^{s-2}: Q^{\otimes 4}_{n} \to Q^{\otimes 4}_{2^{2+6} + 2^{2} - 4}$ being an isomorphism of $kGL_4(k)$-modules for any $s\geq 2.$ So, by Theorem \ref{dlct2} and the results in Sum \cite{N.S2},  deduce that the coinvariant space $k\otimes_{GL_4(k)}P((P_4)_{n}^{*})$ is trivial if $s = 0$ and has dimension $1$ if $s \geq 1.$ On the other hand, from Theorem \ref{dlntg}, it implies that
$${\rm Ext}_{A}^{4, 4+n}(k, k) = \left\{\begin{array}{ll}
 \langle D_3(0) \rangle &\mbox{if $s = 0$},\\
 \langle D_3(s), h_ {s-1}^{2}h_{s+5}^{2}\rangle &\mbox{if $s \geq 1$}.
\end{array}\right.
$$ 
Moreover, following H\uhorn ng \cite{Hung}, the elements $D_3(s)$ are not detected by $Tr_4^{A},$ and therefore the fourth algebraic transfer $Tr_4^{A}: k\otimes_{GL_4(k)}P((P_4)_{2^{s+6} + 2^{s} - 4}^{*})  \to {\rm Ext}_{A}^{4, 2^{s+6} + 2^{s}}(k, k)$ is a monomorphism, but not an isomorphism for every non-negative integer $s.$
\end{rema}

Thus, observing from Corollaries \ref{hqc1}, \ref{hqc2}, \ref{hqct}, \ref{hqc3} and \ref{hqct2}, we can assert that

\begin{corl}\label{hqc4}
Let $s$ and $t$ be two positive integers. Then, Singer's conjecture for the rank 4 transfer holds in degree $2^{s+t+1} + 2^{s+1} - 3$ for $t\neq 3$ and degree $2^{s+t} + 2^{s} - 2$ for $t = 1$ and arbitrary $t\geq 5.$
\end{corl}

\section{Final remarks}

We conclude this article with some remarks. It may need to be recalled that $Q^{\otimes 4}_n$ is trivial if $\mu(n) > 4$ and that $Q^{\otimes 4}_n\cong Q^{\otimes 4}_{(n-4)/2}$ if $\mu(n) = 4.$ This means that we need only to study the cohit module $Q^{\otimes 4}_n$ in degrees $n$ with $\mu(n) < 4$ (or the generic degrees $n$ of the form \eqref{ct}). So, since the domain of the fourth Singer transfer is dual to the invariant $(Q^{\otimes 4}_n)^{GL_4(k)},$ it suffices to verify Singer's conjecture for $Tr_4^{A}$ in the following degrees $n$: 
$$ \begin{array}{lll}
i)  &n &= 2^{s+1} - \ell,\ \ell\in \{1, 2, 3\}, \\[1mm]
ii) &n &= 2^{s+t+1} +2^{s+1} - 3,\\[1mm]
iii) &n &= 2^{s+t} + 2^{s} - 2,\\[1mm]
iv) &n &= 2^{s+t+u} + 2^{s+t} + 2^{s} - 3,
\end{array}$$
whenever $s,\, t,\, u$ are the positive integers (see also \cite{N.S1}).  It is easy to check that the above degrees can be rewritten as \eqref{ct}. For example, consider i), we have $n = \ell(2^{s} - 1) + (2-\ell).2^{s}$ with $\ell\leq 2$ and $\mu(2-\ell) < \ell.$ When $\ell = 3$ and any $s\geq 1,$ we may rewrite $2^{s+1} -3$ as $2^{s+2}-3$ with $s\geq 0$ and so $n = 3(2^{s}-1) + 1.2^{s}$ where $\mu(1) = 1 < 3 < 4.$

The case i) is known by Sum \cite{N.S2}. It is remarkable that, by dualizing and Sum's work \cite{N.S2}, in this case when the pairs $(\ell, s) = (3, 5)$ and $(2, 6),$ we deduce that $k\otimes_{GL_4(k)}P((P_4)_{2^{5+1}-3}^{*})$ is trivial and that $k\otimes_{GL_4(k)}P((P_4)_{2^{6+1}-2}^{*})$ is 1-dimensional. On the other hand, following Theorem \ref{dlntg} and Remark \ref{nxqt}, we have seen that ${\rm Ext}_A^{4, 2^{5+1} + 1}(k, k) = \langle D_3(0)\rangle$ and ${\rm Ext}_A^{4, 2^{6+1} + 2}(k, k) = \langle h_0^{2}h_6^{2},  D_3(1)\rangle,$ where $h_0^{2}h_6^{2} = [\psi_4(a_1^{(0)}a_2^{(0)}a_3^{(63)}a_4^{(63)})] = Tr_4^{A}([a_1^{(0)}a_2^{(0)}a_3^{(63)}a_4^{(63)}])\in {\rm Im}(Tr_4^{A})$ and $D_3(s)\not\in {\rm Im}(Tr_4^{A})$ for every non-negative integer $s.$  Combining these data, we may state that $Tr_4^{A}$ is not an isomorphism when acting on the coinvariants $k\otimes_{GL_4(k)}P((P_4)_{2^{5+1}-3}^{*})$ and $k\otimes_{GL_4(k)}P((P_4)_{2^{6+1}-2}^{*}).$ 

The case iii) where $t = 2, 4$ and the case iv) have been investigated by us in \cite{D.P11, D.P12}.
The results have been computed in this paper for the case ii) with $t \neq 3$ and the case iii) with $t = 1$ and $t\geq 5.$ Thus, Singer's conjecture is undetermined for the cases ii) and iii) where $t = 3$. We now will discuss these remaining cases.

Firstly,  let us consider the case ii) in which $t = 3.$ Naturally, we are tempted to propose the following in viewpoint of Theorems \ref{dlc1}, \ref{dlc2}, \ref{dlct}, \ref{dlc3}, \ref{dlct2} and our previous works \cite{D.P11, D.P12}.

 \begin{conj}\label{gtP1}
For a positive integer $s,$ the coinvariant $k\otimes_{GL_4(k)}P((P_4)_{2^{s+4} +2^{s+1} - 3}^{*})$ is trivial if $s = 2$ and has dimension $1$ if $s\neq 2.$
\end{conj}
On the other hand, by Theorem \ref{dlntg}, we have 
$$ {\rm Ext}_A^{4, 2^{s+4} +2^{s+1} +1}(k, k)\\
=  \left\{\begin{array}{ll}
\langle p_0\rangle &\mbox{if $s = 1$},\\[1mm]
 \langle p'_0\rangle&\mbox{if $s = 2$},\\[1mm]
\langle h_0h_{s+1}h_{s+3}^{2}, h_0h_s^2h_{s+4}\rangle  = \langle h_0h_s^2h_{s+4} \rangle &\mbox{if $s \geq 3$}.
\end{array}\right.$$
Moreover, it is straightforward to see that the elements $ a_1^{(0)}a_2^{(2^{s}-1)}a_3^{(2^{s}-1)}a_4^{(2^{s+4}-1)}$ belong to ${\rm Ext}_A^{0, 2^{s+4} +2^{s+1} - 3}(k, P_4).$ So, following Theorem \ref{dlCH}, the cycles   $$\lambda_0\lambda_{2^{s}-1}^{2}\lambda_{2^{s+4}-1}=\psi_4(a_1^{(0)}a_2^{(2^{s}-1)}a_3^{(2^{s}-1)}a_4^{(2^{s+4}-1)})$$ in $\Lambda$ are  representative of the non-zero elements $h_0h_{s}^{2}h_{s+4}$ for any $s \geq 3$. This fact and the previous results by H\uhorn ng-Qu\`ynh \cite{H.Q} show that the elements $p_0$ and $h_0h_s^2h_{s+4}$ are in the image of $Tr_4^{A},$ except $p'_0.$ And therefore, by Conjecture \ref{gtP1}, it follows that
\begin{corl}\label{hqnx1}
The transfer homomorphism $$Tr_4^A: k\otimes_{GL_4(k)}P((P_4)_{2^{s+4} +2^{s+1} - 3}^{*}) \to {\rm Ext}_A^{4, 2^{s+4} +2^{s+1} +1}(k, k)$$ is an isomorphism for $s\neq 2$, but it is not an epimorphism for $s = 2.$ 
\end{corl}

Next, we discuss the case iii) with $t = 3.$ Because Kameko's squaring operation $$[\overline{Sq}^{0}]_{2^{s+3} + 2^{s}-2}:= \overline{Sq}^{0}: Q^{\otimes 4}_{2^{s+3} + 2^{s}-2} \to Q^{\otimes 4}_{2^{s+2} + 2^{s-1}-3}$$ is an epimorphism of $kGL_4(k)$-modules, we have an isomorphism
\begin{equation}\label{dcc}
 Q^{\otimes 4}_{2^{s+3} + 2^{s}-2}  \cong {\rm Ker}[\overline{Sq}^{0}]_{2^{s+3} + 2^{s}-2} \bigoplus Q^{\otimes 4}_{2^{s+2} + 2^{s-1} - 3}.
\end{equation}
Following Sum \cite{N.S2} and Conjecture \ref{gtP1}, the invariant spaces $(Q^{\otimes 4}_{2^{s+2} + 2^{s-1} - 3})^{GL_4}$ are trivial if $s = 1,\, 4$ and are 1-dimensional otherwise. On the other side, by Theorem \ref{dlntg}, it implies that
$$ 
{\rm Ext}_{A}^{4, 2^{s+3} + 2^{s}+2}(k, k) = \left\{\begin{array}{ll}
\langle h_1^{2}h_3^{2}  \rangle = 0 &\mbox{if $s = 1$},\\[1mm]
\langle h_0^{2}h_2h_5,  h_1h_2h_4^{2} \rangle = \langle h_1^{3}h_5 \rangle &\mbox{if $s = 2$},\\[1mm]
\langle h_0^{2}h_3h_6, h_1h_2^{2}h_6, h_1h_3h_5^{2}, p_1 \rangle = \langle  h_0^{2}h_3h_6, p_1\rangle &\mbox{if $s = 3$},\\[1mm]
\langle h_0^{2}h_4h_7, h_1h_3^{2}h_7, h_1h_4h_6^{2} , p'_1 \rangle  = \langle h_0^{2}h_4h_7, p'_1 \rangle &\mbox{if $s = 4$},\\[1mm]
\langle h_0^{2}h_sh_{s+3}, h_1h_{s-1}^{2}h_{s+3}, h_1h_sh_{s+2}^{2} \rangle = \langle h_0^{2}h_sh_{s+3}, h_1h_{s-1}^{2}h_{s+3} \rangle &\mbox{if $s \geq 5$}.
\end{array}\right.$$ 
Taking these data, we make the following additional conjecture.

\begin{conj}\label{gtP2}
The invariant spaces $({\rm Ker}[\overline{Sq}^{0}]_{2^{s+3} + 2^{s}-2})^{GL_4}$ are trivial if $s = 1,\, 2$ and have dimension $1$ if $s\geq 3.$ 
\end{conj}

As is well known, $p_1\in {\rm Im}(Tr_4^{A})$ (see H\uhorn ng-Qu\`ynh \cite{H.Q}) and $p_1'\not\in {\rm Im}(Tr_4^{A})$ (see H\uhorn ng \cite{Hung}). These, together with the equality \eqref{dcc} and Conjecture \ref{gtP2}, tempt us to propose:

\begin{conj}\label{gtP3}
For a positive integer $s,$ then 
$$ \dim k\otimes_{GL_4(k)}P((P_4)_{2^{s+3} + 2^{s}-2}^{*}) =\left\{\begin{array}{ll}
0&\mbox{if $s = 1$},\\[1mm]
1&\mbox{if $s = 2$},\\[1mm]
2&\mbox{if $s = 3$},\\[1mm]
1&\mbox{if $s = 4$},\\[1mm]
2&\mbox{if $s \geq 5$}.
\end{array}\right.$$ 
\end{conj}

So, the following corollary is immediate.

\begin{corl}\label{hqnx2}
The cohomological transfer $$Tr_4^A: k\otimes_{GL_4(k)}P((P_4)_{2^{s+3} + 2^{s}-2}^{*}) \to {\rm Ext}_A^{4, 2^{s+3} +2^{s} +2}(k, k)$$ is an isomorphism for $s\neq 4$, but it is not an epimorphism for $s = 4.$ 
\end{corl}

Additionally, one has the following important observation: Consider the generic degree $2^{s+6} + 2^{s+3} + 2^{s} - 4$ with $s$ an arbitrary non-negative integer. It is not hard to check that $\mu(2^{s+6} + 2^{s+3} + 2^{s} - 4) = 4$ for all $s > 1$ and so by iteration of the Kameko homomorphism, we deduce that the map $ (\overline{Sq}^{0})^{s-1}: Q^{\otimes 4}_{2^{s+6} + 2^{s+3} + 2^{s} - 4} \to Q^{\otimes 4}_{2^{1+6} + 2^{1+3} + 2^{1} - 4}$ is an isomorphism of $kGL_4(k)$-modules, for arbitrary $s\geq 1.$ Note that $Q^{\otimes 4}_{2^{1+6} + 2^{1+3} + 2^{1} - 4} = Q^{\otimes 4}_{2^{4+3} + 2^{4}-2}.$ Then, due to a previous result in Sum \cite{N.S2} and Conjecture \ref{gtP1}, it can be easily seen that the coinvariant $k\otimes_{GL_4(k)}P((P_4)_{2^{s+6} + 2^{s+3} + 2^{s} - 4}^{*})$ is trivial if $s = 0$ and is 1-dimensional if $s\geq 1.$ On the other side, clearly, Theorem \ref{dlntg} implies that
$${\rm Ext}_{A}^{4, 2^{s+6} + 2^{s+3} + 2^{s}}(k, k) = \left\{\begin{array}{ll}
 \langle p'_0 \rangle &\mbox{if $s = 0$},\\
 \langle p'_s, h_{s-1}^{2}h_{s+3}h_{s+6}\rangle &\mbox{if $s \geq 1$},
\end{array}\right.
$$ 
and moreover, the non-zero elements $p'_s$ are not detected by $Tr_4^{A}$ (see H\uhorn ng \cite{Hung}). Therefore, the fourth transfer  $Tr_4^A: k\otimes_{GL_4(k)}P((P_4)_{2^{s+6} + 2^{s+3} + 2^{s} - 4}^{*}) \to {\rm Ext}_A^{4, 2^{s+6} + 2^{s+3} + 2^{s}}(k, k)$
is a monomorphism, but not an isomorphism for any $s\geq 0.$

Finally, we believe that Conjectures \ref{gtP1}, \ref{gtP2} and \ref{gtP3} are true and verification will be soon communicated.  Taking Corollaries \ref{hqc4}, \ref{hqnx1} and \ref{hqnx2}, together with the work in Sum \cite{N.S2} and our previous results \cite{D.P11, D.P12}, it may be concluded that

\begin{corl}\label{hqnx3}
Conjecture \ref{gtS} holds for $q = 4.$
\end{corl}

\begin{acknow}
This work has been supported in part by the NAFOSTED of Viet Nam under grant no. 101.04-2017.05.  
\end{acknow}

\section{Proofs of main results}

The whole section is devoted to prove Theorems \ref{dlc1}, \ref{dlc2}, \ref{dlct}, \ref{dlc3}, and \ref{dlct2}. Firstly, to make the paper self-contained, let us briefly review Singer's algebraic transfer and the relevant notions.

{\bf Singer's algebraic transfer.} Consider the polynomial ring of one variable, $P_1 = k[x_1].$ The canonical $A$-action on $P_1$ is extended to an $A$-action on $k[x_1, x_1^{-1}],$ the ring of finite Laurent series. Then, $\overline{\mathscr P} = \langle \{x_1^{t}|\ t\geq -1\}\rangle$ is $A$-submodule of $k[x_1, x_1^{-1}].$ One has a short-exact sequence:  
\begin{equation}\label{dkn1}
 0\to P_1\xrightarrow{j} \overline{\mathscr P}\xrightarrow{\pi} \Sigma^{-1}k,
\end{equation}
where $j$ is the inclusion and $\pi$ is given by $\pi(x_1^{t}) = 0$ if $t\neq -1$ and $\pi(x_1^{-1}) = 1.$ Denote by $e_1$ the element in ${\rm Ext}_{A}^1(\Sigma^{-1}k, P_1),$ which is represented by the cocyle associated to \eqref{dkn1}. For each $i\geq 1,$ the short-exact sequence
\begin{equation}\label{dkn2}
 0\to P_{i+1}\cong P_i\otimes_{k} P_1\xrightarrow{1\otimes_{k}j} P_i\otimes_{k}\overline{\mathscr P}\xrightarrow{1\otimes_{k}\pi} \Sigma^{-1}P_i,
\end{equation}
determines a class $(e_1\times P_i)\in {\rm Ext}_{A}^1(\Sigma^{-1}P_i, P_{i+1}).$ Then, using the cross product and Yoneda product, we have the element
$$ e_q = (e_1\times P_{q-1})\circ (e_1\times P_{q-2})\circ \cdots \circ (e_1\times P_{1})\circ e_1 \in {\rm Ext}_{A}^q(\Sigma^{-q}k, P_q).$$
Let $\Delta(e_1\times P_i): {\rm Tor}_{q-i}^{A}(k, \Sigma^{-q}P_i)\to {\rm Tor}_{q-i-1}^{A}(k, P_{i+1})$ be the connecting homomorphism associated to \eqref{dkn2}. Then, we have a composition of the connecting homomorphisms $$ \overline{\varphi}^{A}_q = \Delta(e_1\times  P_{q-1})\circ \Delta(e_1\times P_{q-2})\circ \cdots \circ \Delta(e_1\times P_{1})\circ \Delta(e_1)$$from ${\rm Tor}_q^{A}(k, \Sigma^{-q}k)$ to ${\rm Tor}_0^{A}(k, P_q) = k\otimes_{A} P_q$, determined by $\overline{\varphi}_q(z) = e_q\cap z$ for any $z\in {\rm Tor}_q^{A}(k, \Sigma^{-q}k).$ Here $\cap$ denotes the \textit{cap product} in homology with $k$-coefficients. The image of $\overline{\varphi}_q$ is a submodule of the invariant space $(k \otimes_{A} P_q)^{GL_q(k)}.$ Hence, $\overline{\varphi}^{A}_q$ induces homomorphism
$$ \varphi_q^{A}: {\rm Tor}_q^{A}(k, \Sigma^{-q}k)\to (k \otimes_{A} P_q)^{GL_q(k)}.$$
Since the supension $\Sigma^{-q}$ induces an isomorphism ${\rm Tor}_{q, n}^{A}(k, \Sigma^{-q}k)\cong {\rm Tor}_{q, q+n}^{A}(k, k),$ we have the homomorphism $\varphi_q^{A}: {\rm Tor}_{q, q+n}^{A}(k, k)\to (k \otimes_{A} P_q)_{n}^{GL_q(k)}.$ So its dual $$Tr_q^{A}: k\otimes_{GL_q(k)}P((P_q)_n^{*})\to {\rm Ext}_{A}^{q, q+n}(k, k)$$ is called the \textit{cohomological transfer} (or \textit{algebraic transfer}). 

{\bf Some definitions.} A sequence of non-negative integers $\omega = (\omega_1, \omega_2, \ldots, \omega_i,\ldots)$ is called a \textit{weight vector}, if $\omega_i  = 0$ for $i\gg 0.$ One defines $\deg(\omega) = \sum_{i\geq 1}2^{i-1}\omega_i,$ For a natural number $n,$ let us denote by $\alpha_j(n)$ the $j$-th coefficients in dyadic expansion of $n,$ from which we have  $\alpha(n) = \sum_{j\geq 0}\alpha_j(n)$ and $n = \sum_{j\geq 0}\alpha_j(n)2^j,$ where $\alpha_j(n)\in \{0, 1\},\ j = 0, 1, \ldots.$ Let $x = \prod_{1\leq i\leq 4}x_i^{a_i}$ be a monomial in $P_4,$ define two sequences associated with $x$ by $ \omega(x) :=(\omega_1(x), \omega_2(x), \ldots, \omega_j(x), \ldots)$ and $(a_1, a_2, \ldots, a_4),$ where $\omega_j(x)=\sum_{1\leq i\leq 4}\alpha_{j-1}(a_i),$ for $j\geq 1.$ These sequences are respectively called the {\it weight vector} and the \textit{exponent vector} of $x.$ By convention, the sets of all the weight vectors and the exponent vectors are given the left lexicographical order.

\begin{defn}
Assume that $x = \prod_{1\leq i\leq 4}x_i^{a_i}$ and $y = \prod_{1\leq i\leq 4}x_i^{b_i}$ are the monomials in $(P_4)_n.$ We write $a,\, b$ for the exponent vectors of $x$ and $y,$ respectively. We say that $x  < y$ if and only if one of the following holds:
\begin{enumerate}
\item[(i)] $\omega(x) < \omega(y);$
\item[(ii)] $\omega(x) = \omega(y)$ and $a < b.$
\end{enumerate}
\end{defn}

Let us recall that the algebra $A$ is generated by Steenrod squares $Sq^{i}$ for $i\geq 0$ and that the action of $A$ on the polynomial ring $P_4 = k[x_1, x_2, x_3, x_4]$ is  determined by $Sq^{i}(x_i^{n}) = \binom{n}{i}x^{n+i}$ for every $x_i\in P_4$ and Cartan's formula $Sq^{i}(fg) = \sum_{m+n = i}Sq^{i}(f)Sq^{i}(g),$ for all the homogeneous polynomials $f,\, g\in P_4.$ One should note that $\binom{n}{i} = 0$ if $i > n.$

\begin{defn} For a weight vector $\omega$ of degree $n,$ we denote two subspaces associated with $\omega$: 
$$ \begin{array}{ll}
\medskip
(P_4)_n^{\omega} &= \langle\{ x\in (P_4)_n|\, \deg(x) = \deg(\omega),\  \omega(x)\leq \omega\}\rangle,\\
(P_4)_n^{< \omega} &= \langle \{ x\in (P_4)_n|\, \deg(x) = \deg(\omega),\  \omega(x) < \omega\}\rangle.
\end{array}$$ 
For homogeneous polynomials $f$ and $g$ in $(P_4)_n.$ We define the equivalence relations "$\equiv$" and "$\equiv_{\omega}$" on $(P_4)_n$:    
\begin{enumerate}
\item [(i)]$f \equiv g $ if and only if $(f - g)\in \widehat{A}(P_4)_n.$ 
\item[(ii)] $f \equiv_{\omega} g$ if and only if $f, \, g\in (P_4)_n^{\omega}$ and $(f -g)\in ((\widehat{A}(P_4)_n \cap (P_4)_n^{\omega}) + (P_4)_n^{< \omega}).$\\[1mm]
Specifically, if $f\equiv 0$ (resp. $f\equiv_{\omega} 0$), then $f$ is called \textit{hit} (resp. \textit{$\omega$-hit}).
\end{enumerate}
\end{defn}
From the equivalence relation "$\equiv_\omega$", we have the $k$-quotient space:
$$(Q_n^{\otimes 4})^{\omega} := (P_4)_n^{\omega}/ (((P_4)_n^{\omega}\cap \widehat{A}(P_4)_n)+ (P_4)_n^{< \omega}).$$ 
According to Sum \cite{N.S4}, this space has the structure of an $GL_4$-module. Furthermore, by the definitions above, it is straightforward to show that
$$ \begin{array}{ll}
\medskip
\dim Q^{\otimes 4}_n &=\sum_{\deg(\omega) = n} \dim (Q_n^{\otimes 4})^{\omega},\\ 
 \dim (Q^{\otimes 4}_n)^{GL_4(k)}&\leq \sum_{\deg(\omega) = n}\dim ((Q_n^{\otimes 4})^{\omega})^{GL_4(k)}.
\end{array}$$

Let $\underline{(P_4)_n}$ and $\overline{(P_4)_n}$ denote the $k$-subspaces of $(P_4)_n$ spanned all the monomials $\prod_{1\leq i\leq 4}x_i^{a_i}$ such that $\prod_{1\leq i\leq 4}a_i = 0,$ and $\prod_{1\leq i\leq 4}a_i > 0,$ respectively. We put $$ \underline{Q^{\otimes 4}_n} = k\otimes_{A}\underline{(P_4)_n} ,\ \mbox{and}\ \overline{Q^{\otimes 4}_n}= k\otimes_{A}\overline{(P_4)_n} .$$ Then, we have $$Q^{\otimes 4}_n\cong \underline{Q^{\otimes 4}_n}\bigoplus \overline{Q^{\otimes 4}_n}.$$

\begin{defn}
A monomial $x\in (P_4)_n$ is said to be {\it inadmissible} if there exist monomials $y_1, y_2,\ldots, y_k$ in $(P_4)_n$ such that $y_j < x$ for $1\leq j\leq k$ and $x \equiv  \sum_{1\leq j\leq k} y_j.$ Then, $x$ is said to be {\it admissible} if it is not inadmissible.
\end{defn}

With the above definitions on hand, it is easily seen that $Q^{\otimes 4}_n$ is a $k$-vector space with a basis consisting of all the classes represent by the admissible monomials in $(P_4)_n.$

\begin{defn}\label{dnspike}
A monomial $z = x_1^{v_1}\ldots x_4^{v_4}$ in $(P_4)_n$ is called a {\it spike} if every exponent $v_j$ is of the form $2^{\xi_j} - 1.$ In particular, if the exponents $\xi_j$ can be arranged to satisfy \mbox{$\xi_1 > \xi_2 > \ldots > \xi_{m-1}\geq \xi_m \geq 1$} where only the last two smallest exponents can be equal, and $\xi_r = 0$ for $ m < r  \leq 4,$ then $z$ is called a {\it minimal spike}.
\end{defn}

According to Singer \cite{W.S1}, we have the following technical result on a hit monomial.

\begin{thm}\label{dlSi}
For each $q\geq 1$ and degree $n\geq 0,$ suppose that the monomial $x$ belongs to $(P_q)_n$ and $\mu(n)\leq q.$ Let $z$ be a minimal spike in $(P_q)_n.$ Therefrom, if $\omega(x) < \omega(z),$ then $x$ is hit.
\end{thm}

\begin{defn}
For each $1\leq j\leq 4,$ one defines $k$-homomorphism $$\sigma_j: k^{4}\cong \langle x_1, \ldots, x_4 \rangle\to k^{4}\cong \langle x_1, \ldots, x_4 \rangle,$$ which is given by 
$$ \left\{\begin{array}{ll}
\sigma_j(x_j) &= x_{j+1},\\
 \sigma_j(x_{j+1}) &= x_j,\\
 \sigma_j(x_t) &= x_t,
\end{array}\right.$$
 for $t\not\in \{j, j+1\},\ j = 1, 2, 3,$ and $\sigma_4(x_1) = x_1 + x_2,\ \sigma_4(x_t) = x_t$ for $2\leq t\leq 4.$ 
\end{defn}

The following remark is useful.

\begin{rema}
We denote by $\Sigma_4$ the symmetric group of rank $4.$ Then, $\Sigma_4$ is generated by the ones associated with $\sigma_j,\ 1\leq j\leq 3.$ For each permutation in $\Sigma_4$, consider corresponding permutation matrix; these form a group of matrices isomorphic to $\Sigma_4.$ 
So, 
$GL_4(k)\cong GL(k^{4})$ is generated by the matrices associated with $\sigma_j,\ 1\leq j\leq 4.$ Let $x = \prod_{1\leq i\leq 4}x_i^{a_i}$ be an monomial in $(P_4)_n.$ Then, the weight vector $\omega$ of $x$ is invariant under the permutation of the generators $x_i,\ i = 1, 2, \ldots, 4.$ Hence $(Q^{\otimes 4}_n)^{\omega}$ also has a $\Sigma_4$-module structure. 
\end{rema}

For a polynomial $f\in (P_4)_n,$ we denote by $[f]$ the classes in $Q^{\otimes 4}_n$ represented by $f.$ If $\omega$ is a weight vector and $f\in (P_4)_n^{\omega},$ then denote by $[f]_\omega$ the classes in $(Q^{\otimes 4}_n)^{\omega}$ represented by $f.$ By the above remark and $\sigma_j$ induces a homomorphism of $A$-algebras which is also denoted by $\sigma_j: P_4\to P_4,$  a class $[f]_{\omega}\in (Q^{\otimes 4}_n)^{\omega}$ is an $GL_4(k)$-invariant if and only if $\sigma_j(f)  \equiv_{\omega} f$ for $1\leq j\leq 4.$  It is an $\Sigma_4$-invariant if and only if $\sigma_j(f) \equiv_{\omega} f $ for $1\leq j\leq 3.$

In what follows, for any monomials $v_1, v_2, \ldots, v_s\in (P_4)_n$ and for a subgroup $G$ of $GL_4(k),$ we denote by $G(v_1; v_2; \ldots, v_s)$ the $G$-submodule of $(Q^{\otimes 4}_n)^{\omega}$ generated by the set $\{[v_i]_{\omega}:\, 1\leq i\leq s\}.$ Note that if $\omega$ is a weight vector of a minimal spike, then $[v_i]_{\omega} = [v_i]$ for all $i.$

\subsection{Proof of Theorem \ref{dlc1}}

For convenience, put $n_s:=3(2^{s}-1) + 3.2^{s},$  according to Sum \cite{N.S1},  we have
$$\dim Q^{\otimes 4}_{n_s} = \left\{ \begin{array}{ll}
46 &\mbox{if $s = 1$},\\
94 &\mbox{if $s = 2$},\\
105 &\mbox{if $s \geq 3$}.
\end{array}\right.$$
Recall that $Q^{\otimes 4}_{n_s} \cong \underline{Q^{\otimes 4}_{n_s}}\bigoplus \overline{Q^{\otimes 4}_{n_s}}.$ By Sum \cite{N.S1}, a monomial basis for $\underline{Q^{\otimes 4}_{n_s}}$ is the set consisting of all the classes represented monomials $a_{s, i}$ which are determined as follows:
\begin{center}
\begin{tabular}{lrr}
$a_{s,\,1}= x_2^{2^s-1}x_3^{2^s-1}x_4^{2^{s+2}-1}$, & \multicolumn{1}{l}{$a_{s,\,2}= x_2^{2^s-1}x_3^{2^{s+2}-1}x_4^{2^s-1}$,} & \multicolumn{1}{l}{$a_{s,\,3}= x_2^{2^{s+2}-1}x_3^{2^s-1}x_4^{2^s-1}$,} \\
$a_{s,\,4}= x_1^{2^s-1}x_3^{2^s-1}x_4^{2^{s+2}-1}$, & \multicolumn{1}{l}{$a_{s,\,5}= x_1^{2^s-1}x_3^{2^{s+2}-1}x_4^{2^s-1}$,} & \multicolumn{1}{l}{$a_{s,\,6}= x_1^{2^s-1}x_2^{2^s-1}x_4^{2^{s+2}-1}$,} \\
$a_{s,\,7}= x_1^{2^s-1}x_2^{2^s-1}x_3^{2^{s+2}-1}$, & \multicolumn{1}{l}{$a_{s,\,8}= x_1^{2^s-1}x_2^{2^{s+2}-1}x_4^{2^s-1}$,} & \multicolumn{1}{l}{$a_{s,\,9}= x_1^{2^s-1}x_2^{2^{s+2}-1}x_3^{2^s-1}$,} \\
$a_{s,\,10}= x_1^{2^{s+2}-1}x_3^{2^s-1}x_4^{2^s-1}$, & \multicolumn{1}{l}{$a_{s,\,11}= x_1^{2^{s+2}-1}x_2^{2^s-1}x_4^{2^s-1}$,} & \multicolumn{1}{l}{$a_{s,\,12}= x_1^{2^{s+2}-1}x_2^{2^s-1}x_3^{2^s-1}$,} \\
$a_{s,\,13}= x_2^{2^s-1}x_3^{2^{s+1}-1}x_4^{3.2^s-1}$, & \multicolumn{1}{l}{$a_{s,\,14}= x_2^{2^{s+1}-1}x_3^{2^s-1}x_4^{3.2^s-1}$,} & \multicolumn{1}{l}{$a_{s,\,15}= x_2^{2^{s+1}-1}x_3^{3.2^s-1}x_4^{2^s-1}$,} \\
$a_{s,\,16}= x_1^{2^s-1}x_3^{2^{s+1}-1}x_4^{3.2^s-1}$, & \multicolumn{1}{l}{$a_{s,\,17}= x_1^{2^s-1}x_2^{2^{s+1}-1}x_4^{3.2^s-1}$,} & \multicolumn{1}{l}{$a_{s,\,18}= x_1^{2^s-1}x_2^{2^{s+1}-1}x_3^{3.2^s-1}$,} \\
$a_{s,\,19}= x_1^{2^{s+1}-1}x_3^{2^s-1}x_4^{3.2^s-1}$, & \multicolumn{1}{l}{$a_{s,\,20}= x_1^{2^{s+1}-1}x_3^{3.2^s-1}x_4^{2^s-1}$,} & \multicolumn{1}{l}{$a_{s,\,21}= x_1^{2^{s+1}-1}x_2^{2^s-1}x_4^{3.2^s-1}$,} \\
$a_{s,\,22}= x_1^{2^{s+1}-1}x_2^{2^s-1}x_3^{3.2^s-1}$, & \multicolumn{1}{l}{$a_{s,\,23}= x_1^{2^{s+1}-1}x_2^{3.2^s-1}x_4^{2^s-1}$,} & \multicolumn{1}{l}{$a_{s,\,24}= x_1^{2^{s+1}-1}x_2^{3.2^s-1}x_3^{2^s-1}$,} \\
$a_{s,\,25}= x_2^{2^{s+1}-1}x_3^{2^{s+1}-1}x_4^{2^{s+1}-1}$, & \multicolumn{1}{l}{$a_{s,\,26}= x_1^{2^{s+1}-1}x_3^{2^{s+1}-1}x_4^{2^{s+1}-1}$,} & \multicolumn{1}{l}{$a_{s,\,27}= x_1^{2^{s+1}-1}x_2^{2^{s+1}-1}x_4^{2^{s+1}-1}$,} \\
$a_{s,\,28}= x_1^{2^{s+1}-1}x_2^{2^{s+1}-1}x_3^{2^{s+1}-1}$. &       &  
\end{tabular}%
\end{center}

Thanks to these, a direct computation shows that
$$ \begin{array}{ll}
\medskip
\Sigma_4(a_{s,\,1}) &= \langle \{[a_{s,\,j}]:\, 1\leq j\leq 12\} \rangle,\\
\medskip
\Sigma_4(a_{s,\,13}) &= \langle \{[a_{s,\,j}]:\, 13\leq j\leq 24\} \rangle,\\
\medskip
\Sigma_4(a_{s,\,25}) &= \langle \{[a_{s,\,j}]:\, 25\leq j\leq 28\} \rangle.
\end{array}$$ 
These imply that there is a direct summand decomposition of the $\Sigma_4$-modules:
$$ \underline{Q^{\otimes 4}_{n_s}} = \Sigma_4(a_{s,\,1}) \bigoplus \Sigma_4(a_{s,\,13}) \bigoplus \Sigma_4(a_{s,\,25}).$$
In the following, it will play a key role in the proof of the theorem.
\begin{lema}\label{bd1}
We have that
\begin{itemize}
\item[i)]  $\Sigma_4(a_{s,\,1})^{\Sigma_4} = \langle [q_{s,\,1}] \rangle,$ with $q_{s,\,1} = \sum_{1\leq j\leq 12}a_{s,\,j};$
\item[ii)]  $\Sigma_4(a_{s,\,13})^{\Sigma_4} = \langle [q_{s,\,2}] \rangle,$ with $q_{s,\,1} = \sum_{13\leq j\leq 24}a_{s,\,j};$
\item[iii)]  $\Sigma_4(a_{s,\,25})^{\Sigma_4} = \langle [q_{s,\,3}] \rangle,$ with $q_{s,\,1} = \sum_{25\leq j\leq 28}a_{s,\,j}.$
\end{itemize}
\end{lema}

\begin{proof}

We frist prove Part $i).$ Since the set $\{[a_{s,\,j}]:\, 1\leq j\leq 12 \}$ is a basis of $\Sigma_4(a_{s,\,1}),$ if $[f]\in \Sigma_4(a_{s,\,1})^{\Sigma_4},$ then 
\begin{equation}\label{dt1}
\begin{array}{ll}
f &\equiv \big(\gamma_1a_{s,\,1} + \gamma_2a_{s,\,2} + \gamma_3a_{s,\,3} + \gamma_4a_{s,\,4} + \gamma_5a_{s,\,5} + \gamma_6a_{s,\,6}\\
&\quad \gamma_7a_{s,\,7} + \gamma_8a_{s,\,8} + \gamma_9a_{s,\,9} + \gamma_{10}a_{s,\,10} + \gamma_{11}a_{s,\,11} + \gamma_{12}a_{s,\,12}\big),
\end{array}
\end{equation} in which $\gamma_j\in k$ for every $j.$ Acting the homomorphisms $\sigma_i: (P_4)_{n_s}\to (P_4)_{n_s},\, i = 1, 2, 3,$ on both sides of \eqref{dt1}, we get
\begin{equation}\label{dt2}
\begin{array}{ll}
\sigma_1(f) &\equiv \big(\gamma_4a_{s,\,1} +  \gamma_5a_{s,\,2} +  \gamma_{10}a_{s,\,3} +  \gamma_{1}a_{s,\,4} +  \gamma_{2}a_{s,\,5} +  \gamma_{6}a_{s,\,6} +  \gamma_{7}a_{s,\,7}\\
\medskip
&\quad +  \gamma_{11}a_{s,\,8} +  \gamma_{12}a_{s,\,9} +  \gamma_{3}a_{s,\,10} +  \gamma_{8}a_{s,\,11} +  \gamma_{9}a_{s,\,12}\big);\\
\sigma_2(f) &\equiv \big(\gamma_{1}a_{s,\,1} + \gamma_{3}a_{s,\,2} + \gamma_{2}a_{s,\,3} + \gamma_{6}a_{s,\,4} + \gamma_{8}a_{s,\,5} + \gamma_{4}a_{s,\,6} + \gamma_{9}a_{s,\,7}\\
\medskip
&\quad + \gamma_{5}a_{s,\,8} + \gamma_{7}a_{s,\,9} + \gamma_{11}a_{s,\,10} + \gamma_{10}a_{s,\,11} + \gamma_{12}a_{s,\,12}\big).\\
\sigma_3(f) &\equiv \big(\gamma_{2}a_{s,\,1} + \gamma_{1}a_{s,\,2} + \gamma_{3}a_{s,\,3} + \gamma_{5}a_{s,\,4} + \gamma_{4}a_{s,\,5} + \gamma_{7}a_{s,\,6} + \gamma_{6}a_{s,\,7}\\
&\quad + \gamma_{9}a_{s,\,8} + \gamma_{8}a_{s,\,9} + \gamma_{10}a_{s,\,10} + \gamma_{12}a_{s,\,11} + \gamma_{11}a_{s,\,12}\big).
\end{array}
\end{equation}

Since $[f]\in \Sigma_4(a_{s,\,1})^{\Sigma_4},$ $\sigma_i(f)\equiv f$ for $1\leq i\leq 3.$ Combining this and the equalities \eqref{dt1} and \eqref{dt2}, we find that
$$ \begin{array}{ll}
\medskip
(\gamma_1 + \gamma_4)a_{s,\,1} + (\gamma_2 + \gamma_5)a_{s,\,2} + (\gamma_3 + \gamma_{10})a_{s,\,3} + (\gamma_8 + \gamma_{11})a_{s,\,8} + (\gamma_9 + \gamma_{12})a_{s,\,9}& \equiv 0,\\
\medskip
(\gamma_2 + \gamma_3)a_{s,\,2} + (\gamma_4 + \gamma_6)a_{s,\,4} + (\gamma_5 + \gamma_{8})a_{s,\,5} + (\gamma_7 + \gamma_{9})a_{s,\,7} + (\gamma_{10} + \gamma_{11})a_{s,\,10}& \equiv 0,\\
(\gamma_1 + \gamma_2)a_{s,\,1} + (\gamma_4 + \gamma_5)a_{s,\,4} + (\gamma_6 + \gamma_{7})a_{s,\,6} + (\gamma_8 + \gamma_{9})a_{s,\,8} + (\gamma_{11} + \gamma_{12})a_{s,\,11}& \equiv 0.
\end{array}$$
This implies that $\gamma_1  = \gamma_j$ for all $j,\, 2\leq j\leq 12.$ By a similar computation, we also get Part $iii).$

We now prove Part $ii).$ We see that a basis for $\Sigma_4(a_{s,\,13})$ is the set $\{[a_{s,\,j}]:\, 13\leq j\leq 24\}.$ Then, assume that $[g]\in \Sigma_4(a_{s,\,13})^{\Sigma_4},$ then we have
\begin{equation}\label{dt3}
g\equiv \sum_{13\leq j\leq 24}\beta_ja_{s,\,j},
\end{equation} with $\beta_j\in k,\, j = 13, \ldots, 24. $ Acting the homomorphisms $\sigma_i: (P_4)_{n_s}\to (P_4)_{n_s},\, 1\leq i\leq 3,$ on both sides of \eqref{dt3}, we get
\begin{equation}\label{dt4}
\begin{array}{ll}
\sigma_1(g) &\equiv \big(\beta_{16}a_{s,\,13} + \beta_{19}a_{s,\,14} +  \beta_{20}a_{s,\,15} +  \beta_{13}a_{s,\,16} +  \beta_{21}a_{s,\,17} \\
&\quad +  \beta_{22}a_{s,\,18} +  \beta_{14}a_{s,\,19} +  \beta_{15}a_{s,\,20} +  \beta_{17}a_{s,\,21}+  \beta_{18}a_{s,\,22}\\
\medskip
&\quad  +  \beta_{23}x_1^{3.2^s-1}x_2^{2^{s+1}-1}x_4^{2^s-1} +  \beta_{24}x_1^{3.2^s-1}x_2^{2^{s+1}-1}x_3^{2^s-1}\big);\\
\sigma_2(g) &\equiv \big( \beta_{14}a_{s,\,13} +  \beta_{13}a_{s,\,14} +  \beta_{17}a_{s,\,16} +  \beta_{16}a_{s,\,17} +  \beta_{21}a_{s,\,19}  \\
&\quad  +  \beta_{23}a_{s,\,20}+  \beta_{19}a_{s,\,21}+  \beta_{24}a_{s,\,22} +  \beta_{20}a_{s,\,23} +  \beta_{22}a_{s,\,24}\\
\medskip
&\quad  +  \beta_{15}x_2^{3.2^s-1}x_3^{2^{s+1}-1}x_4^{2^s-1} +   \beta_{18}x_1^{2^s-1}x_2^{3.2^s-1}x_3^{2^{s+1}-1}\big);\\
\sigma_3(g) &\equiv \big( \beta_{15}a_{s,\,14} +  \beta_{14}a_{s,\,15} +  \beta_{18}a_{s,\,17} +  \beta_{17}a_{s,\,18} +  \beta_{20}a_{s,\,19} \\
&\quad +  \beta_{19}a_{s,\,20}+   \beta_{22}a_{s,\,21}  +  \beta_{21}a_{s,\,22} +  \beta_{24}a_{s,\,23} +  \beta_{23}a_{s,\,24}\\
&\quad  +  \beta_{13}x_2^{2^s-1}x_3^{3.2^s-1}x_4^{2^{s+1}-1}+  \beta_{16}x_1^{2^s-1}x_3^{3.2^s-1}x_4^{2^{s+1}-1} \big).
\end{array}
\end{equation}
Using Cartan's formula and Theorem \ref{dlSi}, we have
\begin{equation}\label{dt5}
 \begin{array}{ll}
\medskip
x_1^{3.2^s-1}x_2^{2^{s+1}-1}x_4^{2^s-1} &= Sq^{2}(x_1^{3.2^s-3}x_2^{2^{s+1}-1}x_4^{2^s-1}) + a_{s,\,23} \mod (\widehat{A}(P_4)_{n_s});\\
\medskip
 x_1^{3.2^s-1}x_2^{2^{s+1}-1}x_3^{2^s-1} &= Sq^{2}(x_1^{3.2^s-3}x_2^{2^{s+1}-1}x_3^{2^s-1} ) + a_{s,\,24} \mod (\widehat{A}(P_4)_{n_s});\\
\medskip
x_2^{3.2^s-1}x_3^{2^{s+1}-1}x_4^{2^s-1} &= Sq^{2}(x_2^{3.2^s-3}x_3^{2^{s+1}-1}x_4^{2^s-1}) + a_{s,\,15} \mod (\widehat{A}(P_4)_{n_s});\\
\medskip
x_1^{2^s-1}x_2^{3.2^s-1}x_3^{2^{s+1}-1} &= Sq^{2}(x_1^{2^s-1}x_2^{3.2^s-3}x_3^{2^{s+1}-1}) + a_{s,\,18} \mod (\widehat{A}(P_4)_{n_s});\\
\medskip
x_2^{2^s-1}x_3^{3.2^s-1}x_4^{2^{s+1}-1} &= Sq^{2}(x_2^{2^s-1}x_3^{3.2^s-3}x_4^{2^{s+1}-1} ) + a_{s,\,13} \mod (\widehat{A}(P_4)_{n_s});\\
\medskip
x_1^{2^s-1}x_3^{3.2^s-1}x_4^{2^{s+1}-1} &= Sq^{2}(x_1^{2^s-1}x_3^{3.2^s-3}x_4^{2^{s+1}-1}) + a_{s,\,16} \mod (\widehat{A}(P_4)_{n_s}).
\end{array}
\end{equation}
Then, using the equalities \eqref{dt3}, \eqref{dt4} and \eqref{dt5} and the relations $\sigma_i(g)\equiv g,$ for $i\in \{1, 2, 3\},$ deduce that $\beta_{13} = \beta_j$ for all $j,\, 14\leq j\leq 24.$ The lemma is completely proved.
\end{proof}

\begin{proof}[{\it Proof of Theorem \ref{dlc1}}]
According to Sum \cite{N.S1}, a basis for $\overline{Q^{\otimes 4}_{n_s}}$ is the set consisting of all the classes represented monomials $a_{s, i}$ which are determined as follows:

For $s\geq 1,$
\begin{center}
\begin{tabular}{lll}
$a_{s,\,29}= x_1x_2^{2^s-1}x_3^{2^s-1}x_4^{2^{s+2}-2}$, & $a_{s,\,30}= x_1x_2^{2^s-1}x_3^{2^{s+2}-2}x_4^{2^s-1}$, & $a_{s,\,31}= x_1x_2^{2^{s+2}-2}x_3^{2^s-1}x_4^{2^s-1}$, \\
$a_{s,\,32}= x_1x_2^{2^s-1}x_3^{2^{s+1}-2}x_4^{3.2^{s}-1}$, & $a_{s,\,33}= x_1x_2^{2^{s+1}-2}x_3^{2^s-1}x_4^{3.2^{s}-1}$, & $a_{s,\,34}= x_1x_2^{2^{s+1}-2}x_3^{3.2^{s}-1}x_4^{2^s-1}$, \\
$a_{s,\,35}= x_1x_2^{2^{s+1}-2}x_3^{2^{s+1}-1}x_4^{2^{s+1}-1}$, & $a_{s,\,36}= x_1x_2^{2^{s+1}-1}x_3^{2^{s+1}-2}x_4^{2^{s+1}-1}$, & $a_{s,\,37}= x_1x_2^{2^{s+1}-1}x_3^{2^{s+1}-1}x_4^{2^{s+1}-2}$, \\
$a_{s,\,38}= x_1^{2^{s+1}-1}x_2x_3^{2^{s+1}-2}x_4^{2^{s+1}-1}$, & $a_{s,\,39}= x_1^{2^{s+1}-1}x_2x_3^{2^{s+1}-1}x_4^{2^{s+1}-2}$, & $a_{s,\,40}= x_1^{2^{s+1}-1}x_2^{2^{s+1}-1}x_3x_4^{2^{s+1}-2}$.\\
\end{tabular}%
\end{center}
For $s = 1,$
\begin{center}
\begin{tabular}{lll}
$a_{1,\,41}= x_1x_2x_3^{3}x_4^{4}$, & $a_{1,\,42}= x_1x_2^{3}x_3x_4^{4}$, & $a_{1,\,43}= x_1x_2^{3}x_3^{4}x_4$, \\
$a_{1,\,44}= x_1^{3}x_2x_3x_4^{4}$, & $a_{1,\,45}= x_1^{3}x_2x_3^{4}x_4$, & $a_{1,\,46}= x_1^{3}x_2^{4}x_3x_4$.\\
\end{tabular}%
\end{center}

For $s\geq 2,$
\begin{center}
\begin{tabular}{clrr}
$a_{s,\,41}= x_1x_2^{2^s-2}x_3^{2^s-1}x_4^{2^{s+2}-1}$, & $a_{s,\,42}= x_1x_2^{2^s-2}x_3^{2^{s+2}-1}x_4^{2^s-1}$, & \multicolumn{1}{l}{$a_{s,\,43}= x_1x_2^{2^s-1}x_3^{2^s-2}x_4^{2^{s+2}-1}$,} &  \\
$a_{s,\,44}= x_1x_2^{2^s-1}x_3^{2^{s+2}-1}x_4^{2^s-2}$, & $a_{s,\,45}= x_1x_2^{2^{s+2}-1}x_3^{2^s-2}x_4^{2^s-1}$, & \multicolumn{1}{l}{$a_{s,\,46}= x_1x_2^{2^{s+2}-1}x_3^{2^s-1}x_4^{2^s-2}$,} &  \\
$a_{s,\,47}= x_1^{2^s-1}x_2x_3^{2^s-2}x_4^{2^{s+2}-1}$, & $a_{s,\,48}= x_1^{2^s-1}x_2x_3^{2^{s+2}-1}x_4^{2^s-2}$, & \multicolumn{1}{l}{$a_{s,\,49}= x_1^{2^s-1}x_2^{2^{s+2}-1}x_3x_4^{2^s-2}$,} &  \\
$a_{s,\,50}= x_1^{2^{s+2}-1}x_2x_3^{2^s-2}x_4^{2^s-1}$, & $a_{s,\,51}= x_1^{2^{s+2}-1}x_2x_3^{2^s-1}x_4^{2^s-2}$, & \multicolumn{1}{l}{$a_{s,\,52}= x_1^{2^{s+2}-1}x_2^{2^s-1}x_3x_4^{2^s-2}$,} &  \\
$a_{s,\,53}= x_1x_2^{2^s-2}x_3^{2^{s+1}-1}x_4^{3.2^s-1}$, & $a_{s,\,54}= x_1x_2^{2^{s+1}-1}x_3^{2^s-2}x_4^{3.2^s-1}$, & \multicolumn{1}{l}{$a_{s,\,55}= x_1x_2^{2^{s+1}-1}x_3^{3.2^s-1}x_4^{2^s-2}$,} &  \\
$a_{s,\,56}= x_1^{2^{s+1}-1}x_2x_3^{2^s-2}x_4^{3.2^s-1}$, & $a_{s,\,57}= x_1^{2^{s+1}-1}x_2x_3^{3.2^s-1}x_4^{2^s-2}$, & \multicolumn{1}{l}{$a_{s,\,58}= x_1^{2^{s+1}-1}x_2^{3.2^s-1}x_3x_4^{2^s-2}$,} &  \\
$a_{s,\,59}= x_1^{2^s-1}x_2x_3^{2^s-1}x_4^{2^{s+2}-2}$, & $a_{s,\,60}= x_1^{2^s-1}x_2x_3^{2^{s+2}-2}x_4^{2^s-1}$, & \multicolumn{1}{l}{$a_{s,\,61}= x_1^{2^s-1}x_2^{2^s-1}x_3x_4^{2^{s+2}-2}$,} &  \\
$a_{s,\,62}= x_1^{2^s-1}x_2x_3^{2^{s+1}-2}x_4^{3.2^s-1}$, & $a_{s,\,63}= x_1x_2^{2^s-1}x_3^{2^{s+1}-1}x_4^{3.2^s-2}$, & \multicolumn{1}{l}{$a_{s,\,64}= x_1x_2^{2^{s+1}-1}x_3^{2^s-1}x_4^{3.2^s-2}$,} &  \\
$a_{s,\,65}= x_1x_2^{2^{s+1}-1}x_3^{3.2^s-2}x_4^{2^s-1}$, & $a_{s,\,66}= x_1^{2^s-1}x_2x_3^{2^{s+1}-1}x_4^{3.2^s-2}$, & \multicolumn{1}{l}{$a_{s,\,67}= x_1^{2^s-1}x_2^{2^{s+1}-1}x_3x_4^{3.2^s-2}$,} &  \\
$a_{s,\,68}= x_1^{2^{s+1}-1}x_2x_3^{2^s-1}x_4^{3.2^s-2}$, & $a_{s,\,69}= x_1^{2^{s+1}-1}x_2x_3^{3.2^s-2}x_4^{2^s-1}$, & \multicolumn{1}{l}{$a_{s,\,70}= x_1^{2^{s+1}-1}x_2^{2^s-1}x_3x_4^{3.2^s-2}$,} &  \\
$a_{s,\,71}= x_1^{3}x_2^{2^s-1}x_3^{2^{s+2}-3}x_4^{2^s-2}$, & $a_{s,\,72}= x_1^{3}x_2^{2^{s+2}-3}x_3^{2^s-2}x_4^{2^s-1}$, & \multicolumn{1}{l}{$a_{s,\,73}= x_1^{3}x_2^{2^{s+2}-3}x_3^{2^s-1}x_4^{2^s-2}$,} &  \\
$a_{s,\,74}= x_1^{3}x_2^{2^{s+1}-3}x_3^{2^s-2}x_4^{3.2^s-1}$, & $a_{s,\,75}= x_1^{3}x_2^{2^{s+1}-3}x_3^{3.2^s-1}x_4^{2^s-2}$, & \multicolumn{1}{l}{$a_{s,\,76}= x_1^{3}x_2^{2^{s+1}-1}x_3^{3.2^s-3}x_4^{2^s-2}$,} &  \\
$a_{s,\,77}= x_1^{2^{s+1}-1}x_2^{3}x_3^{3.2^s-3}x_4^{2^s-2}$, & $a_{s,\,78}= x_1^{3}x_2^{2^s-1}x_3^{2^{s+1}-3}x_4^{3.2^s-2}$, & \multicolumn{1}{l}{$a_{s,\,79}= x_1^{3}x_2^{2^{s+1}-3}x_3^{2^s-1}x_4^{3.2^s-2}$,} &  \\
$a_{s,\,80}= x_1^{3}x_2^{2^{s+1}-3}x_3^{3.2^s-2}x_4^{2^s-1}$, & $a_{s,\,81}= x_1^{3}x_2^{2^{s+1}-3}x_3^{2^{s+1}-2}x_4^{2^{s+1}-1}$, & \multicolumn{1}{l}{$a_{s,\,82}= x_1^{3}x_2^{2^{s+1}-3}x_3^{2^{s+1}-1}x_4^{2^{s+1}-2}$,} &  \\
$a_{s,\,83}= x_1^{3}x_2^{2^{s+1}-1}x_3^{2^{s+1}-3}x_4^{2^{s+1}-2}$, & $a_{s,\,84}= x_1^{2^{s+1}-1}x_2^{3}x_3^{2^{s+1}-3}x_4^{2^{s+1}-2}$.\\
&       &  
\end{tabular}
\end{center}

For $s = 2,$
\begin{center}
\begin{tabular}{lrr}
$a_{2,\,85}= x_1^{3}x_2^{3}x_3^{3}x_4^{12}$, & \multicolumn{1}{l}{$a_{2,\,86}= x_1^{3}x_2^{3}x_3^{12}x_4^{3}$,} & \multicolumn{1}{l}{$a_{2,\,87}= x_1^{7}x_2^{9}x_3^{2}x_4^{3}$,} \\
$a_{2,\,88}= x_1^{7}x_2^{9}x_3^{3}x_4^{2}$, & \multicolumn{1}{l}{$a_{2,\,89}= x_1^{3}x_2^{3}x_3^{4}x_4^{11}$,} & \multicolumn{1}{l}{$a_{2,\,90}= x_1^{3}x_2^{3}x_3^{7}x_4^{8}$,} \\
$a_{2,\,91}= x_1^{3}x_2^{7}x_3^{3}x_4^{8}$, & \multicolumn{1}{l}{$a_{2,\,92}= x_1^{3}x_2^{7}x_3^{8}x_4^{3}$,} & \multicolumn{1}{l}{$a_{2,\,93}= x_1^{7}x_2^{3}x_3^{3}x_4^{8}$,} \\
$a_{2,\,94}= x_1^{7}x_2^{3}x_3^{8}x_4^{3}$.\\ &       &  
\end{tabular}%
\end{center}

\newpage
For $s\geq 3,$
\begin{center}
\begin{tabular}{lcl}
$a_{s,\,85}= x_1^{3}x_2^{2^s-3}x_3^{2^s-2}x_4^{2^{s+2}-1}$, & $a_{s,\,86}= x_1^{3}x_2^{2^s-3}x_3^{2^{s+2}-1}x_4^{2^s-2}$, & $a_{s,\,87}= x_1^{3}x_2^{2^{s+2}-1}x_3^{2^s-3}x_4^{2^s-2}$, \\
$a_{s,\,88}= x_1^{2^{s+2}-1}x_2^{3}x_3^{2^s-3}x_4^{2^s-2}$, & $a_{s,\,89}= x_1^{3}x_2^{2^s-3}x_3^{2^s-1}x_4^{2^{s+2}-2}$, & $a_{s,\,90}= x_1^{3}x_2^{2^s-3}x_3^{2^{s+2}-2}x_4^{2^s-1}$, \\
$a_{s,\,91}= x_1^{3}x_2^{2^s-1}x_3^{2^s-3}x_4^{2^{s+2}-2}$, & $a_{s,\,92}= x_1^{2^s-1}x_2^{3}x_3^{2^s-3}x_4^{2^{s+2}-2}$, & $a_{s,\,93}= x_1^{2^s-1}x_2^{3}x_3^{2^{s+2}-3}x_4^{2^s-2}$, \\
$a_{s,\,94}= x_1^{3}x_2^{2^s-3}x_3^{2^{s+1}-2}x_4^{3.2^s-1}$, & $a_{s,\,95}= x_1^{3}x_2^{2^s-3}x_3^{2^{s+1}-1}x_4^{3.2^s-2}$, & $a_{s,\,96}= x_1^{3}x_2^{2^{s+1}-1}x_3^{2^s-3}x_4^{3.2^s-2}$, \\
$a_{s,\,97}= x_1^{2^{s+1}-1}x_2^{3}x_3^{2^s-3}x_4^{3.2^s-2}$, & $a_{s,\,98}= x_1^{2^s-1}x_2^{3}x_3^{2^{s+1}-3}x_4^{3.2^s-2}$, & $a_{s,\,99}= x_1^{7}x_2^{2^{s+2}-5}x_3^{2^s-3}x_4^{2^s-2}$, \\
$a_{s,\,100}= x_1^{7}x_2^{2^{s+1}-5}x_3^{2^s-3}x_4^{3.2^s-2}$, & $a_{s,\,101}= x_1^{7}x_2^{2^{s+1}-5}x_3^{3.2^s-3}x_4^{2^s-2}$, & $a_{s,\,102}= x_1^{7}x_2^{2^{s+1}-5}x_3^{2^{s+1}-3}x_4^{2^{s+1}-2}$.\\
\end{tabular}%
\end{center}
For $s = 3,$
\begin{center}
\begin{tabular}{lcl}
$a_{3,\,103}= x_1^{7}x_2^{7}x_3^{7}x_4^{24}$, & $a_{3,\,104}= x_1^{7}x_2^{7}x_3^{9}x_4^{22}$, & $a_{3,\,105}= x_1^{7}x_2^{7}x_3^{25}x_4^{6}$.\\
\end{tabular}%
\end{center}

For $s\geq 4,$
\begin{center}
\begin{tabular}{lcl}
$a_{s,\,103}= x_1^{7}x_2^{2^s-5}x_3^{2^s-3}x_4^{2^{s+2}-2}$, & $a_{s,\,104}= x_1^{7}x_2^{2^s-5}x_3^{2^{s+1}-3}x_4^{3.2^s-2}$, & $a_{s,\,105}= x_1^{7}x_2^{2^s-5}x_3^{2^{s+2}-3}x_4^{2^s-2}$.
\end{tabular}%
\end{center}

We now consider the following cases.

\underline{\it Case $s = 1$}. We have a direct summand decomposition of the $\Sigma_4$-modules:
$$ \overline{Q^{\otimes 4}_{n_1}} = \Sigma_4(a_{1,\,29}, a_{1,\,32}) \bigoplus \Sigma_4(a_{1,\,35}),$$
where $ \Sigma_4(a_{1,\,29}; a_{1,\,32}) = \langle \{[a_{1,\,j}]:\, j = 29, 30, \ldots, 34, 41, 42, \ldots, 46\} \rangle,$ and 
$$\Sigma_4(a_{1,\,35}) = \langle \{[a_{1,\,j}]:\, 35\leq j\leq 40\} \rangle.$$

\begin{lema}\label{bd2}
The following statements are true:
\begin{itemize}
\item[i)]  $\Sigma_4(a_{1,\,29}; a_{1,\,32})^{\Sigma_4} = \langle [q_{1,\,4}] \rangle,$ with $q_{1,\,4} = \sum_{29\leq j\leq 31}a_{1,\,j} + \sum_{44\leq j\leq 46}a_{1,\,j};$
\item[ii)]  $\Sigma_4(a_{1,\,35})^{\Sigma_4} = 0.$
\end{itemize}
\end{lema}

Combining Lemmas \ref{bd1} and \ref{bd2} gives

\begin{hq}\label{hq1}
The space of $\Sigma_4$-invariants $(Q^{\otimes 4}_{n_1})^{\Sigma_4}$ is generated by the classes $[q_{1,\,t}]$ for all $t,\, 1\leq t\leq 4.$
\end{hq}

\begin{proof}[{\it Proof of Lemma \ref{bd2}}]

The proof of $ii)$ is straightforward. We prove $i)$ in detail. Suppose that $[g]\in \Sigma_4(a_{1,\,29}; a_{1,\,32})^{\Sigma_4}.$ Then, we have 
\begin{equation}\label{dt6}
g\equiv \big(\sum_{29\leq j\leq 34}\gamma_ja_{1,\,j} + \sum_{41\leq j\leq 46}\gamma_ja_{1,\,j}\big),
\end{equation}
with $\gamma_j\in k.$ Acting the homomorphisms $\sigma_i: (P_4)_{n_1}\to (P_4)_{n_1},\, 1\leq i\leq 3,$ on both sides of \eqref{dt6}, we get
\begin{equation}\label{dt7}
\begin{array}{ll}
\sigma_1(g)&\equiv \big(\gamma_{29}a_{1,\,29} + \gamma_{30}a_{1,\,30} + \gamma_{32}a_{1,\,32} + \gamma_{41}a_{1,\,41} + \gamma_{44}a_{1,\,42}\\
    &\quad +\gamma_{45}a_{1,\,43}  + \gamma_{42}a_{1,\,44} + \gamma_{43}a_{1,\,45} + \gamma_{31}x_1^{6}x_2x_3x_4\\   
\medskip
&\quad + \gamma_{33}x_1^{2}x_2x_3x_4^{5} + \gamma_{34}x_1^{2}x_2x_3^{5}x_4 + \gamma_{46}x_1^{4}x_2^{3}x_3x_4\big)\\
\sigma_2(g)&\equiv \big(\gamma_{29}a_{1,\,29} + \gamma_{31}a_{1,\,30} + \gamma_{30}a_{1,\,31} + \gamma_{33}a_{1,\,32} + \gamma_{42}a_{1,\,41}+\gamma_{41}a_{1,\,42} \\
\medskip
&\quad  + \gamma_{44}a_{1,\,44}+ \gamma_{46}a_{1,\,45} + \gamma_{45}a_{1,\,46} + \gamma_{34}x_1x_2^{5}x_3^{2}x_4 + \gamma_{43}x_1x_2^{4}x_3^{3}x_4\big)\\
\sigma_3(g)&\equiv \big(\gamma_{30}a_{1,\,29} + \gamma_{29}a_{1,\,30} + \gamma_{31}a_{1,\,31} + \gamma_{34}a_{1,\,43} + \gamma_{33}a_{1,\,34} + \gamma_{43}a_{1,\,42}+\gamma_{42}a_{1,\,43}\\
&\quad  + \gamma_{45}a_{1,\,44} + \gamma_{44}a_{1,\,45} + \gamma_{46}a_{1,\,46} + \gamma_{32}x_1x_2x_3^{5}x_4^{2} + \gamma_{41}x_1x_2x_3^{4}x_4^{3}\big).
\end{array}
\end{equation}
Using the Cartan formula and Theorem \ref{dlSi}, we have
\begin{equation}\label{dt8}
\begin{array}{ll}
x_1^{6}x_2x_3x_4 &= Sq^{1}(x_1^{5}x_2x_3x_4) + Sq^{2}(x_1^{3}x_2^{2}x_3x_4 + x_1^{3}x_2x_3^{2}x_4 + x_1^{3}x_2x_3x_4^{2})\\
\medskip
&\quad  +  a_{1,\,44} +  a_{1,\,45} +  a_{1,\,46} \mod (\widehat{A}(P_4)_{n_1});  \\
\medskip
x_1^{2}x_2x_3x_4^{5} &= Sq^{1}(x_1x_2x_3x_4^{5}) + a_{1,\,29} +  a_{1,\,32} +  a_{1,\,33} \mod (\widehat{A}(P_4)_{n_1});  \\
\medskip
x_1^{2}x_2x_3^{5}x_4 &= Sq^{1}(x_1x_2x_3^{5}x_4) + Sq^{2}(x_1x_2x_3^{3}x_4^{2}) + a_{1,\,30} +  a_{1,\,34} +  a_{1,\,41} \mod (\widehat{A}(P_4)_{n_1});  \\
x_1^{4}x_2^{3}x_3x_4 &= Sq^{1}(x_1x_2^{5}x_3x_4) + Sq^{2}(x_1x_2^{3}x_3x_4^{2} + x_1x_2^{3}x_3^{2}x_4 + x_1^{2}x_2^{3}x_3x_4)\\
\medskip
&\quad  + a_{1,\,31} +  a_{1,\,42} +  a_{1,\,43} \mod (\widehat{A}(P_4)_{n_1});  \\
\medskip
x_1x_2^{5}x_3^{2}x_4 &= Sq^{2}(x_1x_2^{3}x_3^{2}x_4) +  a_{1,\,43} \mod (\widehat{A}(P_4)_{n_1});  \\
\medskip
x_1x_2^{4}x_3^{3}x_4 &=Sq^{2}(x_1x_2^{2}x_3^{3}x_4) +  a_{1,\,34} \mod (\widehat{A}(P_4)_{n_1});  \\
\medskip
x_1x_2x_3^{5}x_4^{2} &=Sq^{2}(x_1x_2x_3^{3}x_4^{2}) +  a_{1,\,41} \mod (\widehat{A}(P_4)_{n_1});  \\
\medskip
x_1x_2x_3^{4}x_4^{3} &=Sq^{2}(x_1x_2x_3^{2}x_4^{3}) +  a_{1,\,32} \mod (\widehat{A}(P_4)_{n_1}).
\end{array}
\end{equation}
Combining \eqref{dt6}, \eqref{dt7} and \eqref{dt8}, we get
$$ \begin{array}{ll}
(\sigma_1(g) + g) &\equiv \big( \gamma_{33}(a_{1,\,29} +a_{1,\,32}) +  \gamma_{34}(a_{1,\,30} +a_{1,\,41})+ (\gamma_{31} + \gamma_{46})a_{1,\,31}+(\gamma_{42} + \gamma_{44} + \gamma_{46})a_{1,\,42} \\
\medskip
&\quad + (\gamma_{43} + \gamma_{45} + \gamma_{46})a_{1,\,43} + (\gamma_{31} + \gamma_{42} + \gamma_{44})a_{1,\,44} + (\gamma_{31} + \gamma_{43} + \gamma_{45})a_{1,\,45}\big),\\
(\sigma_2(g) + g) &\equiv \big((\gamma_{30} + \gamma_{31})a_{1,\,30}+ (\gamma_{32} + \gamma_{33})a_{1,\,32} + (\gamma_{34} + \gamma_{43})a_{1,\,34}\\
\medskip
&\quad + (\gamma_{41} + \gamma_{42})a_{1,\,41} + (\gamma_{45} + \gamma_{46})a_{1,\,45}\big)\\
(\sigma_3(g) + g) &\equiv \big((\gamma_{29} + \gamma_{30})a_{1,\,29} + (\gamma_{32} + \gamma_{41})a_{1,\,32} + (\gamma_{33} + \gamma_{34})a_{1,\,33} \\
\medskip
&\quad +(\gamma_{42} + \gamma_{43})a_{1,\,42} + (\gamma_{44} + \gamma_{45})a_{1,\,44}\big).
\end{array}$$
From these equalities and the relations $\sigma_j(g) \equiv g,\, 1\leq j\leq 3,$ we obtain $\gamma_j = 0$ for $j\in \{32, 33, 34, 41, 42, 43\},$ and $\gamma_{29} = \gamma_{30} = \gamma_{31} = \gamma_{44} = \gamma_{45} = \gamma_{46}.$ This completes the proof of the lemma.
\end{proof}

The following assertion is useful.

\begin{md}\label{md1}
We have $(Q^{\otimes 4}_{n_1})^{GL_4(k)} = \langle [q_{1,\, 4}] \rangle.$
\end{md}

\begin{proof}
Let $h\in (P_4)_{n_1}$ such that $[h]\in (Q^{\otimes 4}_{n_1})^{GL_4(k)}.$ Since $\Sigma_4\subset GL_4(k),$ $[h]\in (Q^{\otimes 4}_{n_1})^{\Sigma_4}$ and therefore by Corollary \ref{hq1}, we have $ h\equiv \sum_{1\leq t\leq 4}m_tq_{1,\, t},$ where $m_t\in k$ for every $t.$ Since $[h]\in (Q^{\otimes 4}_{n_1})^{GL_4(k)},$ $(\sigma_4(h) + h)\equiv 0$ with $\sigma_4: (P_4)_{n_1}\to (P_4)_{n_1}.$ From Theorem \ref{dlSi} and a direct computation, we obtain
$$ (\sigma_4(h) + h) \equiv (m_1a_{1,\, 1} + (m_1 + m_2)a_{1,\, 8} + (m_2 + m_3)a_{1,\, 17} + \mbox{other terms})\equiv 0.$$
and therefore $h\equiv q_{1,\, 4}.$ The proposition is proved.
\end{proof}

Now, we consider the element
$$ \zeta = (a_1^{(1)}a_2^{(3)}a_3^{(3)}a_4^{(2)} + a_1^{(1)}a_2^{(3)}a_3^{(4)}a_4^{(1)} + a_1^{(1)}a_2^{(5)}a_3^{(2)}a_4^{(1)} +  a_1^{(1)}a_2^{(6)}a_3^{(1)}a_4^{(1)})\in (P_4)_{3(2^{1}-1) + 3.2^{1}}^{*}.$$
Then, it is $\widehat{A}$-annihilated. Indeed, by the unstable condition, we need only to consider the effects of $Sq^{1}$ and $Sq^{2}$. By a direct computation, we get
$$ \begin{array}{ll}
\medskip
(\zeta)Sq^{1} &= a_1^{(1)}a_2^{(3)}a_3^{(3)}a_4^{(1)} + a_1^{(1)}a_2^{(3)}a_3^{(3)}a_4^{(1)} + a_1^{(1)}a_2^{(5)}a_3^{(1)}a_4^{(1)} +a_1^{(1)}a_2^{(5)}a_3^{(1)}a_4^{(1)} = 0,\\
\medskip
(\zeta)Sq^{2} &=  0 + a_1^{(1)}a_2^{(3)}a_3^{(2)}a_4^{(1)} + a_1^{(1)}a_2^{(3)}a_3^{(2)}a_4^{(1)} + 0 = 0.
\end{array}$$
On the other side, since $\langle [q_{1,\, 4}], [\zeta] \rangle = 1,$ by Proposition \ref{md1}, implies that $$ k\otimes_{GL_4(k)}P((P_4)_{n_1}^{*}) = ((Q^{\otimes 4}_{n_1})^{GL_4(k)})^{*} = \langle ([q_{1,\, 4}])^{*}\rangle  = \langle[\zeta]\rangle.$$

\underline{Case $s = 2$}. We first have the following remark.

\begin{rem} Based on the bases of $\underline{Q^{\otimes 4}_{n_s}}$ and $\overline{Q^{\otimes 4}_{n_s}}$ above, the weight vector of $a_{s,\, j}$ is one of the following sequences:
$$\omega_{(s, 1)}:=   \underset{\mbox{{$s$ times of $3$}}}{\underbrace{(3, 3, \ldots, 3}}, 1, 1),\ \ \omega_{(s, 2)}:= \underset{\mbox{{$s+1$ times of $3$}}}{\underbrace{(3, 3, \ldots, 3)}}.$$ Hence, we have the following isomorphisms:
$$ \begin{array}{ll}
\medskip
Q^{\otimes 4}_{n_s} &\cong (Q^{\otimes 4}_{n_s})^{\omega_{(s, 1)}} \bigoplus (Q^{\otimes 4}_{n_s})^{\omega_{(s, 2)}},\\
(Q^{\otimes 4}_{n_s})^{\omega_{(s, j)}} &\cong \underline{(Q^{\otimes 4}_{n_s})}^{\omega_{(s, j)}}\bigoplus \overline{(Q^{\otimes 4}_{n_s})}^{\omega_{(s, j)}},\ j = 1,\, 2,
\end{array}$$
where $$ \begin{array}{ll}
\medskip
\underline{(Q^{\otimes 4}_{n_s})}^{\omega_{(s, j)}} &:= \langle \{[x = \prod_{1\leq i\leq 4}x_i^{a_i}]\in Q^{\otimes 4}_{n_s}:\,\prod_{1\leq i\leq 4}a_i = 0, \omega(x) = \omega_{(s, j)}\} \rangle,\\
\overline{(Q^{\otimes 4}_{n_s})}^{\omega_{(s, j)}} &:= \langle\{ [x = \prod_{1\leq i\leq 4}x_i^{a_i}]\in Q^{\otimes 4}_{n_s}:\,\prod_{1\leq i\leq 4}a_i >0 , \omega(x) = \omega_{(s, j)} \}\rangle.
\end{array}$$
\end{rem}

Then, for $s = 2,$ we have $Q^{\otimes 4}_{n_2} \cong (Q^{\otimes 4}_{n_2})^{\omega_{(2, 1)}} \bigoplus (Q^{\otimes 4}_{n_2})^{\omega_{(2, 2)}}$ and the following.

\begin{md}\label{md2}
The spaces of invariants $((Q^{\otimes 4}_{n_2})^{\omega_{(2, 1)}})^{GL_4(k)}$ and $((Q^{\otimes 4}_{n_2})^{\omega_{(2, 2)}})^{GL_4(k)}$ are trivial.
\end{md}

\begin{proof}
We first prove $((Q^{\otimes 4}_{n_2})^{\omega_{(2, 1)}})^{GL_4(k)} = 0.$ By Definition \ref{dnspike}, the monomial $a_{s,\, 12} = x_1^{2^{s+2}-1}x_2^{2^s-1}x_3^{2^s-1}$ is minimal spike in $(P_4)_{n_s}$ and $\omega(a_{s,\, 12}) = \omega_{(s, 1)}$ for any $s\geq 1.$ Hence, $[a_{s,\, j}]_{\omega_{(s, 1)}} = [a_{s,\, j}].$  Recall that $(Q^{\otimes 4}_{n_2})^{\omega_{(2, 1)}} \cong \underline{(Q^{\otimes 4}_{n_2})}^{\omega_{(2, 1)}}\bigoplus \overline{(Q^{\otimes 4}_{n_2})}^{\omega_{(2, 1)}}.$ Then, by Lemma \ref{bd1}, we get
$\underline{(Q^{\otimes 4}_{n_2})}^{\omega_{(2, 1)}}  =  \Sigma_4(a_{2,\,1}) \bigoplus \Sigma_4(a_{2,\,13}),$ and
\begin{equation}\label{dt9}
(\underline{(Q^{\otimes 4}_{n_2})}^{\omega_{(2, 1)}})^{\Sigma_4} = \langle [q_{2,\, 1}], [q_{2,\, 2}]\rangle.
\end{equation}

Using an admissible basis of $\overline{Q^{\otimes 4}_{n_2}}$ above, we have a direct summand decomposition of the $\Sigma_4$-modules:
$$  \overline{(Q^{\otimes 4}_{n_2})}^{\omega_{(2, 1)}} = \Sigma_4(a_{2,\,41}) \bigoplus \Sigma_4(a_{2,\,53}) \bigoplus \mathbb V,$$
where $$ \begin{array}{ll}
\medskip
& \Sigma_4(a_{2,\,41}) = \langle \{[a_{2,\,j}]:\, 41\leq j\leq 52\} \rangle,\ \Sigma_4(a_{2,\,53}) = \langle \{[a_{2,\,j}]:\, 53\leq j\leq 58\} \rangle,\\
\medskip
&\mathbb V = \langle \{[a_{2,\,j}]:\, j = 29, \ldots, 34, 59, \ldots, 80, 85, \ldots, 94\} \rangle.
\end{array}$$
Then, by a similar technique as in the proof of Lemma \ref{bd2}, we find that
\begin{equation}\label{dt10}
\begin{array}{ll}
\medskip
&\Sigma_4(a_{2,\,41})^{\Sigma_4} = \langle [\widehat{q_{2,\, 1}}] \rangle\ \mbox{with}\ \widehat{q_{2,\, 1}}:= \sum_{41\leq j\leq 52}a_{2,\,j};\\
\medskip
&\Sigma_4(a_{2,\,53})^{\Sigma_4} = \langle [\widehat{q_{2,\, 2}}] \rangle\ \mbox{with}\ \widehat{q_{2,\, 2}}:= \sum_{53\leq j\leq 58}a_{2,\,j};\\
\medskip
&\mathbb V^{\Sigma_4} = 0.
\end{array}
\end{equation}
Thus, from the equalities \eqref{dt9} and \eqref{dt10}, we get 
\begin{equation}\label{dt11}
((Q^{\otimes 4}_{n_2})^{\omega_{(2, 1)}})^{\Sigma_4} = \langle [q_{2,\, 1}], [q_{2,\, 2}], [\widehat{q_{2,\, 1}}], [\widehat{q_{2,\, 2}}] \rangle.
\end{equation}

Now, let $h\in (P_4)_{n_2}$ such that $[h]\in ((Q^{\otimes 4}_{n_2})^{\omega_{(2, 1)}})^{GL_4(k)}.$ Since  $[h]\in ((Q^{\otimes 4}_{n_2})^{\omega_{(2, 1)}})^{\Sigma_4},$ by the equality \eqref{dt11}, we have 
\begin{equation}\label{dt12}
h\equiv (\xi_1q_{2,\, 1} + \xi_2q_{2,\, 2} + \xi_3\widehat{q_{2,\, 1}} + \xi_4\widehat{q_{2,\, 2}}),
\end{equation}
where $\xi_t\in k$ for $1\leq t\leq 4.$ Acting the homomorphism $\sigma_4: (P_4)_{n_2}\to (P_4)_{n_2}$ on both sides of \eqref{dt12}, and using Theorem \ref{dlSi} with the relation $(\sigma_4(h) + h)\equiv 0,$ we obtain
$$ (\sigma_4(h) + h) \equiv ((\xi_1 + \xi_3)a_{2,\, 1} + \xi_1a_{2,\, 3} + (\xi_1 + \xi_2)a_{2,\, 8} + (\xi_2 + \xi_4)a_{2,\, 13} + \mbox{other terms})  \equiv 0.$$
The last equality implies $\xi_1 = \xi_2 = \xi_3 = \xi_4.$ This means that $((Q^{\otimes 4}_{n_2})^{\omega_{(2, 1)}})^{GL_4(k)} = 0.$ 

Next, we have $(Q^{\otimes 4}_{n_2})^{\omega_{(2, 2)}} \cong \underline{(Q^{\otimes 4}_{n_2})}^{\omega_{(2, 2)}}\bigoplus \overline{(Q^{\otimes 4}_{n_2})}^{\omega_{(2, 2)}}.$ According to Lemma \ref{bd1}, we get $$(\underline{(Q^{\otimes 4}_{n_2})}^{\omega_{(2, 2)}})^{\Sigma_4} = \langle [q_{2,\, 3}]_{\omega_{(2, 2)}}\rangle.$$
A basis of $\overline{(Q^{\otimes 4}_{n_2})}^{\omega_{(2, 2)}}$ is the set $\{[a_{2,\,j}]_{\omega_{(2, 2)}}:\, j = 35, 36, \ldots, 40, 81, \ldots, 84\}$ and we have a direct summand decomposition of the $\Sigma_4$-modules:
$$  \overline{(Q^{\otimes 4}_{n_2})}^{\omega_{(2, 2)}}  = \Sigma_4(a_{2,\,35})\bigoplus \Sigma_4(a_{2,\,81}),$$
where $\Sigma_4(a_{2,\,35}) = \langle \{[[a_{2,\,j}]_{\omega_{(2, 2)}}:\, j = 35, \ldots, 40]\} \rangle$ and $\Sigma_4(a_{2,\,81}) = \langle \{[[a_{2,\,j}]_{\omega_{(2, 2)}}:\, j = 81, \ldots, 84]\} \rangle.$ By using Theorem \ref{dlSi} and a similar computation as above, we also obtain 
$$  \Sigma_4(a_{2,\,35})^{\Sigma_4} = \langle [\widehat{q_{2,\, 3}}]_{\omega_{(2,\, 2)}}\rangle,\ \  \Sigma_4(a_{2,\,81})^{\Sigma_4} = \langle [\widehat{q_{2,\, 4}}]_{\omega_{(2,\, 2)}} \rangle,$$
where $\widehat{q_{2,\, 3}} = \sum_{35\leq j\leq 40}a_{2,\, j}$ and $\widehat{q_{2,\, 4}} = \sum_{81\leq j\leq 84}a_{2,\, j}.$ Then, $$ ((Q^{\otimes 4}_{n_2})^{\omega_{(2, 2)}})^{\Sigma_4} = \langle [q_{2,\, 3}]_{\omega_{(2,\, 2)}}, \langle [\widehat{q_{2,\, 3}}]_{\omega_{(2,\, 2)}}, \langle [\widehat{q_{2,\, 4}}]_{\omega_{(2,\, 2)}} \rangle.$$
Let $[g]_{_{\omega_{(2,\, 2)}}}\in ((Q^{\otimes 4}_{n_2})^{\omega_{(2, 2)}})^{GL_4(k)},$ we have $g\equiv_{_{\omega_{(2,\, 2)}}} (\ell_1q_{2,\, 3} + \ell_2\widehat{q_{2,\, 3}} + \ell_3\widehat{q_{2,\, 4}})$ with $\ell_i\in k$ for every $i.$ By a direct computation using the relation $(\sigma_4(g) + g)\equiv_{\omega_{(2,\, 2)}} 0,$ we obtain
$$  (\sigma_4(g) + g)\equiv_{_{\omega_{(2,\, 2)}}} ( (\ell_1 + \ell_2)a_{2,\, 25} + (\ell_2 + \ell_3)a_{2,\, 36} + \ell_3a_{2,\, 83}+ \mbox{other terms})  \equiv_{\omega_{(2,\, 2)}} 0.$$
This means that $\ell_1 = \ell_2 = \ell_3 = 0.$ The proposition follows.
\end{proof}

Now, since $\dim (Q^{\otimes 4}_{n_2})^{GL_4(k)} \leq \dim ((Q^{\otimes 4}_{n_2})^{\omega_{(2, 1)}})^{GL_4(k)} +\dim ((Q^{\otimes 4}_{n_2})^{\omega_{(2, 2)}})^{GL_4(k)},$ by Proposition \ref{md2},  we get
$$ k\otimes_{GL_4(k)}P((P_4)_{n_2}^{*}) = ((Q^{\otimes 4}_{n_2})^{GL_4(k)})^{*}  = 0.$$

\underline{Case $s \geq 3$}. For simplicity, we prove the cases $s\geq 4.$ Calculations of the case $s = 3$ use similar techniques. Recall that 
$$ 
Q^{\otimes 4}_{n_s} \cong (Q^{\otimes 4}_{n_s})^{\omega_{(s, 1)}} \bigoplus (Q^{\otimes 4}_{n_s})^{\omega_{(s, 2)}},
$$
where $\dim (Q^{\otimes 4}_{n_s})^{\omega_{(s, 1)}} = 90$ and $\dim (Q^{\otimes 4}_{n_s})^{\omega_{(s, 2)}} = 15$ for all $s\geq 4.$ 

\begin{md}\label{md3}
The following hold:
\begin{itemize}
\item [i)] $((Q^{\otimes 4}_{n_s})^{\omega_{(s, 1)}})^{GL_4(k)}$ is trivial.

\item[ii)] $((Q^{\otimes 4}_{n_s})^{\omega_{(s, 2)}})^{GL_4(k)}  = \langle [\zeta_s^{*}]_{\omega_{(s, 2)}} \rangle,$
where $$\zeta_s^{*} = \sum_{25\leq j\leq 28}a_{s,\, j} +  \sum_{35\leq j\leq 40}a_{s,\, j} + \sum_{81\leq j\leq 84}a_{s,\, j} + a_{s,\, 102}$$
\end{itemize}
\end{md}

\begin{proof}[{\it Outline of the proof}]

We will give the sketch of proof of ii). The proofs of i) use similar idea. 

We have an isomorphism of the $k$-vector spaces
$$ (Q^{\otimes 4}_{n_s})^{\omega_{(s, 2)}}  \cong \underline{(Q^{\otimes 4}_{n_s}})^{\omega_{(s, 2)}} \bigoplus \overline{(Q^{\otimes 4}_{n_s}})^{\omega_{(s, 2)}},$$
where $$ \begin{array}{ll}
\medskip
\underline{(Q^{\otimes 4}_{n_s}})^{\omega_{(s, 2)}} &= \langle \{[a_{s,\, j}]_{\omega_{(s, 2)}}:\, 25\leq j\leq 28\}\rangle,\\
 \overline{(Q^{\otimes 4}_{n_s}})^{\omega_{(s, 2)}} &= \langle \{[a_{s,\, j}]_{\omega_{(s, 2)}}:\, j = 35, \ldots, 40, 81, \ldots, 84, 102\} \rangle.
\end{array}
$$
Then, by Lemma \ref{bd1}, $$ (\underline{(Q^{\otimes 4}_{n_s}})^{\omega_{(s, 2)}})^{\Sigma_4} = \langle  [q_{s,\, 3}]_{\omega_{(s, 2)}} \rangle.$$
Now, if $[f]\in (\overline{(Q^{\otimes 4}_{n_s}})^{\omega_{(s, 2)}})^{\Sigma_4},$ then we have $$f\equiv_{\omega_{(s, 2)}} \big(\sum_{35\leq j\leq 40} \gamma_ja_{s,\, j} + \sum_{81\leq j\leq 84} \gamma_ja_{s,\, j} + \gamma_{102}a_{s,\, 102}\big)$$ in which $\gamma_j\in k.$ So, by a simple computation using the relations $\sigma_i(f)\equiv_{\omega_{(s, 2)}} f$ for $i\in \{1, 2, 3\},$ where $\sigma_i: (P_4)_{n_s}\to (P_4)_{n_s},$ we get $\gamma_{35} = \gamma_{36} = \cdots = \gamma_{40},$ and $\gamma_{81} = \gamma_{82} = \gamma_{83} = \gamma_{84}.$ This leads to $$ (\overline{(Q^{\otimes 4}_{n_s}})^{\omega_{(s, 2)}})^{\Sigma_4} = \langle  [q_{s,\, 4}]_{\omega_{(s, 2)}},\, [q_{s,\, 5}]_{\omega_{(s, 5)}}, [a_{s,\, 102}]_{\omega_{(s, 2)}}  \rangle,$$
with $q_{s,\, 4}:= \sum_{35\leq j\leq 40}a_{s,\, j}$ and $q_{s,\, 5}:= \sum_{81\leq j\leq 84}a_{s,\, j}.$

Let $h\in (P_4)_{n_s}$ such that $[h]\in ((Q^{\otimes 4}_{n_s})^{\omega_{(s, 2)}})^{GL_4(k)}.$ Since $\Sigma_4\subset GL_4(k),$ from the above calculations, we have
$$ h\equiv_{\omega_{(s, 2)}} \big( \beta_1q_{s,\, 3} + \beta_2q_{s,\, 4} + \beta_3q_{s,\, 5} +\beta_4a_{s,\, 102}\big),$$
with $\beta_i\in k$ for every $i.$ Direct computing from the relation $(\sigma_4(h) + h)\equiv 0,$ we obtain
$$ \begin{array}{ll}
 (\sigma_4(h) + h) &\equiv_{\omega_{(s, 2)}}  \big((\beta_1 + \beta_2)a_{s,\, 25} + (\beta_2 + \beta_3)(a_{s,\, 36} + a_{s,\, 37})  + (\beta_3 + \beta_4)a_{s,\, 83}\big) \equiv 0,
\end{array}$$
and this is apparently that $\beta_1 = \beta_2 = \beta_3 = \beta_4.$ The proposition is proved.
\end{proof}

The following inequality is immediate from Proposition \ref{md3}. 
$$ 
 \dim k\otimes_{GL_4(k)}P((P_4)_{n_s}^{*}) \leq \dim ((Q^{\otimes 4}_{n_s})^{\omega_{(s, 1)}})^{GL_4(k)} + \dim ((Q^{\otimes 4}_{n_s})^{\omega_{(s, s)}})^{GL_4(k)} = 1.$$
On the other hand, as indicated in the introduction, we can verify that the element $$ \zeta_s = a_1^{(0)}a_2^{(2^{s+1}-1)}a_3^{(2^{s+1}-1)}a_4^{(2^{s+1}-1)}\in (P_4)^{*}_{3(2^{s} -1) + 3.2^{s}}\in P((P_4)_{n_s}^{*})$$ and that the cycle $\psi_4(\zeta_s) = \lambda_0\lambda_{s+1}^{3}$ in $\Lambda$ is a representative of the non-zero element $h_0h_{s+1}^{3}\in {\rm Ext}_A^{4, 6.2^{s}  +1}(k, k).$ Hence, $h_0h_{s+1}^{3}$ is in the image of $Tr_4^{A}.$ Moreover, by Theorem \ref{dlntg}, we have $ {\rm Ext}_A^{4, 6.2^{s}  +1}(k, k) = \langle h_0h_{s+1}^{3}\rangle $ with $h_0h_{s+1}^{3} = h_0h_s^{2}h_{s+2}\neq 0$ for any $s\geq 3.$ Combining these data gives $ \dim k\otimes_{GL_4(k)}P((P_4)_{n_s}^{*}) =  1.$ Moreover, by a simple computation, we find that $ (Q^{\otimes 4}_{n_s})^{GL_4(k)} = \langle [\zeta_s^{*}]\rangle.$ From this and the fact that $\langle [\zeta_s^*], [\zeta_s] \rangle = 1,$ it may be concluded that $ k\otimes_{GL_4(k)}P((P_4)_{n_s}^{*}) = \langle [\zeta_s]\rangle.$ The theorem is proved.
\end{proof}

\subsection{Proof of Theorem \ref{dlc2}}

Let $n_s:=3(2^{s}-1) + 7.2^{s},$ in order to prove the theorem,  we first have the following.

\begin{md}\label{mdbsc2}
For a positive integer $s,$ the invariant spaces $(Q^{\otimes 4}_{n_s})^{GL_4(k)}$ are determined as follows:
$$ (Q^{\otimes 4}_{n_s})^{GL_4(k)} = \left\{\begin{array}{ll}
\langle [\widetilde{\zeta}] \rangle &\mbox{for $s = 1$},\\
0 &\mbox{otherwise},
\end{array}\right.$$
where $$ \begin{array}{ll}
\widetilde{\zeta}&=
 x_1x_2x_3x_4^{14}+
 x_1x_2x_3^{14}x_4+
 x_1x_2^{3}x_3x_4^{12}+
\medskip
 x_1x_2^{3}x_3^{12}x_4\\
&\quad +
 x_1^{3}x_2x_3x_4^{12}+
 x_1^{3}x_2x_3^{12}x_4+
 x_1^{3}x_2^{5}x_3x_4^{8}+
 x_1^{3}x_2^{5}x_3^{8}x_4.
\end{array}$$
\end{md}

\begin{proof}

The proof of the proposition is based on an admissible monomial basis of the $k$-vector space $Q^{\otimes 4}_{n_s}.$ We have an isomorphism of $k$-vector spaces: $ Q^{\otimes 4}_{n_s} \cong \underline{Q^{\otimes 4}_{n_s}}\bigoplus \overline{Q^{\otimes 4}_{n_s}}.$ According to Sum \cite{N.S}, $\underline{Q^{\otimes 4}_{n_s}} = \langle \{[b_{s,\, j}]:\, 1\leq j\leq 40 \}\rangle,$ for all $s\geq 1,$ where the admissible monomials $b_{s,\, j}$ are determined as follows:\\[1mm]

\begin{tabular}{lrr}
$b_{s,\,1}= x_2^{2^{s+1}-1}x_3^{2^{s+2}-1}x_4^{2^{s+2}-1}$, & \multicolumn{1}{l}{$b_{s,\,2}= x_2^{2^{s+2}-1}x_3^{2^{s+1}-1}x_4^{2^{s+2}-1}$,} & \multicolumn{1}{l}{$b_{s,\,3}= x_2^{2^{s+2}-1}x_3^{2^{s+2}-1}x_4^{2^{s+1}-1}$,} \\
$b_{s,\,4}= x_1^{2^{s+1}-1}x_3^{2^{s+2}-1}x_4^{2^{s+2}-1}$, & \multicolumn{1}{l}{$b_{s,\,5}= x_1^{2^{s+2}-1}x_3^{2^{s+1}-1}x_4^{2^{s+2}-1}$,} & \multicolumn{1}{l}{$b_{s,\,6}= x_1^{2^{s+2}-1}x_3^{2^{s+2}-1}x_4^{2^{s+1}-1}$,} \\
$b_{s,\,7}= x_1^{2^{s+1}-1}x_2^{2^{s+2}-1}x_4^{2^{s+2}-1}$, & \multicolumn{1}{l}{$b_{s,\,8}= x_1^{2^{s+2}-1}x_2^{2^{s+1}-1}x_4^{2^{s+2}-1}$,} & \multicolumn{1}{l}{$b_{s,\,9}= x_1^{2^{s+2}-1}x_2^{2^{s+2}-1}x_4^{2^{s+1}-1}$,} \\
$b_{s,\,10}= x_1^{2^{s+1}-1}x_2^{2^{s+2}-1}x_3^{2^{s+2}-1}$, & \multicolumn{1}{l}{$b_{s,\,11}= x_1^{2^{s+2}-1}x_2^{2^{s+1}-1}x_3^{2^{s+2}-1}$,} & \multicolumn{1}{l}{$b_{s,\,12}= x_1^{2^{s+2}-1}x_2^{2^{s+2}-1}x_3^{2^{s+1}-1}$,} \\
$b_{s,\,13}= x_1^{2^{s}-1}x_2^{2^{s}-1}x_3^{2^{s+3}-1}$, & \multicolumn{1}{l}{$b_{s,\,14}= x_1^{2^{s}-1}x_2^{2^{s+3}-1}x_3^{2^{s+3}-1}$,} & \multicolumn{1}{l}{$b_{s,\,15}= x_1^{2^{s+3}-1}x_2^{2^{s}-1}x_3^{2^{s}-1}$,} \\
$b_{s,\,16}= x_2^{2^{s}-1}x_3^{2^{s+3}-1}$, & \multicolumn{1}{l}{$b_{s,\,17}= x_2^{2^{s+3}-1}x_3^{2^{s+3}-1}$,} & \multicolumn{1}{l}{$b_{s,\,18}= x_2^{2^{s}-1}x_3^{2^{s}-1}$,} \\
$b_{s,\,19}= x_1^{2^{s}-1}x_3^{2^{s+3}-1}$, & \multicolumn{1}{l}{$b_{s,\,20}= x_1^{2^{s+3}-1}x_3^{2^{s+3}-1}$,} & \multicolumn{1}{l}{$b_{s,\,21}= x_1^{2^{s}-1}x_3^{2^{s}-1}$,} \\
$b_{s,\,22}= x_1^{2^{s}-1}x_2^{2^{s+3}-1}$, & \multicolumn{1}{l}{$b_{s,\,23}= x_1^{2^{s+3}-1}x_2^{2^{s+3}-1}$,} & \multicolumn{1}{l}{$b_{s,\,24}= x_1^{2^{s}-1}x_2^{2^{s}-1}$,} \\
$b_{s,\,25}= x_1^{2^{s}-1}x_2^{2^{s+1}-1}x_3^{7.2^{s}-1}$, & \multicolumn{1}{l}{$b_{s,\,26}= x_1^{2^{s+1}-1}x_2^{2^{s}-1}x_3^{7.2^{s}-1}$,} & \multicolumn{1}{l}{$b_{s,\,27}= x_1^{2^{s+1}-1}x_2^{7.2^{s}-1}x_3^{2^{s}-1}$,} \\
$b_{s,\,28}= x_2^{2^{s+1}-1}x_3^{7.2^{s}-1}$, & \multicolumn{1}{l}{$b_{s,\,29}= x_2^{2^{s}-1}x_3^{7.2^{s}-1}$,} & \multicolumn{1}{l}{$b_{s,\,30}= x_2^{7.2^{s}-1}x_3^{2^{s}-1}$,} \\
$b_{s,\,31}= x_1^{2^{s+1}-1}x_3^{7.2^{s}-1}$, & \multicolumn{1}{l}{$b_{s,\,32}= x_1^{2^{s}-1}x_3^{7.2^{s}-1}$,} & \multicolumn{1}{l}{$b_{s,\,33}= x_1^{7.2^{s}-1}x_3^{2^{s}-1}$,} \\
$b_{s,\,34}= x_1^{2^{s+1}-1}x_2^{7.2^{s}-1}$, & \multicolumn{1}{l}{$b_{s,\,35}= x_1^{2^{s}-1}x_2^{7.2^{s}-1}$,} & \multicolumn{1}{l}{$b_{s,\,36}= x_1^{7.2^{s}-1}x_2^{2^{s}-1}$,} \\
$b_{s,\,37}= x_1^{2^{s+1}-1}x_2^{3.2^{s}-1}x_3^{5.2^{s}-1}$, & \multicolumn{1}{l}{$b_{s,\,38}= x_2^{3.2^{s}-1}x_3^{5.2^{s}-1}$,} & \multicolumn{1}{l}{$b_{s,\,39}= x_1^{3.2^{s}-1}x_3^{5.2^{s}-1}$,} \\
$b_{s,\,40}= x_1^{3.2^{s}-1}x_2^{5.2^{s}-1}$. &       &  
\end{tabular}%

\medskip

Then, we have a direct summand decomposition of $\Sigma_4$-modules:
$$ \underline{Q^{\otimes 4}_{n_s}} = \Sigma_4(b_{s,\, 1}) \bigoplus \Sigma_4(b_{s,\, 13}) \bigoplus \Sigma_4(b_{s,\, 25}) \bigoplus \Sigma_4(b_{s,\, 37}),$$
where $$ \begin{array}{ll}
\medskip
\Sigma_4(b_{s,\, 1}) &= \langle \{[b_{s,\, j}]:\, 1\leq j\leq 12\} \rangle;\ \ \ \Sigma_4(b_{s,\, 13}) = \langle \{[b_{s,\, j}]:\, 13\leq j\leq 24\} \rangle \\
\Sigma_4(b_{s,\, 25}) &= \langle \{[b_{s,\, j}]:\, 25\leq j\leq 36\} \rangle;\ \ \ \Sigma_4(b_{s,\, 37}) = \langle \{[b_{s,\, j}]:\, 37\leq j\leq 40\} \rangle.
\end{array}$$
Denote the bases of $\Sigma_4(b_{s,\, j})$ by the sets $[\mathscr B(b_{s,\, j})]$ for $j \in \{1, 13, 25, 37\},$ where $\mathscr B(b_{s,\, j})  = \{b_{s,\, j}\}.$ Suppose that $f\in (P_4)_{n_s}$ such that $[f]\in \Sigma_4(b_{s,\, j})^{\Sigma_4}$. Then, $f\equiv \sum_{x \in \mathscr B(b_{s,\, j})}\gamma_x.x$ in which $\gamma_x\in k.$ By a direct computation, we can see that the action of $\Sigma_4$ on $\underline{Q^{\otimes 4}_{n_s}}$ induces the one of it on $[\mathscr B(b_{s,\, j})]$ and this action is transitive. So, $\gamma_x = \gamma_{x'} = \gamma\in k$ for all $x,\, x'\in \mathscr B(b_{s,\, j}),$ and therefore, the spaces of $\Sigma_4$-invariants $\Sigma_4(b_{s,\, j})^{\Sigma_4}$ are determined as follows.

\begin{lema}\label{bdc21}
We find that
\begin{itemize}
\item[i)] $\Sigma_4(b_{s,\, 1})^{\Sigma_4} = \langle [q_{s,\, 1}] \rangle$ with $q_{s,\, 1}:= \sum_{1\leq j\leq 12}b_{s,\, j};$

\item[ii)] $\Sigma_4(b_{s,\, 13})^{\Sigma_4} = \langle [q_{s,\, 2}] \rangle$ with $q_{s,\, 2}:= \sum_{13\leq j\leq 24}b_{s,\, j};$

\item[iii)] $\Sigma_4(b_{s,\, 25})^{\Sigma_4} = \langle [q_{s,\, 3}] \rangle$ with $q_{s,\, 3}:= \sum_{25\leq j\leq 36}b_{s,\, j};$

\item[iv)] $\Sigma_4(b_{s,\, 37})^{\Sigma_4} = \langle [q_{s,\, 4}] \rangle$ with $q_{s,\, 4}:= \sum_{37\leq j\leq 40}b_{s,\, j}.$

\end{itemize}
\end{lema}

By Sum \cite{N.S}, the admissible bases of the spaces $\overline{Q^{\otimes 4}_{n_s}}$ are represented by the following admissible monomials:

For $s\geq 1,$

\begin{center}
\begin{tabular}{ll}
$b_{s,\,41}= x_1x_2^{2^{s}-1}x_3^{2^{s}-1}x_4^{2^{s+3}-2}$, & $b_{s,\,42}= x_1x_2^{2^{s}-1}x_3^{2^{s+3}-2}x_4^{2^{s}-1}$, \\
$b_{s,\,43}= x_1x_2^{2^{s+3}-2}x_3^{2^{s}-1}x_4^{2^{s}-1}$, & $b_{s,\,44}= x_1x_2^{2^{s}-1}x_3^{2^{s+1}-2}x_4^{7.2^{s}-1}$, \\
$b_{s,\,45}= x_1x_2^{2^{s+1}-2}x_3^{2^{s}-1}x_4^{7.2^{s}-1}$, & $b_{s,\,46}= x_1x_2^{2^{s+1}-2}x_3^{7.2^{s}-1}x_4^{2^{s}-1}$, \\
$b_{s,\,47}= x_1x_2^{2^{s+1}-2}x_3^{3.2^{s}-1}x_4^{5.2^{s}-1}$, & $b_{s,\,48}= x_1x_2^{2^{s+1}-2}x_3^{2^{s+2}-1}x_4^{2^{s+2}-1}$, \\
$b_{s,\,49}= x_1x_2^{2^{s+2}-1}x_3^{2^{s+1}-2}x_4^{2^{s+2}-1}$, & $b_{s,\,50}= x_1x_2^{2^{s+2}-1}x_3^{2^{s+2}-1}x_4^{2^{s+1}-2}$, \\
$b_{s,\,51}= x_1^{2^{s+2}-1}x_2x_3^{2^{s+1}-2}x_4^{2^{s+2}-1}$, & $b_{s,\,52}= x_1^{2^{s+2}-1}x_2x_3^{2^{s+2}-1}x_4^{2^{s+1}-2}$, \\
$b_{s,\,53}= x_1^{2^{s+2}-1}x_2^{2^{s+2}-1}x_3x_4^{2^{s+1}-2}$, & $b_{s,\,54}= x_1x_2^{2^{s+1}-1}x_3^{2^{s+2}-2}x_4^{2^{s+2}-1}$, \\
\end{tabular}
\end{center}

\newpage

\begin{center}
\begin{tabular}{ll}
$b_{s,\,55}= x_1x_2^{2^{s+1}-1}x_3^{2^{s+2}-1}x_4^{2^{s+2}-2}$, & $b_{s,\,56}= x_1x_2^{2^{s+2}-2}x_3^{2^{s+1}-1}x_4^{2^{s+2}-1}$, \\
$b_{s,\,57}= x_1x_2^{2^{s+2}-2}x_3^{2^{s+2}-1}x_4^{2^{s+1}-1}$, & $b_{s,\,58}= x_1x_2^{2^{s+2}-1}x_3^{2^{s+1}-1}x_4^{2^{s+2}-2}$, \\
$b_{s,\,59}= x_1x_2^{2^{s+2}-1}x_3^{2^{s+2}-2}x_4^{2^{s+1}-1}$, & $b_{s,\,60}= x_1^{2^{s+1}-1}x_2x_3^{2^{s+2}-2}x_4^{2^{s+2}-1}$, \\
$b_{s,\,61}= x_1^{2^{s+1}-1}x_2x_3^{2^{s+2}-1}x_4^{2^{s+2}-2}$, & $b_{s,\,62}= x_1^{2^{s+1}-1}x_2^{2^{s+2}-1}x_3x_4^{2^{s+2}-2}$, \\
$b_{s,\,63}= x_1^{2^{s+2}-1}x_2x_3^{2^{s+1}-1}x_4^{2^{s+2}-2}$, & $b_{s,\,64}= x_1^{2^{s+2}-1}x_2x_3^{2^{s+2}-2}x_4^{2^{s+1}-1}$, \\
$b_{s,\,65}= x_1^{2^{s+2}-1}x_2^{2^{s+1}-1}x_3x_4^{2^{s+2}-2}$, & $b_{s,\,66}= x_1^{3}x_2^{2^{s+2}-3}x_3^{2^{s+1}-2}x_4^{2^{s+2}-1}$, \\
$b_{s,\,67}= x_1^{3}x_2^{2^{s+2}-3}x_3^{2^{s+2}-1}x_4^{2^{s+1}-2}$, & $b_{s,\,68}= x_1^{3}x_2^{2^{s+2}-1}x_3^{2^{s+2}-3}x_4^{2^{s+1}-2}$, \\
$b_{s,\,69}= x_1^{2^{s+2}-1}x_2^{3}x_3^{2^{s+2}-3}x_4^{2^{s+1}-2}$, & $b_{s,\,70}= x_1^{3}x_2^{2^{s+1}-1}x_3^{2^{s+2}-3}x_4^{2^{s+2}-2}$, \\
$b_{s,\,71}= x_1^{3}x_2^{2^{s+2}-3}x_3^{2^{s+1}-1}x_4^{2^{s+2}-2}$, & $b_{s,\,72}= x_1^{3}x_2^{2^{s+2}-3}x_3^{2^{s+2}-2}x_4^{2^{s+1}-1}$.
\end{tabular}%
\end{center}

For $s = 1,$

\begin{center}
\begin{tabular}{lll}
$b_{1,\,73}= x_1^{3}x_2^{3}x_3^{4}x_4^{7}$, & $b_{1,\,74}= x_1^{3}x_2^{3}x_3^{7}x_4^{4}$, & $b_{1,\,75}= x_1^{3}x_2^{7}x_3^{3}x_4^{4}$, \\
$b_{1,\,76}= x_1^{7}x_2^{3}x_3^{3}x_4^{4}$, & $b_{1,\,77}= x_1x_2x_3^{3}x_4^{12}$, & $b_{1,\,78}= x_1x_2^{3}x_3x_4^{12}$, \\
$b_{1,\,79}= x_1x_2^{3}x_3^{12}x_4$, & $b_{1,\,80}= x_1^{3}x_2x_3x_4^{12}$, & $b_{1,\,81}= x_1^{3}x_2x_3^{12}x_4$, \\
$b_{1,\,82}= x_1x_2^{3}x_3^{4}x_4^{9}$, & $b_{1,\,83}= x_1^{3}x_2x_3^{4}x_4^{9}$, & $b_{1,\,84}= x_1x_2^{3}x_3^{5}x_4^{8}$, \\
$b_{1,\,85}= x_1^{3}x_2x_3^{5}x_4^{8}$, & $b_{1,\,86}= x_1^{3}x_2^{5}x_3x_4^{8}$, & $b_{1,\,87}= x_1^{3}x_2^{5}x_3^{8}x_4$.
\end{tabular}%
\end{center}

For $s\geq 2,$

\begin{center}
\begin{tabular}{ll}
$b_{s,\,73}= x_1^{2^{s+1}-1}x_2^{3}x_3^{2^{s+2}-3}x_4^{2^{s+2}-2}$, & $b_{s,\,74}= x_1^{3}x_2^{2^{s+1}-3}x_3^{2^{s+2}-2}x_4^{2^{s+2}-1}$, \\
$b_{s,\,75}= x_1^{3}x_2^{2^{s+1}-3}x_3^{2^{s+2}-1}x_4^{2^{s+2}-2}$, & $b_{s,\,76}= x_1^{3}x_2^{2^{s+2}-1}x_3^{2^{s+1}-3}x_4^{2^{s+2}-2}$, \\
$b_{s,\,77}= x_1^{2^{s+2}-1}x_2^{3}x_3^{2^{s+1}-3}x_4^{2^{s+2}-2}$, & $b_{s,\,78}= x_1^{7}x_2^{2^{s+2}-5}x_3^{2^{s+1}-3}x_4^{2^{s+2}-2}$, \\
$b_{s,\,79}= x_1^{7}x_2^{2^{s+2}-5}x_3^{2^{s+2}-3}x_4^{2^{s+1}-2}$, & $b_{s,\,80}= x_1^{7}x_2^{2^{s+1}-5}x_3^{2^{s+2}-3}x_4^{2^{s+2}-2}$, \\
$b_{s,\,81}= x_1x_2^{2^{s}-2}x_3^{2^{s}-1}x_4^{2^{s+3}-1}$, & $b_{s,\,82}= x_1x_2^{2^{s}-2}x_3^{2^{s+3}-1}x_4^{2^{s}-1}$, \\
$b_{s,\,83}= x_1x_2^{2^{s}-1}x_3^{2^{s}-2}x_4^{2^{s+3}-1}$, & $b_{s,\,84}= x_1x_2^{2^{s}-1}x_3^{2^{s+3}-1}x_4^{2^{s}-2}$, \\
$b_{s,\,85}= x_1x_2^{2^{s+3}-1}x_3^{2^{s}-2}x_4^{2^{s}-1}$, & $b_{s,\,86}= x_1x_2^{2^{s+3}-1}x_3^{2^{s}-1}x_4^{2^{s}-2}$, \\
$b_{s,\,87}= x_1^{2^{s}-1}x_2x_3^{2^{s}-2}x_4^{2^{s+3}-1}$, & $b_{s,\,88}= x_1^{2^{s}-1}x_2x_3^{2^{s+3}-1}x_4^{2^{s}-2}$, \\
$b_{s,\,89}= x_1^{2^{s}-1}x_2^{2^{s+3}-1}x_3x_4^{2^{s}-2}$, & $b_{s,\,90}= x_1^{2^{s+3}-1}x_2x_3^{2^{s}-2}x_4^{2^{s}-1}$, \\
$b_{s,\,91}= x_1^{2^{s+3}-1}x_2x_3^{2^{s}-1}x_4^{2^{s}-2}$, & $b_{s,\,92}= x_1^{2^{s+3}-1}x_2^{2^{s}-1}x_3x_4^{2^{s}-2}$, \\
$b_{s,\,93}= x_1^{2^{s}-1}x_2x_3^{2^{s}-1}x_4^{2^{s+3}-2}$, & $b_{s,\,94}= x_1^{2^{s}-1}x_2x_3^{2^{s+3}-2}x_4^{2^{s}-1}$, \\
$b_{s,\,95}= x_1^{2^{s}-1}x_2^{2^{s}-1}x_3x_4^{2^{s+3}-2}$, & $b_{s,\,96}= x_1x_2^{2^{s}-2}x_3^{2^{s+1}-1}x_4^{7.2^{s}-1}$, \\
$b_{s,\,97}= x_1x_2^{2^{s+1}-1}x_3^{2^{s}-2}x_4^{7.2^{s}-1}$, & $b_{s,\,98}= x_1x_2^{2^{s+1}-1}x_3^{7.2^{s}-1}x_4^{2^{s}-2}$, \\
$b_{s,\,99}= x_1^{2^{s+1}-1}x_2x_3^{2^{s}-2}x_4^{7.2^{s}-1}$, & $b_{s,\,100}= x_1^{2^{s+1}-1}x_2x_3^{7.2^{s}-1}x_4^{2^{s}-2}$, \\
$b_{s,\,101}= x_1^{2^{s+1}-1}x_2^{7.2^{s}-1}x_3x_4^{2^{s}-2}$, & $b_{s,\,102}= x_1^{2^{s}-1}x_2x_3^{2^{s+1}-2}x_4^{7.2^{s}-1}$, \\
$b_{s,\,103}= x_1x_2^{2^{s}-1}x_3^{2^{s+1}-1}x_4^{7.2^{s}-2}$, & $b_{s,\,104}= x_1x_2^{2^{s+1}-1}x_3^{2^{s}-1}x_4^{7.2^{s}-2}$, \\
$b_{s,\,105}= x_1x_2^{2^{s+1}-1}x_3^{7.2^{s}-2}x_4^{2^{s}-1}$, & $b_{s,\,106}= x_1^{2^{s}-1}x_2x_3^{2^{s+1}-1}x_4^{7.2^{s}-2}$, \\
$b_{s,\,107}= x_1^{2^{s}-1}x_2^{2^{s+1}-1}x_3x_4^{7.2^{s}-2}$, & $b_{s,\,108}= x_1^{2^{s+1}-1}x_2x_3^{2^{s}-1}x_4^{7.2^{s}-2}$, \\
$b_{s,\,109}= x_1^{2^{s+1}-1}x_2x_3^{7.2^{s}-2}x_4^{2^{s}-1}$, & $b_{s,\,110}= x_1^{2^{s+1}-1}x_2^{2^{s}-1}x_3x_4^{7.2^{s}-2}$, \\
$b_{s,\,111}= x_1x_2^{2^{s+1}-1}x_3^{3.2^{s}-2}x_4^{5.2^{s}-1}$, & $b_{s,\,112}= x_1^{2^{s+1}-1}x_2x_3^{3.2^{s}-2}x_4^{5.2^{s}-1}$, \\
$b_{s,\,113}= x_1x_2^{2^{s+1}-1}x_3^{3.2^{s}-1}x_4^{5.2^{s}-2}$, & $b_{s,\,114}= x_1^{2^{s+1}-1}x_2x_3^{3.2^{s}-1}x_4^{5.2^{s}-2}$, \\
$b_{s,\,115}= x_1^{2^{s+1}-1}x_2^{3.2^{s}-1}x_3x_4^{5.2^{s}-2}$, & $b_{s,\,116}= x_1^{3}x_2^{2^{s+1}-3}x_3^{2^{s}-2}x_4^{7.2^{s}-1}$, \\
$b_{s,\,117}= x_1^{3}x_2^{2^{s+1}-3}x_3^{7.2^{s}-1}x_4^{2^{s}-2}$, & $b_{s,\,118}= x_1^{3}x_2^{2^{s+1}-1}x_3^{7.2^{s}-3}x_4^{2^{s}-2}$, \\
$b_{s,\,119}= x_1^{2^{s+1}-1}x_2^{3}x_3^{7.2^{s}-3}x_4^{2^{s}-2}$, & $b_{s,\,120}= x_1^{3}x_2^{2^{s+1}-3}x_3^{3.2^{s}-2}x_4^{5.2^{s}-1}$, \\
$b_{s,\,121}= x_1^{3}x_2^{2^{s+1}-3}x_3^{3.2^{s}-1}x_4^{5.2^{s}-2}$, & $b_{s,\,122}= x_1^{3}x_2^{2^{s+1}-1}x_3^{3.2^{s}-3}x_4^{5.2^{s}-2}$, \\
$b_{s,\,123}= x_1^{2^{s+1}-1}x_2^{3}x_3^{3.2^{s}-3}x_4^{5.2^{s}-2}$, & $b_{s,\,124}= x_1^{3}x_2^{2^{s}-1}x_3^{2^{s+3}-3}x_4^{2^{s}-2}$, \\
$b_{s,\,125}= x_1^{3}x_2^{2^{s+3}-3}x_3^{2^{s}-2}x_4^{2^{s}-1}$, & $b_{s,\,126}= x_1^{3}x_2^{2^{s+3}-3}x_3^{2^{s}-1}x_4^{2^{s}-2}$, \\
$b_{s,\,127}= x_1^{3}x_2^{2^{s}-1}x_3^{2^{s+1}-3}x_4^{7.2^{s}-2}$, & $b_{s,\,128}= x_1^{3}x_2^{2^{s+1}-3}x_3^{2^{s}-1}x_4^{7.2^{s}-2}$.
\end{tabular}%
\end{center}

For $s =2,$

\begin{center}
\begin{tabular}{lrc}
$b_{2,\,129}= x_1^{3}x_2^{3}x_3^{3}x_4^{28}$, & \multicolumn{1}{l}{$b_{2,\,130}= x_1^{3}x_2^{3}x_3^{28}x_4^{3}$,} & $b_{2,\,131}= x_1^{3}x_2^{3}x_3^{4}x_4^{27}$, \\
$b_{2,\,132}= x_1^{3}x_2^{3}x_3^{7}x_4^{24}$, & \multicolumn{1}{l}{$b_{2,\,133}= x_1^{3}x_2^{7}x_3^{3}x_4^{24}$,} & $b_{2,\,134}= x_1^{7}x_2^{3}x_3^{3}x_4^{24}$, \\
$b_{2,\,135}= x_1^{7}x_2^{7}x_3^{9}x_4^{14}$. &       &  
\end{tabular}%
\end{center}

For $s\geq 3,$

\begin{center}
\begin{tabular}{lr}
$b_{s,\,129}= x_1^{3}x_2^{2^{s+1}-3}x_3^{7.2^{s}-2}x_4^{2^{s}-1}$, & \multicolumn{1}{l}{$b_{s,\,130}= x_1^{2^{s}-1}x_2^{3}x_3^{2^{s+1}-3}x_4^{7.2^{s}-2}$,} \\
$b_{s,\,131}= x_1^{2^{s}-1}x_2^{3}x_3^{2^{s+3}-3}x_4^{2^{s}-2}$, & \multicolumn{1}{l}{$b_{s,\,132}= x_1^{3}x_2^{2^{s}-3}x_3^{2^{s}-2}x_4^{2^{s+3}-1}$,} \\
$b_{s,\,133}= x_1^{3}x_2^{2^{s}-3}x_3^{2^{s+3}-1}x_4^{2^{s}-2}$, & \multicolumn{1}{l}{$b_{s,\,134}= x_1^{3}x_2^{2^{s+3}-1}x_3^{2^{s}-3}x_4^{2^{s}-2}$,} \\
$b_{s,\,135}= x_1^{2^{s+3}-1}x_2^{3}x_3^{2^{s}-3}x_4^{2^{s}-2}$, & \multicolumn{1}{l}{$b_{s,\,136}= x_1^{3}x_2^{2^{s}-3}x_3^{2^{s}-1}x_4^{2^{s+3}-2}$,} \\
$b_{s,\,137}= x_1^{3}x_2^{2^{s}-3}x_3^{2^{s+3}-2}x_4^{2^{s}-1}$, & \multicolumn{1}{l}{$b_{s,\,138}= x_1^{3}x_2^{2^{s}-1}x_3^{2^{s}-3}x_4^{2^{s+3}-2}$,} \\
$b_{s,\,139}= x_1^{2^{s}-1}x_2^{3}x_3^{2^{s}-3}x_4^{2^{s+3}-2}$, & \multicolumn{1}{l}{$b_{s,\,140}= x_1^{3}x_2^{2^{s}-3}x_3^{2^{s+1}-2}x_4^{7.2^{s}-1}$,} \\
$b_{s,\,141}= x_1^{3}x_2^{2^{s}-3}x_3^{2^{s+1}-1}x_4^{7.2^{s}-2}$, & \multicolumn{1}{l}{$b_{s,\,142}= x_1^{3}x_2^{2^{s+1}-1}x_3^{2^{s}-3}x_4^{7.2^{s}-2}$,} \\
$b_{s,\,143}= x_1^{2^{s+1}-1}x_2^{3}x_3^{2^{s}-3}x_4^{7.2^{s}-2}$, & \multicolumn{1}{l}{$b_{s,\,144}= x_1^{7}x_2^{2^{s+3}-5}x_3^{2^{s}-3}x_4^{2^{s}-2}$,} \\
$b_{s,\,145}= x_1^{7}x_2^{2^{s+1}-5}x_3^{2^{s}-3}x_4^{7.2^{s}-2}$, & \multicolumn{1}{l}{$b_{s,\,146}= x_1^{7}x_2^{2^{s+1}-5}x_3^{3.2^{s}-3}x_4^{5.2^{s}-2}$,} \\
$b_{s,\,147}= x_1^{7}x_2^{2^{s+1}-5}x_3^{7.2^{s}-3}x_4^{2^{s}-2}$. &  
\end{tabular}%
\end{center}

For $s = 3,$

\begin{center}
\begin{tabular}{lll}
$b_{3,\,148}= x_1^{7}x_2^{7}x_3^{7}x_4^{56}$, & $b_{3,\,149}= x_1^{7}x_2^{7}x_3^{9}x_4^{54}$, & $b_{3,\,150}= x_1^{7}x_2^{7}x_3^{57}x_4^{6}$.
\end{tabular}
\end{center}%

For $s\geq 4,$

\begin{center}
\begin{tabular}{lll}
$b_{s,\,148}= x_1^{7}x_2^{2^{s}-5}x_3^{2^{s}-3}x_4^{2^{s+3}-2}$, & $b_{s,\,149}= x_1^{7}x_2^{2^{s}-5}x_3^{2^{s+1}-3}x_4^{7.2^{s}-2}$, & $b_{s,\,150}= x_1^{7}x_2^{2^{s}-5}x_3^{2^{s+3}-3}x_4^{2^{s}-2}$.
\end{tabular}%
\end{center}

\begin{rem} Based upon the bases of the spaces $Q^{\otimes 4}_{n_s}$ above, $ \omega(b_{s,\, j})$ is one of the following sequences: 
$$  \omega_{(s, 1)}:=   \underset{\mbox{{$(s+1)$ times of $3$}}}{\underbrace{(3, 3, \ldots, 3}},2),\ \ \omega_{(s, 2)}:=   \underset{\mbox{{$s$ times of $3$}}}{\underbrace{(3, 3, \ldots, 3}}, 1,1,1).$$
Moreover, since $ \omega_{(s, 2)}$ is the weight vector of the minimal spike $b_{s,\, 15},$ $[x]_{ \omega_{(s, 2)}} = [x]$ for all monomials $x$ in $(P_4)_{n_s}.$
\end{rem}

\underline{Case $s =1$}. From the admissible basis of $ \overline{Q^{\otimes 4}_{n_1}}$ above, we have a direct summand decomposition of $\Sigma_4$-modules: $\overline{Q^{\otimes 4}_{n_1}} = \mathcal M_1\bigoplus \mathcal M_2,$ where $\mathcal M_1 = \langle \{[b_{1,\, j}]:\, j = 41, \ldots, 47, 77, \ldots, 87 \} \rangle$ and $\mathcal M_2 = \langle \{[b_{1,\, j}]:\, j = 48, \ldots, 76\} \rangle.$

\begin{lema}\label{bdc22}
The following assertions are true:
\begin{itemize}

\item[i)] $\mathcal M_1^{\Sigma_4} = \langle [\widetilde{\zeta}] \rangle$ with $\widetilde{\zeta} = b_{1,\, 41} + b_{1,\, 42} + b_{1,\, 78} + b_{1,\, 79}+b_{1,\, 80}+b_{1,\, 81}  + b_{1,\, 86} + b_{1,\, 87};$

\item[ii)] $\mathcal M_2^{\Sigma_4} = 0.$

\end{itemize}
\end{lema}

\begin{proof}
We prove i) in detail. By a similar technique, we also get ii). Suppose that $f\in (P_4)_{n_1}$ such that $[f]\in \mathcal M_1^{\Sigma_4}.$ Then, we have $$f\equiv \big(\sum_{41\leq j\leq 47}\gamma_jb_{1,\, j} + \sum
_{77\leq j\leq 87}\gamma_jb_{1,\, j}\big).$$
Direct calculating $\sigma_i(f)$ in terms of $b_{1,\, j},\, j\in \{41, \ldots, 47, 77, \ldots, 87\}$ modulo ($\widehat{A}(P_4)_{n_1}$) and using the relations $(\sigma_i(f) + f) \equiv 0,$ for $i\in \{1, 2, 3\},$ we get
$$ \begin{array}{ll}
\medskip
(\sigma_1(f) + f)&\equiv \big(\gamma_{45}b_{1,\, 41} +  \gamma_{45}b_{1,\, 46} + (\gamma_{45} + \gamma_{47})b_{1,\, 44} + (\gamma_{46} + \gamma_{47})b_{1,\, 77}\\
\medskip
&\quad + (\gamma_{43} + \gamma_{78} + \gamma_{80})b_{1,\, 78}  + (\gamma_{43} + \gamma_{79} + \gamma_{81})b_{1,\, 79}\\
\medskip
&\quad +(\gamma_{82} + \gamma_{83})b_{1,\, 82} + (\gamma_{84} + \gamma_{85})b_{1,\, 84} + \gamma_{43}b_{1,\, 86} + \mbox{others terms}\big)\equiv 0,\\
\medskip
(\sigma_2(f) + f)&\equiv \big(\gamma_{43}b_{1,\, 41} + (\gamma_{42} + \gamma_{43} + \gamma_{87})b_{1,\, 42} +  (\gamma_{42} + \gamma_{43} + \gamma_{81})b_{1,\, 43}\\
\medskip
&\quad + (\gamma_{44} + \gamma_{45} + \gamma_{83})b_{1,\, 44} + (\gamma_{46} + \gamma_{79} + \gamma_{87})b_{1,\, 46} + (\gamma_{47} + \gamma_{82})b_{1,\, 47}\\
\medskip
&\quad +(\gamma_{77} + \gamma_{78} + \gamma_{87})b_{1,\, 77} + (\gamma_{77} + \gamma_{78} + \gamma_{81})b_{1,\, 78} + (\gamma_{46} + \gamma_{79} + \gamma_{81})b_{1,\, 79}\\
\medskip
&\quad + \gamma_{83}b_{1,\, 80} + (\gamma_{81} + \gamma_{87})b_{1,\, 81} + (\gamma_{81} + \gamma_{87})b_{1,\, 81} + (\gamma_{85} +\gamma_{86}+  \gamma_{87})b_{1,\, 85}\\
\medskip
&\quad + (\gamma_{81} + \gamma_{85} + \gamma_{86})b_{1,\, 86} + \mbox{others terms}\big)\equiv 0,\\
   \medskip
(\sigma_3(f) + f)&\equiv \big((\gamma_{41} + \gamma_{42})b_{1,\, 41} + (\gamma_{44} + \gamma_{47})b_{1,\, 44} + (\gamma_{45} + \gamma_{46})b_{1,\, 45} + (\gamma_{78} + \gamma_{79})b_{1,\, 78}\\
\medskip
&\quad + (\gamma_{80} + \gamma_{81})b_{1,\, 80} + (\gamma_{82} + \gamma_{84})b_{1,\, 82} + (\gamma_{83} + \gamma_{85})b_{1,\, 83}\\
&\quad + (\gamma_{86} + \gamma_{87})b_{1,\, 86} + \mbox{others terms}\big)\equiv 0.
\end{array}$$
These equalities imply that $\gamma_{41} = \gamma_{42} = \gamma_{78} = \cdots  = \gamma_{81} = \gamma_{86} = \gamma_{87}$ and $\gamma_j = 0$ with $j\not\in \{41, 42, 78,79,80, 81, 86, 87\}.$ The lemma follows.
\end{proof}

Now , assume that $[h]\in (Q^{\otimes 4}_{n_1})^{GL_4(k)},$ then since $\Sigma_4\subset GL_4(k),$ by Lemmas \ref{bdc21} and \ref{bdc22}, we have
$$ h\equiv \big(\beta_1q_{2,2} + \beta_2q_{2,3} + \beta_3q_{2,4} + \beta_4\widetilde{\zeta}\big),$$
where $\beta_i\in k$ for every $i.$ Computing $\sigma_4(h)$ in terms of $b_{1,\, j},\, 1\leq i\leq 87$ modulo ($\widehat{A}(P_4)_{n_1}$) and using the relation $(\sigma_4(f) + f) \equiv 0,$ we obtain
$$ (\sigma_4(f) + f) \equiv \big(\beta_1b_{1,\, 13} + (\beta_1+\beta_2)b_{1,\, 20} + (\beta_2+\beta_3)b_{1,\, 31} + \mbox{others terms}\big)\equiv 0.$$
This leads to $\beta_1 = \beta_2 = \beta_3 = 0.$ 

\medskip

\underline{Case $s\geq 2.$} It is sufficient to show that the spaces  $((Q^{\otimes 4}_{n_s})^{\omega_{(s, 1)}})^{GL_4(k)}$ and $((Q^{\otimes 4}_{n_s})^{\omega_{(s, 2)}})^{GL_4(k)}$  are trivial for every $s\geq 2.$ Indeed, we prove this statement for the invariant $((Q^{\otimes 4}_{n_s})^{\omega_{(s, 1)}})^{GL_4(k)}$ with $s\geq 4.$ The other cases can be proved by similar computations.

\medskip

We have $(Q^{\otimes 4}_{n_s})^{\omega_{(s, 1)}}\cong \underline{(Q^{\otimes 4}_{n_s})}^{\omega_{(s, 1)}}\bigoplus \overline{(Q^{\otimes 4}_{n_s})}^{\omega_{(s, 1)}},$ where $$\underline{(Q^{\otimes 4}_{n_s})}^{\omega_{(s, 1)}} = \langle \{[b_{s, j}]_{\omega_{(s, 1)}}:\, 1\leq j\leq 12\}\rangle,\ \mbox{and}\ \overline{(Q^{\otimes 4}_{n_s})}^{\omega_{(s, 1)}} = \langle \{[b_{s, j}]_{\omega_{(s, 1)}}:\, 48\leq j\leq 80\}\rangle.$$
Then, it is easy to check that 
$$ \overline{(Q^{\otimes 4}_{n_s})}^{\omega_{(s, 1)}}  \cong \Sigma_4(b_{s,\, 48})\bigoplus  \Sigma_4(b_{s,\, 54}, b_{s,\, 66})\bigoplus \mathcal M,$$
where 
$$ \begin{array}{ll}
\medskip
\Sigma_4(b_{s,\, 48}) &= \langle \{[b_{s,\, j}]_{\omega_{(s, 1)}}:\, j = 48, \ldots, 53 \} \rangle, \\
\medskip
  \Sigma_4(b_{s,\, 54}, b_{s,\, 66}) &= \langle \{[b_{s,\, j}]_{\omega_{(s, 1)}}:\, j = 54, \ldots, 69, 74, 75, 76, 77 \} \rangle,\\
\medskip
\mathcal M &= \langle \{[b_{s,\, j}]_{\omega_{(s, 1)}}:\, j = 70, 71, 72, 73, 78, 79, 80\} \rangle.
\end{array}$$

\begin{lema}\label{bdc23}
The following hold:
\item[i)] $\Sigma_4(b_{s,\, 48})^{\Sigma_4} = \langle [\widehat{q_{s,\, 2}}]_{\omega_{(s, 1)}}\rangle$ with $\widehat{q_{s,\, 2}}:= \sum_{48\leq j\leq 53}b_{s,\, j}.$

\item[ii)] $\Sigma_4(b_{s,\, 54}, b_{s,\, 66})^{\Sigma_4} = \langle [\widehat{q_{s,\, 3}}]_{\omega_{(s, 1)}}\rangle$ with $\widehat{q_{s,\, 3}}:= \sum_{54\leq j\leq 69}b_{s,\, j} + \sum_{74\leq j\leq 77}b_{s,\, j}.$

\item[iii)] $\mathcal M^{\Sigma_4} = \langle [\widehat{q_{s,\, 4}}]_{\omega_{(s, 1)}}, [\widehat{q_{s,\, 5}}]_{\omega_{(s, 1)}}\rangle$ with $$\widehat{q_{s,\, 4}}:= \sum_{70\leq j\leq 73}b_{s,\, j} + b_{s,\, 78} + b_{s,\, 79}, \ \  \widehat{q_{s,\, 5}}:= \sum_{70\leq j\leq 73}b_{s,\, j} + b_{s,\, 80}.$$
\end{lema}

\begin{proof}
We prove Part iii) in detail. The others parts can be computed by a similar idea. Computing the general case is very technical. Indeed, we have the set $\{[b_{s,\, j}]_{\omega_{(s, 1)}}:\, j = 70, 71, 72, 73, 78, 79, 80\},$ which is an admissible basis of $\mathcal M.$ Suppose that $[f]\in \mathcal M^{\Sigma_4},$ then $$ f\equiv_{\omega_{(s, 1)}}\big(\sum_{70\leq j\leq 73}\gamma_jb_{s,\, j} + \sum_{78\leq j\leq 80}\gamma_jb_{s,\, j}\big),$$
with $\gamma_j\in k$ for every $j.$ Using the homomorphisms $\sigma_i: (P_4)_{n_s}\to (P_4)_{n_s},\, i = 1, 2, 3,$ we have
$$ \begin{array}{ll}
\medskip
\sigma_1(f)&\equiv_{\omega_{(s, 1)}} \big(\gamma_{73}b_{s,\, 70} + \gamma_{70}b_{s,\, 73} + \gamma_{71}x_1^{2^{s+2}-3}x_2^{3}x_3^{2^{s+1}-1}x_4^{2^{s+2}-2}\\
\medskip
&\quad + \gamma_{72}x_1^{2^{s+2}-3}x_2^{3}x_3^{2^{s+2}-2}x_4^{2^{s+1}-1} + \gamma_{78}x_1^{2^{s+2}-5}x_2^{7}x_3^{2^{s+1}-3}x_4^{2^{s+2}-2}\\
\medskip
&\quad + \gamma_{79}x_1^{2^{s+2}-5}x_2^{7}x_3^{2^{s+2}-3}x_4^{2^{s+1}-2} + \gamma_{80}x_1^{2^{s+1}-5}x_2^{7}x_3^{2^{s+2}-3}x_4^{2^{s+2}-2}\big),\\
\medskip
\sigma_2(f)&\equiv_{\omega_{(s, 1)}} \big(\gamma_{71}b_{s,\, 70} + \gamma_{70}b_{s,\, 71} + \gamma_{72}x_1^{3}x_2^{2^{s+2}-2}x_3^{2^{s+2}-3}x_4^{2^{s+1}-1}\\
\medskip
&\quad +  \gamma_{73}x_1^{2^{s+1}-1}x_2^{2^{s+2}-3}x_3^{3}x_4^{2^{s+2}-2} +\gamma_{78}x_1^{7}x_2^{2^{s+1}-3}x_3^{2^{s+2}-5}x_4^{2^{s+2}-2}\\
\medskip
&\quad + \gamma_{79}x_1^{7}x_2^{2^{s+2}-3}x_3^{2^{s+2}-5}x_4^{2^{s+1}-2} +  \gamma_{80}x_1^{7}x_2^{2^{s+2}-3}x_3^{2^{s+1}-5}x_4^{2^{s+2}-2}\big),\\
\medskip
\sigma_3(f)&\equiv_{\omega_{(s, 1)}} \big(\gamma_{72}b_{s,\, 71} + \gamma_{71}b_{s,\, 72} + \gamma_{70}x_1^{3}x_2^{2^{s+1}-1}x_3^{2^{s+2}-2}x_4^{2^{s+2}-3}\\
\medskip
&\quad +  \gamma_{73}x_1^{2^{s+1}-1}x_2^{3}x_3^{2^{s+2}-2}x_4^{2^{s+2}-3} +\gamma_{78}x_1^{7}x_2^{2^{s+2}-5}x_3^{2^{s+2}-2}x_4^{2^{s+1}-3}\\
\medskip
&\quad + \gamma_{79}x_1^{7}x_2^{2^{s+2}-5}x_3^{2^{s+1}-2}x_4^{2^{s+2}-3} +  \gamma_{80}x_1^{7}x_2^{2^{s+1}-5}x_3^{2^{s+2}-2}x_4^{2^{s+2}-3}\big).
\end{array}$$
Using the actions of the Steenrod squares $Sq^{i}$ for $i\in \{1,2,4,8\}$ and the Cartan formula, we get
$$  \begin{array}{ll}
x_1^{3}x_2^{2^{s+1}-1}x_3^{2^{s+2}-2}x_4^{2^{s+2}-3}&\equiv_{\omega_{(s, 1)}} b_{s,\, 70},\bigg|\ x_1^{3}x_2^{2^{s+2}-2}x_3^{2^{s+2}-3}x_4^{2^{s+1}-1}\equiv_{\omega_{(s, 1)}} b_{s,\, 72},\\
x_1^{7}x_2^{2^{s+1}-5}x_3^{2^{s+2}-2}x_4^{2^{s+2}-3}&\equiv_{\omega_{(s, 1)}} b_{s,\, 80},\bigg|\ x_1^{7}x_2^{2^{s+1}-3}x_3^{2^{s+2}-5}x_4^{2^{s+2}-2}\equiv_{\omega_{(s, 1)}} b_{s,\, 80},\\
x_1^{7}x_2^{2^{s+2}-5}x_3^{2^{s+1}-2}x_4^{2^{s+2}-3}&\equiv_{\omega_{(s, 1)}} b_{s,\, 78},\bigg|\ x_1^{7}x_2^{2^{s+2}-5}x_3^{2^{s+2}-2}x_4^{2^{s+1}-3}\equiv_{\omega_{(s, 1)}} b_{s,\, 79},\\
x_1^{7}x_2^{2^{s+2}-3}x_3^{2^{s+1}-5}x_4^{2^{s+2}-2}&\equiv_{\omega_{(s, 1)}} b_{s,\, 78},\bigg|\ x_1^{7}x_2^{2^{s+2}-3}x_3^{2^{s+2}-5}x_4^{2^{s+1}-2}\equiv_{\omega_{(s, 1)}} b_{s,\, 79},\\
x_1^{2^{s+1}-5}x_2^{7}x_3^{2^{s+2}-3}x_4^{2^{s+2}-2}&\equiv_{\omega_{(s, 1)}} b_{s,\, 80}, \bigg|\ x_1^{2^{s+1}-1}x_2^{3}x_3^{2^{s+2}-2}x_4^{2^{s+2}-3}\equiv_{\omega_{(s, 1)}} b_{s,\, 73},\\
\medskip
x_1^{2^{s+1}-1}x_2^{2^{s+2}-3}x_3^{3}x_4^{2^{s+2}-2}&\equiv_{\omega_{(s, 1)}} \big(b_{s,\, 73} + b_{s,\, 78} + b_{s,\, 80}\big),  \\
\medskip
x_1^{2^{s+2}-5}x_2^{7}x_3^{2^{s+1}-3}x_4^{2^{s+2}-2}&\equiv_{\omega_{(s, 1)}}  \big(b_{s,\, 78} + b_{s,\, 80}\big),\\
\medskip
x_1^{2^{s+2}-5}x_2^{7}x_3^{2^{s+2}-3}x_4^{2^{s+1}-2}&\equiv_{\omega_{(s, 1)}}  \big(b_{s,\, 79} + b_{s,\, 80}\big),\\
\medskip
x_1^{2^{s+2}-3}x_2^{3}x_3^{2^{s+1}-1}x_4^{2^{s+2}-2}&\equiv_{\omega_{(s, 1)}}  \big(b_{s,\, 71} + b_{s,\, 80}\big),\\
\medskip
x_1^{2^{s+2}-3}x_2^{3}x_3^{2^{s+2}-2}x_4^{2^{s+1}-1}&\equiv_{\omega_{(s, 1)}}  \big(b_{s,\, 72} + b_{s,\, 80}\big).
\end{array}$$
Combining the above computations and the relations $\sigma_i(f)\equiv_{\omega_{(s,\, 1)}} f$ for $1\leq i\leq 3,$ we obtain
$\gamma_{78} = \gamma_{79}$ and $\gamma_{j} + \gamma_{79} + \gamma_{80} = 0,$ for $70\leq j\leq 73.$  The lemma is proved.
\end{proof}

Now, let $[g]\in (Q^{\otimes 4}_{n_s})^{\omega_{(s, 1)}})^{GL_4(k)},$ by Lemmas \ref{bdc21} and \ref{bdc23}, we have
$$ g\equiv_{\omega_{(s, 1)}}\big(\rho_1q_{s,\, 1} + \rho_2\widehat{q_{s,\, 2}} + \rho_3\widehat{q_{s,\, 3}} + \rho_4\widehat{q_{s,\, 4}}+ \rho_5\widehat{q_{s,\, 5}}\big),$$
with $\rho_i\in k$ for all $i,\, 1\leq i\leq 5.$ Direct computing from the relation $(\sigma_4(g) + g)\equiv_{\omega_{(s, 1)}} 0,$ we get
$$ \begin{array}{ll}
\medskip
(\sigma_4(g) + g) &\equiv_{\omega_{(s, 1)}} \big((\rho_1 + \rho_2)b_{s,\, 1} + (\rho_1 + \rho_3)b_{s,\, 2} + \rho_1b_{s,\, 7} + (\rho_4 + \rho_5)b_{s,\, 58}\\
&\quad\quad\quad + \rho_4b_{s,\, 76} + \mbox{others terms}\big) \equiv_{\omega_{(s, 1)}} 0,
\end{array}$$
and therefore $\rho_i = 0$ for all $i.$ This completes the proof of the proposition.
\end{proof}

Now, we consider the element $\zeta\in (P_4)^{*}_{n_1}$, which is the following sum:\\[2mm]
$\begin{array}{ll}
&a_1^{(5)}a_2^{(5)}a_3^{(5)}a_4^{(2)}+
 a_1^{(5)}a_2^{(5)}a_3^{(6)}a_4^{(1)}+
 a_1^{(3)}a_2^{(5)}a_3^{(8)}a_4^{(1)}+
a_1^{(5)}a_2^{(3)}a_3^{(8)}a_4^{(1)} +
\medskip
 a_1^{(3)}a_2^{(6)}a_3^{(7)}a_4^{(1)}\\
&+ a_1^{(5)}a_2^{(7)}a_3^{(4)}a_4^{(1)}+
a_1^{(7)}a_2^{(5)}a_3^{(4)}a_4^{(1)}+
 a_1^{(3)}a_2^{(9)}a_3^{(4)}a_4^{(1)}+
 a_1^{(9)}a_2^{(3)}a_3^{(4)}a_4^{(1)}+
\medskip
a_1^{(3)}a_2^{(9)}a_3^{(3)}a_4^{(2)}\\
&+
a_1^{(9)}a_2^{(3)}a_3^{(3)}a_4^{(2)}+
a_1^{(5)}a_2^{(9)}a_3^{(2)}a_4^{(1)}+
a_1^{(9)}a_2^{(5)}a_3^{(2)}a_4^{(1)} +
a_1^{(5)}a_2^{(10)}a_3^{(1)}a_4^{(1)}+
\medskip
a_1^{(9)}a_2^{(6)}a_3^{(1)}a_4^{(1)}\\
&+
 a_1^{(3)}a_2^{(11)}a_3^{(2)}a_4^{(1)}+
 a_1^{(11)}a_2^{(3)}a_3^{(2)}a_4^{(1)} +
 a_1^{(5)}a_2^{(5)}a_3^{(3)}a_4^{(4)}+
 a_1^{(5)}a_2^{(3)}a_3^{(5)}a_4^{(4)}+
\medskip
 a_1^{(3)}a_2^{(5)}a_3^{(5)}a_4^{(4)}\\
&+
 a_1^{(3)}a_2^{(12)}a_3^{(1)}a_4^{(1)}+
  a_1^{(11)}a_2^{(4)}a_3^{(1)}a_4^{(1)}+
 a_1^{(7)}a_2^{(8)}a_3^{(1)}a_4^{(1)}+
 a_1^{(7)}a_2^{(7)}a_3^{(1)}a_4^{(2)}+
\medskip
a_1^{(13)}a_2^{(2)}a_3^{(1)}a_4^{(1)}\\
&+
 a_1^{(14)}a_2^{(1)}a_3^{(1)}a_4^{(1)}+
 a_1^{(6)}a_2^{(5)}a_3^{(3)}a_4^{(3)}+
 a_1^{(5)}a_2^{(3)}a_3^{(6)}a_4^{(3)}+
 a_1^{(3)}a_2^{(6)}a_3^{(5)}a_4^{(3)} +
\medskip
 a_1^{(6)}a_2^{(3)}a_3^{(3)}a_4^{(5)}\\
&+
 a_1^{(3)}a_2^{(3)}a_3^{(6)}a_4^{(5)}+
 a_1^{(3)}a_2^{(6)}a_3^{(3)}a_4^{(5)}+
  a_1^{(5)}a_2^{(3)}a_3^{(3)}a_4^{(6)}+
 a_1^{(3)}a_2^{(5)}a_3^{(3)}a_4^{(6)}+
\medskip
 a_1^{(3)}a_2^{(3)}a_3^{(5)}a_4^{(6)}\\
&+
 a_1^{(3)}a_2^{(3)}a_3^{(3)}a_4^{(8)}+
a_1^{(3)}a_2^{(3)}a_3^{(4)}a_4^{(7)} +
 a_1^{(3)}a_2^{(5)}a_3^{(2)}a_4^{(7)}+
 a_1^{(3)}a_2^{(6)}a_3^{(1)}a_4^{(7)}+
\medskip
 a_1^{(3)}a_2^{(3)}a_3^{(9)}a_4^{(2)}\\
&+
 a_1^{(3)}a_2^{(3)}a_3^{(10)}a_4^{(1)}+
a_1^{(5)}a_2^{(3)}a_3^{(7)}a_4^{(2)}+
 a_1^{(5)}a_2^{(7)}a_3^{(3)}a_4^{(2)}+
 a_1^{(7)}a_2^{(5)}a_3^{(3)}a_4^{(2)}.
\end{array}$\\[2mm]
Then, $\zeta$ is $\widehat{A}$-annihilated (or $\zeta\in P((P_4)^{*}_{n_1})$). Indeed, notice that by the unstable condition, we need only to compute the (right) action of $Sq^{i}$ for $i = 1, 2, 4.$ It is easy to check that $(\zeta)Sq^{i} = 0$ for $i = 1, 2.$ A direct calculation shows that $(\zeta)Sq^{4}$ is the following sum:\\[2mm]
$\begin{array}{ll}
& a_1^{(3)}a_2^{(3)}a_3^{(5)}a_4^{(2)}+
 a_1^{(3)}a_2^{(5)}a_3^{(3)}a_4^{(2)}+
 a_1^{(5)}a_2^{(3)}a_3^{(3)}a_4^{(2)}+
 a_1^{(3)}a_2^{(3)}a_3^{(6)}a_4^{(1)}+
\medskip
 a_1^{(3)}a_2^{(5)}a_3^{(4)}a_4^{(1)}\\
&+
 a_1^{(3)}a_2^{(3)}a_3^{(6)}a_4^{(1)}+
 a_1^{(5)}a_2^{(3)}a_3^{(4)}a_4^{(1)}+
 a_1^{(3)}a_2^{(3)}a_3^{(6)}a_4^{(1)}+
 a_1^{(3)}a_2^{(7)}a_3^{(2)}a_4^{(1)}+
\medskip
 a_1^{(7)}a_2^{(3)}a_3^{(2)}a_4^{(1)}\\
&+
 a_1^{(3)}a_2^{(5)}a_3^{(4)}a_4^{(1)}+
 a_1^{(3)}a_2^{(7)}a_3^{(2)}a_4^{(1)}+
 a_1^{(5)}a_2^{(3)}a_3^{(4)}a_4^{(1)}+
 a_1^{(7)}a_2^{(3)}a_3^{(2)}a_4^{(1)}+
\medskip
 a_1^{(3)}a_2^{(5)}a_3^{(3)}a_4^{(2)}\\
&+
 a_1^{(5)}a_2^{(3)}a_3^{(3)}a_4^{(2)}+
 a_1^{(5)}a_2^{(5)}a_3^{(2)}a_4^{(1)}+
 a_1^{(3)}a_2^{(7)}a_3^{(2)}a_4^{(1)}+
 a_1^{(5)}a_2^{(5)}a_3^{(2)}a_4^{(1)}+
\medskip
 a_1^{(7)}a_2^{(3)}a_3^{(2)}a_4^{(1)}\\
&+
 a_1^{(5)}a_2^{(6)}a_3^{(1)}a_4^{(1)}+
 a_1^{(5)}a_2^{(6)}a_3^{(1)}a_4^{(1)}+
 a_1^{(3)}a_2^{(7)}a_3^{(2)}a_4^{(1)}+
 a_1^{(7)}a_2^{(3)}a_3^{(2)}a_4^{(1)}+
\medskip
 a_1^{(3)}a_2^{(3)}a_3^{(3)}a_4^{(4)}\\
&+
 a_1^{(3)}a_2^{(5)}a_3^{(3)}a_4^{(2)}+
 a_1^{(5)}a_2^{(3)}a_3^{(3)}a_4^{(2)}+
 a_1^{(3)}a_2^{(3)}a_3^{(3)}a_4^{(4)}+
 a_1^{(3)}a_2^{(3)}a_3^{(5)}a_4^{(2)}+
\medskip
 a_1^{(5)}a_2^{(3)}a_3^{(3)}a_4^{(2)}\\
&+
 a_1^{(3)}a_2^{(3)}a_3^{(3)}a_4^{(4)}+
 a_1^{(3)}a_2^{(3)}a_3^{(5)}a_4^{(2)}+
 a_1^{(3)}a_2^{(5)}a_3^{(3)}a_4^{(2)}+
 a_1^{(7)}a_2^{(4)}a_3^{(1)}a_4^{(1)}+
\medskip
 a_1^{(7)}a_2^{(4)}a_3^{(1)}a_4^{(1)}\\
&+
 a_1^{(3)}a_2^{(3)}a_3^{(3)}a_4^{(4)}+
 a_1^{(3)}a_2^{(3)}a_3^{(5)}a_4^{(2)}+
 a_1^{(3)}a_2^{(3)}a_3^{(6)}a_4^{(1)}.
\end{array}$\\[1mm]
Thus, $(\zeta)Sq^{4} = 0.$ So, combining Proposition \ref{mdbsc2} and the fact that $\langle [\widetilde{\zeta}], [\zeta] \rangle  =1$ gives
$$ k\otimes_{GL_4(k)}P((P_4)_{n_s}^{*}) = ((Q^{\otimes 4}_{n_s})^{GL_4(k)})^{*}= \left\{\begin{array}{ll}
\langle ([\widetilde{\zeta}])^*\rangle = \langle [\zeta] \rangle &\mbox{for $s = 1$},\\
0 &\mbox{otherwise}.
\end{array}\right.$$
The theorem is proved.

\subsection{Proof of Theorem \ref{dlct}}

We put $n_{s,t}:= 3(2^{s}-1) + 2^{s}(2^{t+1}-1),$ then from a result in Sum \cite{N.S},  we have an isomorphism $ Q^{\otimes 4}_{n_{s, t}} \cong (Q^{\otimes 4}_{n_{s,t}})^{\omega_{(s, t, 1)}}\bigoplus (Q^{\otimes 4}_{n_{s,t}})^{\omega_{(s, t, 2)}},$ where $\omega_{(s, t, 1)}:= \underset{\mbox{{$s$ times }}}{\underbrace{(3, 3, \ldots, 3}}, \underset{\mbox{{$(t+1)$ times }}}{\underbrace{2, 2, \ldots, 2)}}$ and $\omega_{(s, t, 2)}:= \underset{\mbox{{$(s+1)$ times }}}{\underbrace{(3, 3, \ldots, 3}}, \underset{\mbox{{$(t-1)$ times }}}{\underbrace{2, 2, \ldots, 2)}}.$ By using an admissible basis of $(Q^{\otimes 4}_{n_{s,t}})^{\omega_{(s, t, j)}}$ in \cite{N.S} and similar techniques as in the proof of Proposition \ref{md3}, we have the following.

\begin{md}\label{mdt}
Let $s$ and $t$ be positive integers such that $t\geq 4.$ Then, 
$$ \dim((Q^{\otimes 4}_{n_{s,t}})^{\omega_{(s, t, j)}})^{GL_4(k)} =\left\{\begin{array}{ll}
0&\mbox{for $j = 1$ and $s\in \{1, 2\}$},\\
1&\mbox{for $j = 2$ and $s\in \{1, 2\}$},\\
1&\mbox{for $j\in \{1, 2\}$ and $s\geq 3$}.
\end{array}\right.$$
\end{md}

On the other hand, we see that the elements $\zeta_{s, t, 1} = a_1^{(0)}a_2^{(2^{s+1}-1)}a_3^{(2^{s+t}-1)}a_4^{(2^{s+t}-1)}$ and $\zeta_{s, t, 2} = a_1^{(0)}a_2^{(2^{s}-1)}a_3^{(2^{s}-1)}a_4^{(2^{s+t+1}-1)}$ belong to ${\rm Ext}_A^{0, n_{s,t}}(k, P_4).$ So, by Theorem \ref{dlCH}, $\zeta_{s, t, 1}$ and $\zeta_{s, t, 2}$ are representative of the non-zero elements $h_0h_{s+1}h_{s+t}^{2}$ and $h_0h_{s}^{2}h_{s+t+1}\in {\rm Ext}_A^{4, n_{s,t}+4}(k, k)$  respectively. Moreover, by Theorem \ref{dlntg},  ${\rm Ext}_A^{4, n_{s,t}+4}(k, k)$ has dimension $2$ for all $s > 0$ and $t > 3.$ In addition, by a direct calculation, we can show that $k\otimes_{GL_4(k)}P((P_4)_{n_{s, t}}^{*})  = \langle [\zeta_{s, t, 1}]\rangle$ for $s = 1, 2$ and that $k\otimes_{GL_4(k)}P((P_4)_{n_{s, t}}^{*})  = \langle [\zeta_{s, t, 1}],\, [\zeta_{s, t, 2}]\rangle$ for all $s\geq 3.$  Now, the theorem follows from these data and Proposition \ref{mdt}. 

\subsection{Proof of Theorem \ref{dlc3}}

For simplicity, we prove the theorem for $s\geq 4.$ The others cases use a similar technique. We denote by $n_s = 2(2^{s} - 1) + 2^{s}.$ Based upon a result in Sum \cite{N.S}, we have that
$$ {\rm Ker}[\overline{Sq}^{0}]_{n_s} = (Q^{\otimes 4}_{n_s})^{\omega_{(s)}} = \underline{(Q^{\otimes 4}_{n_s})}^{\omega_{(s)}}\bigoplus \overline{(Q^{\otimes 4}_{n_s})}^{\omega_{(s)}},\ \ \omega_{(s)}:= \underset{\mbox{{$s$ times of $2$}}}{\underbrace{(2, 2, \ldots, 2}},1).$$
and that an admissible monomial basis of $\underline{(Q^{\otimes 4}_{n_s})}^{\omega_{(s)}}$ is determined as follows:

For $s = 1,$

\begin{center}
\begin{tabular}{llll}
$c_{1,\,1}= x_3x_4^{3}$, & $c_{1,\,2}= x_3^{3}x_4$, & $c_{1,\,3}= x_2x_4^{3}$, & $c_{1,\,4}= x_2x_3^{3}$, \\
$c_{1,\,5}= x_2^{3}x_4$, & $c_{1,\,6}= x_2^{3}x_3$, & $c_{1,\,7}= x_1x_4^{3}$, & $c_{1,\,8}= x_1x_3^{3}$, \\
$c_{1,\,9}= x_1x_2^{3}$, & $c_{1,\,10}= x_1^{3}x_4$, & $c_{1,\,11}= x_1^{3}x_3$, & $c_{1,\,12}= x_1^{3}x_2$, \\
$c_{1,\,13}= x_2x_3x_4^{2}$, & $c_{1,\,14}= x_2x_3^{2}x_4$, & $c_{1,\,15}= x_1x_3x_4^{2}$, & $c_{1,\,16}= x_1x_3^{2}x_4$, \\
$c_{1,\,17}= x_1x_2x_4^{2}$, & $c_{1,\,18}= x_1x_2x_3^{2}$, & $c_{1,\,19}= x_1x_2^{2}x_4$, & $c_{1,\,20}= x_1x_2^{2}x_3$.
\end{tabular}%
\end{center}

For $s\geq 2,$

\begin{center}
\begin{tabular}{llr}
$c_{s,\,1}= x_2x_3^{2^{s}-2}x_4^{2^{s+1}-1}$, & $c_{s,\,2}= x_2x_3^{2^{s+1}-1}x_4^{2^{s}-2}$, & \multicolumn{1}{l}{$c_{s,\,3}= x_2^{2^{s+1}-1}x_3x_4^{2^{s}-2}$,} \\
$c_{s,\,4}= x_1x_3^{2^{s}-2}x_4^{2^{s+1}-1}$, & $c_{s,\,5}= x_1x_3^{2^{s+1}-1}x_4^{2^{s}-2}$, & \multicolumn{1}{l}{$c_{s,\,6}= x_1^{2^{s+1}-1}x_3x_4^{2^{s}-2}$,} \\
$c_{s,\,7}= x_1x_2^{2^{s}-2}x_4^{2^{s+1}-1}$, & $c_{s,\,8}= x_1x_2^{2^{s+1}-1}x_4^{2^{s}-2}$, & \multicolumn{1}{l}{$c_{s,\,9}= x_1^{2^{s+1}-1}x_2x_4^{2^{s}-2}$,} \\
$c_{s,\,10}= x_1x_2^{2^{s}-2}x_3^{2^{s+1}-1}$, & $c_{s,\,11}= x_1x_2^{2^{s+1}-1}x_3^{2^{s}-2}$, & \multicolumn{1}{l}{$c_{s,\,12}= x_1^{2^{s+1}-1}x_2x_3^{2^{s}-2}$,} \\
$c_{s,\,13}= x_3^{2^{s}-1}x_4^{2^{s+1}-1}$, & $c_{s,\,14}= x_3^{2^{s+1}-1}x_4^{2^{s}-1}$, & \multicolumn{1}{l}{$c_{s,\,15}= x_2^{2^s-1}x_4^{2^{s+1}-1}$,} \\
$c_{s,\,16}= x_2^{2^s-1}x_3^{2^{s+1}-1}$, & $c_{s,\,17}= x_2^{2^{s+1}-1}x_4^{2^{s}-1}$, & \multicolumn{1}{l}{$c_{s,\,18}= x_2^{2^{s+1}-1}x_3^{2^{s}-1}$,} \\
$c_{s,\,19}= x_1^{2^s-1}x_4^{2^{s+1}-1}$, & $c_{s,\,20}= x_1^{2^s-1}x_3^{2^{s+1}-1}$, & \multicolumn{1}{l}{$c_{s,\,21}= x_1^{2^{s+1}-1}x_4^{2^{s}-1}$,} \\
$c_{s,\,22}= x_1^{2^{s+1}-1}x_3^{2^{s}-1}$, & $c_{s,\,23}= x_1^{2^s-1}x_2^{2^{s+1}-1}$, & \multicolumn{1}{l}{$c_{s,\,24}= x_1^{2^{s+1}-1}x_2^{2^{s}-1}$,} \\
$c_{s,\,25}= x_2x_3^{2^{s}-1}x_4^{2^{s+1}-2}$, & $c_{s,\,26}= x_2x_3^{2^{s+1}-2}x_4^{2^{s}-1}$, & \multicolumn{1}{l}{$c_{s,\,27}= x_2^{2^s-1}x_3x_4^{2^{s+1}-2}$,} \\
$c_{s,\,28}= x_1x_3^{2^{s}-1}x_4^{2^{s+1}-2}$, & $c_{s,\,29}= x_1x_3^{2^{s+1}-2}x_4^{2^{s}-1}$, & \multicolumn{1}{l}{$c_{s,\,30}= x_1^{2^s-1}x_3x_4^{2^{s+1}-2}$,} \\
$c_{s,\,31}= x_1x_2^{2^{s}-1}x_4^{2^{s+1}-2}$, & $c_{s,\,32}= x_1x_2^{2^{s+1}-2}x_4^{2^{s}-1}$, & \multicolumn{1}{l}{$c_{s,\,33}= x_1^{2^s-1}x_2x_4^{2^{s+1}-2}$,} \\
$c_{s,\,34}= x_1x_2^{2^{s}-1}x_3^{2^{s+1}-2}$, & $c_{s,\,35}= x_1x_2^{2^{s+1}-2}x_3^{2^{s}-1}$, & \multicolumn{1}{l}{$c_{s,\,36}= x_1^{2^s-1}x_2x_3^{2^{s+1}-2}$,} \\
$c_{s,\,37}= x_2^{3}x_3^{2^{s+1}-3}x_4^{2^{s}-2}$, & $c_{s,\,38}= x_1^{3}x_3^{2^{s+1}-3}x_4^{2^{s}-2}$, & \multicolumn{1}{l}{$c_{s,\,39}= x_1^{3}x_2^{2^{s+1}-3}x_4^{2^{s}-2}$,} \\
$c_{s,\,40}= x_1^{3}x_2^{2^{s+1}-3}x_3^{2^{s}-2}$. &  &  
\end{tabular}%
\end{center}

For $s = 2,$

\begin{center}
\begin{tabular}{lllc}
$c_{2,\,41}= x_2^{3}x_3^{3}x_4^{4}$, & $c_{2,\,42}= x_1^{3}x_3^{3}x_4^{4}$, & $c_{2,\,43}= x_1^{3}x_2^{3}x_4^{4}$, & $c_{2,\,44}= x_1^{3}x_2^{3}x_3^{4}$.
\end{tabular}%
\end{center}

For $s\geq 3,$

\begin{center}
\begin{tabular}{lcr}
$c_{s,\,41}= x_2^{3}x_3^{2^{s}-3}x_4^{2^{s+1}-2}$, & $c_{s,\,42}= x_1^{3}x_3^{2^{s}-3}x_4^{2^{s+1}-2}$, & \multicolumn{1}{l}{$c_{s,\,43}= x_1^{3}x_2^{2^{s}-3}x_4^{2^{s+1}-2}$,} \\
$c_{s,\,44}= x_1^{3}x_2^{2^{s}-3}x_3^{2^{s+1}-2}$. &       &  
\end{tabular}%
\end{center}

\newpage
\begin{lema}\label{bdc31}
The following statements are true:

\begin{itemize}

\item[i)]
For $s = 1,$ we have a direct summand decomposition of $\Sigma_4$-modules:
$$  \underline{(Q^{\otimes 4}_{n_1})}^{\omega_{(1)}} = \Sigma_4(c_{1,\, 1})\bigoplus \Sigma_4(c_{1,\, 13}),$$
where $ \Sigma_4(c_{1,\, 1}) = \langle \{[c_{1,\, j}]:\, 1\leq j\leq 12\} \rangle$ and $ \Sigma_4(c_{1,\, 13}) = \langle \{[c_{1,\, j}]:\, 13\leq j\leq 20\} \rangle.$ Moreover, 
$$ \Sigma_4(c_{1,\, 1})^{\Sigma_4} = \big\langle [\sum_{1\leq j\leq 12}c_{1,\, j}]\big \rangle,\ \ \Sigma_4(c_{1,\, 13})^{\Sigma_4} = 0.$$

\item[ii)] For $s\geq 4,$ we have a direct summand decomposition of $\Sigma_4$-modules:
$$ \underline{(Q^{\otimes 4}_{n_s})}^{\omega_{(s)}} = \Sigma_4(c_{s,\, 1})\bigoplus \Sigma_4(c_{s,\, 13})\bigoplus \Sigma_4(c_{s,\, 25})\bigoplus \Sigma_4(c_{s,\, 37}),$$
where $$ \begin{array}{ll}
\medskip
\Sigma_4(c_{s,\, 1})& = \langle \{[c_{s,\, j}]:\, 1\leq j\leq 12\} \rangle, \ \ \Sigma_4(c_{s,\, 13}) = \langle \{[c_{s,\, j}]:\, 13\leq j\leq 24\} \rangle,\\
\Sigma_4(c_{s,\, 25})& = \langle \{[c_{s,\, j}]:\, 25\leq j\leq 36\} \rangle, \ \ \Sigma_4(c_{s,\, 37}) = \langle \{[c_{s,\, j}]:\, 37\leq j\leq 44\} \rangle.
\end{array}$$

\item[iii)] $\Sigma_4(c_{s,\, j})^{\Sigma_4} = 0$ for $j = 25, 37.$

\item[iv)]  $\Sigma_4(c_{s,\, 1})^{\Sigma_4} = \langle [q_{s,\, 1}]\rangle$ with $q_{s,\, 1} := \sum_{1\leq j\leq 12}.$

\item[v)] $\Sigma_4(c_{s,\, 13})^{\Sigma_4} = \langle [q_{s,\, 2}]\rangle$ with $q_{s,\, 2} := \sum_{13\leq j\leq 24}.$

\end{itemize}
\end{lema}

 It is rather straightforward to prove this lemma directly,  so we omit the details by leaving them to the interested reader for a direct check.  Now, according to \cite{N.S}, the monomial bases of $\overline{(Q^{\otimes 4}_{n_s})}^{\omega_{(s)}}$ are given as follows:

For $s\geq 2,$

\begin{center}
\begin{tabular}{lcr}
$c_{s,\,45}= x_1x_2^{2^{s}-2}x_3x_4^{2^{s+1}-2}$, & $c_{s,\,46}= x_1x_2^{2^{s+1}-2}x_3x_4^{2^{s}-2}$, & \multicolumn{1}{l}{$c_{s,\,47}= x_1x_2^{2}x_3^{2^{s}-3}x_4^{2^{s+1}-2}$,} \\
$c_{s,\,48}= x_1x_2^{3}x_3^{2^{s}-4}x_4^{2^{s+1}-2}$, & $c_{s,\,49}= x_1x_2^{3}x_3^{2^{s+1}-2}x_4^{2^{s}-4}$, & \multicolumn{1}{l}{$c_{s,\,50}= x_1^{3}x_2x_3^{2^{s}-4}x_4^{2^{s+1}-2}$,} \\
$c_{s,\,51}= x_1^{3}x_2x_3^{2^{s+1}-2}x_4^{2^{s}-4}$, & $c_{s,\,52}= x_1^{3}x_2^{2^{s+1}-3}x_3^{2}x_4^{2^{s}-4}$, & \multicolumn{1}{l}{$c_{s,\,53}= x_1x_2x_3^{2^{s}-2}x_4^{2^{s+1}-2}$,} \\
$c_{s,\,54}= x_1x_2x_3^{2^{s+1}-2}x_4^{2^{s}-2}$, & $c_{s,\,55}= x_1x_2^{2}x_3^{2^{s+1}-3}x_4^{2^{s}-2}$. &  
\end{tabular}%
\end{center}

For $s = 2,$ $c_{2,\, 56} =  x_1^{3}x_2^{4}x_3x_4^{2}.$

For $s\geq 3,$

\begin{center}
\begin{tabular}{lcr}
$c_{s,\,56}= x_1^{3}x_2^{5}x_3^{2^{s+1}-6}x_4^{2^{s}-4}$, & $c_{s,\,57}= x_1^{3}x_2^{5}x_3^{2^{s}-6}x_4^{2^{s+1}-4}$, & \multicolumn{1}{l}{$c_{s,\,58}= x_1x_2^{3}x_3^{2^{s+1}-4}x_4^{2^{s}-2}$,} \\
$c_{s,\,59}= x_1^{3}x_2x_3^{2^{s+1}-4}x_4^{2^{s}-2}$, & $c_{s,\,60}= x_1x_2^{3}x_3^{2^{s}-2}x_4^{2^{s+1}-4}$, & \multicolumn{1}{l}{$c_{s,\,61}= x_1^{3}x_2x_3^{2^{s}-2}x_4^{2^{s+1}-4}$,} \\
$c_{s,\,62}= x_1^{3}x_2^{2^{s}-3}x_3^{2}x_4^{2^{s+1}-4}$, & $c_{s,\,63}= x_1x_2^{2}x_3^{2^{s}-4}x_4^{2^{s+1}-1}$, & \multicolumn{1}{l}{$c_{s,\,64}= x_1x_2^{2}x_3^{2^{s+1}-1}x_4^{2^{s}-4}$,} \\
$c_{s,\,65}= x_1x_2^{2^{s+1}-1}x_3^{2}x_4^{2^{s}-4}$, & $c_{s,\,66}= x_1^{2^{s+1}-1}x_2x_3^{2}x_4^{2^{s}-4}$, & \multicolumn{1}{l}{$c_{s,\,67}= x_1x_2^{2}x_3^{2^{s}-1}x_4^{2^{s+1}-4}$,} \\
$c_{s,\,68}= x_1x_2^{2}x_3^{2^{s+1}-4}x_4^{2^{s}-1}$, & $c_{s,\,69}= x_1x_2^{2^{s}-1}x_3^{2}x_4^{2^{s+1}-4}$. &  
\end{tabular}%
\end{center}

For $s = 3,$ $c_{3,\, 70} = x_1^{3}x_2^{5}x_3^{6}x_4^{8}.$

For $s\geq 4,$ $c_{s,\,70}= x_1^{2^{s}-1}x_2x_3^{2}x_4^{2^{s+1}-4}.$\\[2mm] 
Then, by direct computations, that

\begin{lema}\label{bdc32}
The space of $\Sigma_4$-invariants $(\overline{(Q^{\otimes 4}_{n_s})}^{\omega_{(s)}})^{\Sigma_4}$ is generated by the classes $[q_{s,\, 3}]$  and $[q_{s,\, 4}],$ where  
$$q_{s,\, 3} = c_{s,\, 47} + c_{s,\, 48} + \sum_{62\leq j\leq 67}c_{s,\, j}, \ \ q_{s,\, 4} = \sum_{58\leq j\leq 61}c_{s,\, j}.$$
\end{lema}

\begin{proof}
It should be noticed that $\dim \overline{(Q^{\otimes 4}_{n_s})}^{\omega_{(s)}} = 26$ with the basis $\{[c_{s,\, j}]:\, 45\leq j\leq 70\}.$ Suppose  $[g]\in (\overline{(Q^{\otimes 4}_{n_s})}^{\omega_{(s)}})^{\Sigma_4},$ then we have $g\equiv \big(\sum_{45\leq j\leq 70}\beta_jc_{s,\, j}\big)$ with $\beta_j\in k.$ Using Cartan's formula and the relations $(\sigma_i(g) + g)\equiv 0,$ for $1\leq i\leq 3,$ we get
$$ \begin{array}{ll}
\medskip
\sigma_1(g) + g &\equiv \big((\beta_{47} + \beta_{48} + \beta_{62} + \beta_{67} + \beta_{68} + \beta_{70})c_{s,\, 45} + (\beta_{48} + \beta_{49} + \beta_{67} + \beta_{69})c_{s,\, 46}\\
\medskip
&\quad + (\beta_{52} + \beta_{53})(c_{s,\, 52} + c_{s,\, 53}) + (\beta_{54} + \beta_{55})(c_{s,\, 54} + c_{s,\, 55})\\ 
\medskip
&\quad + (\beta_{48} + \beta_{56} + \beta_{57} + \beta_{67})(c_{s,\, 56} + c_{s,\, 57}) + (\beta_{60} + \beta_{61})(c_{s,\, 60} + c_{s,\, 61})\\
\medskip
&\quad + (\beta_{63} + \beta_{65})(c_{s,\, 63}+c_{s,\, 65}) + (\beta_{64} + \beta_{66})(c_{s,\, 64} +  c_{s,\, 66}) + (\beta_{48}\\
\medskip
 &\quad + \beta_{67})(c_{s,\, 68}+c_{s,\, 70})\big)\equiv 0,\\
\medskip
\sigma_2(g) + g &\equiv \big((\beta_{45} + \beta_{47} + \beta_{49} + \beta_{54} + \beta_{55} + \beta_{56} +\beta_{65})c_{s,\, 45} + (\beta_{46} + \beta_{48} + \beta_{55} + \beta_{64}  +\beta_{69})c_{s,\, 46}\\
\medskip
&\quad + (\beta_{45} + \beta_{47} + \beta_{56} + \beta_{65})c_{s,\, 47} +  (\beta_{46} + \beta_{48} + \beta_{55} + \beta_{64})c_{s,\, 48} + (\beta_{49} + \beta_{54})c_{s,\, 49}\\
\medskip
&\quad + (\beta_{50} + \beta_{52})(c_{s,\, 50} + c_{s,\, 52}) + (\beta_{49} + \beta_{54} + \beta_{69})c_{s,\, 47} + \beta_{69}c_{s,\, 55}\\
\medskip
&\quad +  (\beta_{49} + \beta_{54} + \beta_{55} + \beta_{69})c_{s,\, 56} +  (\beta_{57} + \beta_{68})c_{s,\, 57} + (\beta_{59}+ \beta_{60})(c_{s,\, 59} + c_{s,\, 60})\\
\medskip
&\quad +(\beta_{57} + \beta_{68})(c_{s,\, 57} + c_{s,\, 68}) + (\beta_{51} + \beta_{54} + \beta_{62} + \beta_{63})c_{s,\, 62} \\
\medskip
&\quad +  (\beta_{49} + \beta_{51} + \beta_{62} + \beta_{63})c_{s,\, 63}+ (\beta_{66} + \beta_{67})(c_{s,\, 66} + c_{s,\, 67}) + \beta_{55}c_{s,\, 70}\big)\equiv 0,\\
\medskip
\sigma_3(g) + g &\equiv \big((\beta_{45} + \beta_{46} + \beta_{49} + \beta_{52} + \beta_{53} + \beta_{68} +\beta_{69})c_{s,\, 45}\\
\medskip
 &\quad + (\beta_{45} + \beta_{46} + \beta_{47} + \beta_{52}  +\beta_{53} + \beta_{62} + \beta_{70})c_{s,\, 46}\\
  \medskip
&\quad + (\beta_{47} + \beta_{49} + \beta_{62})c_{s,\, 49} + (\beta_{50} + \beta_{51})(c_{s,\, 50} + c_{s,\, 51})\\
\medskip
&\quad +(\beta_{48} + \beta_{52} + \beta_{54} + \beta_{56} + \beta_{67} + \beta_{68})c_{s,\, 54}+ (\beta_{53} + \beta_{55} + \beta_{57} + \beta_{68})c_{s,\, 55}\\
\medskip
&\quad +  (\beta_{48} + \beta_{52} + \beta_{54} + \beta_{56} + \beta_{67})c_{s,\, 56} +   (\beta_{53} + \beta_{55} + \beta_{57})c_{s,\, 57} \\
\medskip
&\quad+ (\beta_{58} + \beta_{59})(c_{s,\, 58} + c_{s,\, 59}) +  (\beta_{47} + \beta_{49} + \beta_{62})c_{s,\, 62}\\
\medskip
&\quad + (\beta_{52} + \beta_{63} + \beta_{64})(c_{s,\, 63} + c_{s,\, 64})  + (\beta_{53} + \beta_{65} + \beta_{66})(c_{s,\, 65} + c_{s,\, 66})\\
\medskip
&\quad +  (\beta_{68} + \beta_{69} + \beta_{70})c_{s,\, 69} +  (\beta_{69} + \beta_{70})c_{s,\, 70}\big)\equiv 0.
 \end{array}$$
The above equalities show that 
$$ \beta_{j} = \left\{\begin{array}{ll}
\beta_{47}&\mbox{for $j \in \{48, 62, \ldots, 67\}$},\\[2mm]
\beta_{58}&\mbox{for $j \in \{59, 60, 61\}$},\\[2mm]
0&\mbox{otherwise}.
\end{array}\right.$$
The lemma follows.
\end{proof}

Now, assume that $[f]$ belongs to the invariant space $({\rm Ker}[\overline{Sq}^{0}]_{n_s})^{GL_4(k)}.$ Since $\Sigma_4\subset GL_4(k),$ following Lemmas \ref{bdc31} and \ref{bdc32}, we have
$$ f\equiv \big(\gamma_1q_{s,\, 1} + \gamma_2q_{s,\, 2} + \gamma_3q_{s,\, 3} + \gamma_4q_{s,\, 4}\big)$$
where $\gamma_i\in k$ for all $i.$ Direct computing based on the relation $(\sigma_4(f) + f)\equiv 0,$ we obtain
$$ 
(\sigma_4(f) + f)\equiv \big((\gamma_1 + \gamma_4)c_{s,\, 1} + (\gamma_1 + \gamma_3)c_{s,\, 3} + \gamma_1c_{s,\, 8}+ (\gamma_1 + \gamma_2)c_{s,\, 15} + \mbox{others terms}\big)\equiv 0.
$$
This equality implies $\gamma_i = 0$ for $1\leq i\leq 4.$ The proof of the theorem is finished.

\subsection{Proof of Theorem \ref{dlct2}}

Let $n_{s,\, t}: = 2(2^{s} - 1) + 2^{s}(2^{t}-1).$ We have seen that the Kameko map
$$[\overline{Sq}^{0}]_{n_{s,\, t}}:= \overline{Sq}^{0}: Q^{\otimes 4}_{n_{s,\, t}} \to Q^{\otimes 4}_{2^{s+t - 1} + 2^{s-1} - 3}$$
is an epimorphism of $kGL_4(k)$-modules and so 
\begin{equation}\label{bdtt2}
 Q^{\otimes 4}_{n_{s,\, t}}  \cong {\rm Ker}[\overline{Sq}^{0}]_{n_{s,\, t}} \bigoplus Q^{\otimes 4}_{2^{s+t - 1} + 2^{s-1} - 3}.
\end{equation}
Using the calculations in Sum \cite{N.S2} and Theorem \ref{dlct}, it follows that, for each $t\geq 4,$ the coinvariants $k\otimes_{GL_4(k)}P((P_4)_{2^{s+t - 1} + 2^{s-1} - 3}^{*})$ are 1-dimensional if $1\leq s\leq 4$ and are 2-dimensional if $s\geq 5.$ Thus, due to \eqref{bdtt2}, we need compute $({\rm Ker}[\overline{Sq}^{0}]_{n_{s,\, t}})^{GL_4}.$ We are going to prove the following lemma, which is crucial in the proof of the theorem.

\begin{lema}\label{bdct2}
Let $s$ and $t$ be two positive integers such that $t\geq 5.$ Then, the invariant spaces $({\rm Ker}[\overline{Sq}^{0}]_{n_{s,\, t}})^{GL_4}$ is trivial if $s = 1, 2$ and has dimension $1$ if $s\geq 3.$
\end{lema}

\begin{proof}[{\it Outline of the proof}]
As well known, from the work of Sum \cite{N.S}, if $x$ is an admissible monomial in $(P_4)_{n_{s,\, t}}$ such that $[x]\in {\rm Ker}[\overline{Sq}^{0}]_{n_{s,\, t}},$ then $\omega_{(s,\, t)}:=\omega(x)= {\underset{\mbox{{$s$ times }}}{\underbrace{(2, 2, \ldots, 2}}, {\underset{\mbox{{$t$ times }}}{\underbrace{1, 1, \ldots, 1})}}}.$ The admissible monomial bases of $(Q_{n_{s,\, t}}^{\otimes 4})^{\omega^{2}_{(s,\, t)}}$ are the sets $[\{b_{t,\,1,\, j}|\, 1\leq j\leq 138\}\cup \{c_{t,\,1,\, j}|\, 1\leq j\leq 7\}]$ for $s  =1,$ and $[\{b_{t,\,s,\, j}|\, 1\leq j\leq 105\}]$ for $s\geq 2,$ where the monomials $b_{t,\, 1,\, j},$ $c_{t,\, 1,\, j}$ and $b_{t,\, s,\, j}$ are given in \cite{N.S}. Using these data and our previous results in \cite{D.P11}, by similar computations as in the proof of Theorem \ref{dlc3}, we find that the invariant spaces $((Q^{\otimes 4}_{n_{1,\, t}})^{\omega_{(1,\, t)}})^{GL_4}$ and $((Q^{\otimes 4}_{n_{2,\, t}})^{\omega_{(2,\, t)}})^{GL_4}$  are trivial and that $((Q^{\otimes 4}_{n_{s,\, t}})^{\omega_{(s,\, t)}})^{GL_4} = \langle [\sum_{1\leq j\leq 105}b_{t,\,s,\, j}] \rangle$ for $s\geq 3.$ These imply that 
$$ ({\rm Ker}[\overline{Sq}^{0}]_{n_{s,\, t}})^{GL_4} = \left\{\begin{array}{ll}
0 &\mbox{if $s = 1,2$},\\[1mm]
\langle [\sum_{1\leq j\leq 105}b_{t,\,s,\, j}] \rangle& \mbox{if $s \geq 3$}.
\end{array}\right.$$
The lemma follows.
\end{proof}

Now, for $s = 1, 2,$ let $\rho\in (P_4)_{n_{1,\, t}}$ and $\overline{\rho}\in (P_4)_{n_{2,\, t}}$ such that $[\rho]\in (Q^{\otimes 4}_{n_{1,\, t}})^{GL_4}$ and $[\overline{\rho}]\in (Q^{\otimes 4}_{n_{2,\, t}})^{GL_4},$ respectively. Since Kameko's homomorphism $[\overline{Sq}^{0}]_{n_{s,\, t}}: Q^{\otimes 4}_{n_{s,\, t}}  \longrightarrow Q^{\otimes 4}_{\frac{n_{s,\, t}}{2}}$ is an epimorphism of $kGL_4(k)$-modules, $[\overline{Sq}^{0}]_{n_{1,\, t}}([\rho])\in (Q^{\otimes 4}_{\frac{n_{1,\, t}}{2}})^{GL_4}$ and $[\overline{Sq}^{0}]_{n_{2,\, t}}([\overline{\rho}])\in (Q^{\otimes 4}_{\frac{n_{2,\, t}}{2}})^{GL_4}$. Following \cite{N.S2}, we have $(Q^{\otimes 4}_{\frac{n_{1,\, t}}{2}})^{GL_4} \subseteq \langle [p_{4,\, t}:=\sum_{1\leq j\leq 35}d_{t,\, j}] \rangle,$ and $(Q^{\otimes 4}_{\frac{n_{2,\, t}}{2}})^{GL_4} \subseteq \langle [\overline{p}_{4,\, t}] \rangle,$ where $\overline{p}_{4,\, t} = \sum_{1\leq \ell\leq 3}\sum_{1\leq i_1\leq \ldots\leq i_{\ell}\leq 4}x_{i_1}x_{i_2}^{2}\ldots x^{2^{\ell-2}}_{i_{\ell-1}}x^{2^{t+1}-2^{\ell-1}}_{i_{\ell}} + x_1x_2^{2}x_3^{4}x_4^{2^{t+1}-8},$ and the elements $d_{t,\, j}$ are listed as follows: 

\begin{center}
\begin{tabular}{llr}
$d_{t,\, 1}=x_3^{2^{t}-1}x_4^{2^{t}-1}$, & $d_{t,\, 2}=x_2^{2^{t}-1}x_4^{2^{t}-1}$, & \multicolumn{1}{l}{$d_{t,\, 3}=x_2^{2^{t}-1}x_3^{2^{t}-1}$,} \\
$d_{t,\, 4}=x_1^{2^{t}-1}x_4^{2^{t}-1}$, & $d_{t,\, 5}=x_1^{2^{t}-1}x_3^{2^{t}-1}$, & \multicolumn{1}{l}{$d_{t,\, 6}=x_1^{2^{t}-1}x_2^{2^{t}-1}$,} \\
$d_{t,\, 7}=x_2x_3^{2^{t}-2}x_4^{2^{t}-1}$, & $d_{t,\, 8}=x_2x_3^{2^{t}-1}x_4^{2^{t}-2}$, & \multicolumn{1}{l}{$d_{t,\, 9}=x_2^{2^{t}-1}x_3x_4^{2^{t}-2}$,} \\
$d_{t,\, 10}=x_1x_3^{2^{t}-2}x_4^{2^{t}-1}$, & $d_{t,\, 11}=x_1x_3^{2^{t}-1}x_4^{2^{t}-2}$, & \multicolumn{1}{l}{$d_{t,\, 12}=x_1x_2^{2^{t}-2}x_4^{2^{t}-1}$,} \\
$d_{t,\, 13}=x_1x_2^{2^{t}-2}x_3^{2^{t}-1}$, & $d_{t,\, 14}=x_1x_2^{2^{t}-1}x_4^{2^{t}-2}$, & \multicolumn{1}{l}{$d_{t,\, 15}=x_1x_2^{2^{t}-1}x_3^{2^{t}-2}$,} \\
$d_{t,\, 16}=x_1^{2^{t}-1}x_3x_4^{2^{t}-2}$, & $d_{t,\, 17}=x_1^{2^{t}-1}x_2x_4^{2^{t}-2}$, & \multicolumn{1}{l}{$d_{t,\, 18}=x_1^{2^{t}-1}x_2x_3^{2^{t}-2}$,} \\
$d_{t,\, 19}=x_2^{3}x_3^{2^{t}-3}x_4^{2^{t}-2}$, & $d_{t,\, 20}=x_1^{3}x_3^{2^{t}-3}x_4^{2^{t}-2}$, & \multicolumn{1}{l}{$d_{t,\, 21}=x_1^{3}x_2^{2^{t}-3}x_4^{2^{t}-2}$,} \\
$d_{t,\, 22}=x_1^{3}x_2^{2^{t}-3}x_3^{2^{t}-2}$, & $d_{t,\, 23}=x_1x_2^{2}x_3^{2^{t}-4}x_4^{2^{t}-1}$, & \multicolumn{1}{l}{$d_{t,\, 24}=x_1x_2^{2}x_3^{2^{t}-1}x_4^{2^{t}-4}$,} \\
$d_{t,\, 25}=x_1x_2^{2^{t}-1}x_3^{2}x_4^{2^{t}-4}$, & $d_{t,\, 26}=x_1^{2^{t}-1}x_2x_3^{2}x_4^{2^{t}-4}$, & \multicolumn{1}{l}{$d_{t,\, 27}=x_1x_2x_3^{2^{t}-2}x_4^{2^{t}-2}$,} \\
$d_{t,\, 28}=x_1x_2^{2^{t}-2}x_3x_4^{2^{t}-2}$, & $d_{t,\, 29}=x_1^{3}x_2^{5}x_3^{2^{t}-6}x_4^{2^{t}-4}$, & \multicolumn{1}{l}{$d_{t,\, 30}=x_1x_2^{2}x_3^{2^{t}-3}x_4^{2^{t}-2}$,} \\
$d_{t,\, 31}=x_1x_2^{3}x_3^{2^{t}-4}x_4^{2^{t}-2}$, & $d_{t,\, 32}=x_1x_2^{3}x_3^{2^{t}-2}x_4^{2^{t}-4}$, & \multicolumn{1}{l}{$d_{t,\, 33}=x_1^{3}x_2x_3^{2^{t}-4}x_4^{2^{t}-2}$,} \\
$d_{t,\, 34}=x_1^{3}x_2x_3^{2^{t}-2}x_4^{2^{t}-4}$, & $d_{t,\, 35}=x_1^{3}x_2^{2^{t}-3}x_3^{2}x_4^{2^{t}-4}$. &  
\end{tabular}%

\end{center}

Hence, we have  $[\overline{Sq}^{0}]_{n_{1,\, t}}([\rho]) = \gamma[\varphi(p_{4,\, t})]$ and $[\overline{Sq}^{0}]_{n_{2,\, t}}([\overline{\rho}]) = \beta[\varphi(\overline{p}_{4,\, t})]$ where $\gamma,\, \beta\in k$ and $\varphi$ is the linear transformation $\varphi: (P_4)_{\frac{n_{s,\, t}}{2}}\to (P_4)_{n_{s,\, t}},$ which is determined by $\varphi(u) = x_1x_2x_3x_4u^{2}$ for any $u\in (P_4)_{\frac{n_{s,\, t}}{2}}.$ Because $[\rho]\in [Q^{\otimes 4}_{n_{1,\, t}}]^{GL_4}$ and $[\overline{\rho}]\in [Q^{\otimes 4}_{n_{2,\, t}}]^{GL_4},$ we have $\rho \equiv \gamma\varphi(p_{4,\, t}) + \rho^{*}$ and $\overline{\rho} \equiv \beta\varphi(\overline{p}_{4,\, t}) + \overline{\rho}^{*}$ where $\rho^{*}\in (P_4)_{n_{1,\, t}}$ and $\overline{\rho}^{*}\in (P_4)_{n_{2,\, t}}$ such that $[\rho^{*}]\in {\rm Ker}[\overline{Sq}^{0}]_{n_{1,\, t}}$ and $[\overline{\rho}^{*}]\in {\rm Ker}[\overline{Sq}^{0}]_{n_{2,\, t}},$ respectively. Then, by direct calculations using Lemma \ref{bdct2} and the relations $\sigma_i(\rho) + \rho \equiv 0,$ and $\sigma_i(\overline{\rho}) + \overline{\rho} \equiv 0$ for $1\leq i\leq 4$, we get $(Q^{\otimes 4}_{n_{1,\, t}})^{GL_4} = \langle [\varphi(p_{4,\, t})]\rangle$ and $(Q^{\otimes 4}_{n_{2,\, t}})^{GL_4} = \langle [\varphi(\overline{p}_{4,\, t})]\rangle.$ Now, consider the elements $\zeta_{1,\, t}:=a_1^{(1)}a_2^{(1)}a_3^{(2^{t}-1)}a_4^{(2^{t}-1)}\in (P_4)_{n_{1,\, t}}^{*}$ and  $\zeta_{2,\, t}:=a_1^{(1)}a_2^{(1)}a_3^{(1)}a_4^{(2^{t+2}-1)}\in (P_4)_{n_{2,\, t}}^{*}.$ Straightforward calculations show that they are $\widehat{A}$-annihilated. Furthermore, $\langle [\zeta_{1,\, t}], [\varphi(p_{4,\, t})] \rangle = 1$ and $\langle [\zeta_{2,\, t}], [\varphi(\overline{p}_{4,\, t})] \rangle = 1.$ So, it may be concluded that $k\otimes_{GL_4(k)}P((P_4)_{n_{1,\, t}}^{*}) = \langle [\zeta_{1,\, t}] \rangle$ and $k\otimes_{GL_4(k)}P((P_4)_{n_{2,\, t}}^{*}) = \langle [\zeta_{2,\, t}] \rangle.$ 

By similar arguments using Theorem \ref{dlct} and Lemma \ref{bdct2}, it is not too difficult to verify that
$$k\otimes_{GL_4(k)}P((P_4)_{n_{s,\, t}}^{*}) \\
 = \left\{\begin{array}{ll}
\langle [\zeta_{s,\, t}], [\widetilde{\zeta}_{s,\, t}] \rangle &\mbox{if $s = 3,\, 4$},\\[1mm]
\langle [\zeta_{s,\, t}], [\widetilde{\zeta}_{s,\, t}],  [\widehat{\zeta}_{s,\, t}] \rangle &\mbox{if $s \geq 5$},
\end{array}\right.$$
where the elements $$ \begin{array}{lll}
\medskip
\zeta_{s,\, t}&:=a_1^{(1)}a_2^{(2^{s}-1)}a_3^{(2^{s+t-1}-1)}a_4^{(2^{s+t-1}-1)}, \\
\medskip
 \widetilde{\zeta}_{s,\, t}&:= a_1^{(0)}a_2^{(0)}a_3^{(2^{s}-1)}a_4^{(2^{s+t}-1)},\\
\medskip
\widehat{\zeta}_{s,\, t}&:= a_1^{(1)}a_2^{(2^{s-1}-1)}a_3^{(2^{s-1}-1)}a_4^{(2^{s+t}-1)}
\end{array}$$
belong to $P((P_4)_{n_{s,\, t}}^{*}).$ And therefore it leads to the desired conclusion. The proof of Theorem \ref{dlct2} is completed.

\end{document}